\documentclass[11pt]{article}

\usepackage{amsmath, amsthm, enumitem, amssymb, graphicx, tikz-cd, tikz}
\usetikzlibrary{matrix,arrows,decorations.pathmorphing}
\usetikzlibrary{arrows,decorations.pathmorphing,backgrounds,fit,positioning,shapes.symbols,chains,calc}

\DeclareMathOperator{\Sp}{Sp}

\DeclareMathOperator{\GL}{GL}
\DeclareMathOperator{\SL}{SL}

\newcommand\Field{\ensuremath{\mathbb{F}}}

\DeclareMathOperator{\hh}{H}
\DeclareMathOperator{\rh}{\widetilde{\hh}}

\DeclareMathOperator{\sgn}{sgn }
\DeclareMathOperator{\height}{ht }
\DeclareMathOperator{\lk}{Link }
\DeclareMathOperator{\starry}{star }
\DeclareMathOperator{\simp}{\mathcal{S}}
\DeclareMathOperator{\posety}{\mathcal{P}}

\DeclareMathOperator{\pf}{Pf }

\newcommand\Span[1]{\ensuremath{\langle #1 \rangle}}

\DeclareMathOperator{\Dim}{dim}
\DeclareMathOperator{\rank}{rank }

\newcommand\cT{\ensuremath{\mathcal{T}}}

\newcommand\cTD{\ensuremath{\mathcal{T}^{\pm}}}
\newcommand\cbT{\ensuremath{\mathbb{T}}}

\newcommand\cbTD{\ensuremath{\mathbb{T}^{\pm}}}
\newcommand\sT{\ensuremath{\mathcal{T}^\omega}}

\newcommand\sTD{\ensuremath{\mathcal{T}^{\omega, \pm}}}
\newcommand\sbT{\ensuremath{\mathbb{T}^\omega}}

\newcommand\sbTD{\ensuremath{\mathbb{T}^{\omega, \pm}}}
\newcommand\BP{\ensuremath{\mathcal{B}}}
\newcommand\IP{\ensuremath{\mathcal{I}}}
\DeclareMathOperator{\Gr}{Gr}
\DeclareMathOperator{\St}{St}

\usepackage[vmargin=1in, hmargin=1.25in]{geometry}

\setlist[itemize]{noitemsep}

\usepackage[skip=.5\baselineskip plus 2pt, indent=0pt]{parskip}

\usepackage[bookmarks, bookmarksdepth=2, colorlinks=true, linkcolor=blue, citecolor=blue, urlcolor=blue]{hyperref}

\makeatletter
\newcommand{\setword}[2]{%
  \phantomsection
  #1\def\@currentlabel{\unexpanded{#1}}\label{#2}%
}

\numberwithin{equation}{section}

\newtheorem{theorem}{Theorem}[section]

\newtheorem{maintheoremA}{Theorem}
\newtheorem{maintheoremB}{Theorem}
\newtheorem{maincorollaryC}{Corollary}

\newtheorem{proposition}[theorem]{Proposition}
\newtheorem{lemma}[theorem]{Lemma}
\newtheorem{corollary}[theorem]{Corollary}

\newtheorem{conjecture}[theorem]{Conjecture}

\theoremstyle{definition}
\newtheorem{definition}[theorem]{Definition}
\newtheorem{notation}[theorem]{Notation}

\theoremstyle{remark}
\newtheorem{remark}[theorem]{Remark}
\newtheorem{example}[theorem]{Example}

\usepackage[inkscapelatex=false]{svg}
\usepackage[backend=biber,style=alphabetic,sorting=nty]{biblatex}
\title{\vspace{-40pt}On the top-degree cohomology groups of congruence subgroups of $\Sp_{2n}(\mathbb{Z})$\vspace{-15pt}}
\author{Fabio Capovilla-Searle}
\date{}

\begin{document}

\vspace{-10pt}
\maketitle

\vspace{-18pt}
\begin{abstract}
\noindent
Let $\Gamma_{2n}^\omega(p)$ be the level-$p$ principal congruence subgroup of $\Sp_{2n}(\mathbb{Z})$ for all prime $p$. Borel--Serre demonstrated that the cohomology of $\Gamma_{2n}^\omega(p)$ vanishes above degree $n^2$. We prove that $\hh^{n^2}(\Gamma_{2n}^\omega(p); \mathbb{Q})$ surjects onto the homology of the quotient of the symplectic Tits building for $\mathbb{Q}$ by $\Gamma_{2n}^\omega(p)$ and we compute the homology of this quotient. We conclude that $\hh^{n^2}(\Gamma_{2n}^\omega(p);\mathbb{Q})$ is nontrivial and provide a lower bound of its rank. 
\end{abstract}

\setlength{\parskip}{0pt}
\tableofcontents 
\setlength{\parskip}{\baselineskip}


\section{Introduction}

Let $R$ be a commutative ring. Equip $R^{2n}$ with the \emph{standard symplectic form} $\omega$. This is a skew-symmetric, non-degenerate bilinear form, which, on the \emph{standard symplectic basis} $\{\vec{e}_1, \vec{f}_1, \dots, \vec{e}_n, \vec{f}_n \}$ for $R^{2n}$, satisfies
\begin{equation}
\omega(\vec{e}_i, \vec{e}_j) = \omega(\vec{f}_i, \vec{f}_j) = 0 \text{ and } \omega(\vec{e}_i, \vec{f}_j) = \delta_{ij},
\end{equation}
where $\delta_{ij}$ is the Kronecker delta and $i, j \in \{1, \dots, n\}$. The \emph{symplectic group}, $\Sp_{2n}(R)$, is the group of invertible $2n \times 2n$ matrices with entries in $R$ that preserve the symplectic form. Let
\begin{equation}
\Gamma_{2n}^\omega(p):= \ker(\Sp_{2n}(\mathbb{Z}) \xrightarrow{\pi} \Sp_{2n}(\mathbb{Z}/p\mathbb{Z}))
\end{equation}
denote the \emph{level-$p$ principal congruence subgroup of $\Sp_{2n}(\mathbb{Z})$}, where $p$ is a positive integer and $\pi$ is the reduction mod-$p$ map. $\Gamma_{2n}^\omega(p)$ is a finite index subgroup of $\Sp_{2n}(\mathbb{Z})$. In this article, we study the top-degree cohomology of $\Gamma_{2n}^\omega(p)$ when $p$ is an odd prime. \\
For $n \geq 0$, Borel--Serre \cite[Theorem 11.4.2]{BorelSerreCorners} demonstrated that every finite index subgroup $\Gamma \leq \Sp_{2n}(\mathbb{Z})$ is a rational duality group of dimension $n^2$ in the sense of Bieri--Eckmann \cite{BieriEckmann73}, and they demonstrated that the \emph{symplectic Steinberg module of $\mathbb{Q}$}, denote $\St_{2n}^\omega(\mathbb{Q})$, is its dualizing module. This implies that 
\begin{equation}\label{BorelSerreDualityEq}
\hh^{n^2 - i}(\Gamma; \mathbb{Q}) \cong \hh_i(\Gamma;\St_{2n}^\omega(\mathbb{Q}) \otimes \mathbb{Q})
\end{equation}
for all $i$. The symplectic Steinberg module is a $\Sp_{2n}(\mathbb{Q})$-representation. We can describe $\St_{2n}^\omega(\mathbb{Q})$ as the top reduced homology of the symplectic Tits building of $\mathbb{Q}$ which we denote by $\sT_{2n}(\mathbb{Q})$ (cf. Definitions \ref{SymplecticTitsBuilding} and \ref{SteinbergModules}). $\sT_{2n}(\mathbb{Q})$ is an $(n - 1)$-dimensional simplicial complex that Solomon--Tits \cite[Theorem 1]{Solomon} proved was homotopy equivalent to a wedge of $(n - 1)$-spheres. Thus $\sT_{2n}(\mathbb{Q})$ has only one interesting homology group and this is defined to be the symplectic Steinberg module. That is, 
\begin{equation*}
\St_{2n}^\omega(\mathbb{Q}) := \rh_{n - 1}(\sT_{2n}(\mathbb{Q});\mathbb{Z}).
\end{equation*}
Borel--Serre's result \cite[Theorem 11.4.2]{BorelSerreCorners}, often called \emph{Borel--Serre duality}, implies that the cohomology of finite index subgroups of $\Sp_{2n}(\mathbb{Z})$ vanishes above degree $n^2$. 

\begin{conjecture} \textup{\cite{BruckPatztSroka2023Arxiv}} \label{conjecture}
$\hh^{n^2 - i}(\Sp_{2n}(\mathbb{Z});\mathbb{Q}) = 0$ for $n \geq i + 1$.
\end{conjecture}

Br\"uck--Patzt--Sroka's Conjecture \ref{conjecture} is the symplectic analogue of a conjecture by Church--Farb--Putman \cite{MR3290086} about the top-degree cohomology $\SL_n(\mathbb{Z})$. Br\"uck--Patzt--Sroka \cite[Chapter 5, Theorem 44]{68a8695d93b544d594c5cde50fd83301} proved that Br\"uck--Patzt--Sroka's Conjecture \ref{conjecture} is true for $i = 0$. Br\"uck--Santos Rego--Sroka's \cite{MR4806366} reproved and generalized this result. Br\"uck--Patzt--Sroka's \cite[Chapter 5, Theorem 44]{68a8695d93b544d594c5cde50fd83301} argument builds off of Gunnells' \cite[Theorem 4.11]{MR1749441} work. Br\"uck--Sroka \cite[Proposition 5.1]{BruckSroka24} provide an alternative geometric proof to Gunnells' \cite[Theorem 4.11]{MR1749441} result. Br\"uck--Patzt--Sroka \cite[Theorem A]{BruckPatztSroka2023Arxiv} demonstrated that Br\"uck--Patzt--Sroka's Conjecture \ref{conjecture} is also true for $i = 1$. \\
We show that $\hh^{n^2}(\Gamma_{2n}^\omega(p);\mathbb{Q})$ is nontrivial for all prime $p$.

\begin{maintheoremA}
\label{theorem:main:surjection}
For $n \geq 0$ and $p$ a prime integer, the map
\begin{equation*}
\hh^{n^2}(\Gamma_{2n}^\omega(p);\mathbb{Q}) \rightarrow \rh_{n - 1}(\vert \sT_{2n}(\mathbb{Q})\vert /\Gamma_{2n}^\omega(p); \mathbb{Q})
\end{equation*}
is a surjection.
\end{maintheoremA}

Before discussing the codomain of the map in Theorem \ref{theorem:main:surjection}, we will now describe the map in Theorem \ref{theorem:main:surjection}. Let $\Gamma = \Gamma_{2n}^\omega(p)$ and $i = 0$ in Equation \ref{BorelSerreDualityEq}:
\begin{equation*}
\hh^{n^2}(\Gamma_{2n}^\omega(p);\mathbb{Q}) \cong \hh_0(\Gamma_{2n}^\omega(p);\St_{2n}^\omega(\mathbb{Q})) \cong (\St_{2n}^\omega(\mathbb{Q}))_{\Gamma_{2n}^\omega(p)}. 
\end{equation*}
The subscript indicates we are taking coinvariants with respect to $\Gamma_{2n}^\omega(p)$. Consider the quotient map 
\begin{equation}\label{quotientmap}
f\colon \vert \sT_{2n}(\mathbb{Q})\vert \rightarrow \lvert \sT_{2n}(\mathbb{Q})\rvert/\Gamma_{2n}^\omega(p).
\end{equation}
$f$ is $\Gamma_{2n}^\omega(p)$-invariant and it induces the following map on homology. 
\begin{equation*}
f_*\colon\St_{2n}^\omega(\mathbb{Q}) = \rh_{n - 1}(\lvert \sT_{2n}(\mathbb{Q})\rvert; \mathbb{Q}) \rightarrow \rh_{n - 1}(\lvert \sT_{2n}(\mathbb{Q})\rvert/\Gamma_{2n}^\omega(p); \mathbb{Q}). 
\end{equation*}
The coinvariants factor through this map. Hence, 
\begin{equation*}
f_* \colon (\St_{2n}^\omega(\mathbb{Q}))_{\Gamma_{2n}^\omega(p)} \rightarrow \rh_{n - 1}(\vert \sT_{2n}(\mathbb{Q})\vert /\Gamma_{2n}^\omega(p); \mathbb{Q}).
\end{equation*}
Theorem \ref{theorem:main:surjection} is similar to work of Church--Farb--Putman \cite[Theorem D]{MR4011804}, Miller--Patzt--Putman \cite[Theorem A]{MiPaP21}, and Br\"uck-Himes \cite[Theorem 1.1]{MR4917220}. The techniques of those papers do not seem to generalize to this context. Instead, we prove Theorem \ref{theorem:main:surjection} by studying the fibers of the map in Equation \ref{quotientmap}. 

\begin{maintheoremB}
\label{theorem:main:computational}
Fix a prime $p \geq 3$ and let $t^\omega(n, p):= \rank (\rh_{n - 1}(\vert \sT_{2n}(\mathbb{Q})\vert /\Gamma_{2n}^\omega(p); \mathbb{Z}))$. For $n \geq 1$, $t^\omega(n, p)$ equals
\begin{equation} \label{mainequation}
\sum_{0 = m_{-1} < m_0 < \dots < m_k \leq n} \left( \prod_{i = 0}^{k - 1} p^{m_{i+1}-m_i\choose 2} \prod_{j = 0}^{m_i - 1} \frac{p^{m_{i + 1} - j} - 1}{p^{m_i - j} - 1} \right) \left( \prod_{i = 0}^{m_k - 1} \frac{p^{2n - 2i} - 1}{p^{m_k - i} - 1}\right) p^{{m_0 \choose 2} + (n - m_k)^2} \left(\frac{p - 3}{2}\right)^{k + 1}. 
\end{equation}
\end{maintheoremB}

Theorem \ref{theorem:main:computational} yields a nonzero lower bound of the rank of $\hh^{n^2}(\Gamma_{2n}^\omega(p);\mathbb{Q})$. Equation \ref{mainequation} grows exponentially with respect to both $p$ and $n$, demonstrating that the top-degree cohomology of $\Gamma_{2n}^\omega(p)$ is large. We refer the reader to Theorem \ref{theorem:main:computational:restated}, for a conceptual version of Theorem \ref{theorem:main:computational}. \\
Paraschivescu \cite[Theorem 1]{Par97} demonstrated that the rank of the top-degree cohomology of level-$p$ principal congruence subgroups of $\SL_n(\mathbb{Z})$ is at least 
\begin{equation*}
p^{n \choose 2} \left(\frac{p - 1}{2} \right)^{n - 1}
\end{equation*}
for $n \geq 1$. We present a symplectic version of Paraschivescu's \cite[Theorem 1]{Par97} lower bound that follows from Theorem \ref{theorem:main:computational}. 

\begin{maincorollaryC} \label{corollary:main:computational}
Fix a prime $p \geq 3$ and let $t^\omega(n, p)= \rank (\rh_{n - 1}(\vert \sT_{2n}(\mathbb{Q})\vert /\Gamma_{2n}^\omega(p); \mathbb{Z}))$. For $n \geq 1$,
\begin{equation}\label{ParSymp}
t^\omega(n, p) \geq p^{n^2} \left(\frac{p - 1}{2} \right)^n. 
\end{equation}
\end{maincorollaryC}

Similar to Equation \ref{mainequation}, Equation \ref{ParSymp} grows exponentially with respect to $p$ and $n$. But Equation \ref{ParSymp} is easier to compute for large $p$ and $n$. 

\begin{remark}
    It is easy to see that $\St_{2n}^\omega(\mathbb{Z})$ surjects onto $\St_{2n}^\omega(\Field_p)$ for all prime $p$, where $\Field_p$ is the finite field with $p$-many elements. This gives the weaker lower bound of $p^{n^2}$ for the rank of the top degree cohomology of these congruence subgroups.  
\end{remark}

\textbf{Acknowledgments.} I thank my advisor, Jeremy Miller. This article would not have been written without his help. I thank Peter Patzt for explaining his proof of his Theorem \ref{patzt}. I also thank Robin Sroka, Benjamin Br\"uck, and Jennifer Wilson for helpful discussions. I am grateful to Tatiana Abdelnaim, Urshita Pal, and Sarah Anderson for comments on drafts of this article. 


\section{Preliminaries} \label{Top:PosetsSimp}

\subsection{Symplectic Linear Algebra}

Throughout this article $R$ is a principal ideal domain. We work exclusively with finitely generated (f.g.) free $R$-modules. All symplectic $R$-modules are free. We always equip $R^{2n}$ with a symplectic form $\omega$ while we use $R^n$ to denote an $R$-module without a symplectic form. \\
Let $V \subseteq R^n$ be a submodule of rank $k$. We call $V$ a \emph{(direct) summand} if $R^n = V \oplus W$ for some submodule $W \subseteq R^n$ of rank $n - k$. The following two facts are elementary. A proof of the first fact can be found in \cite[Lemma 2.6]{ChurchPutnam17}.  

\begin{lemma}\label{LinearAlgebraSummands}
If $V$ and $W$ are summands of $R^n$ and $V \subseteq W$, then $V$ is a summand of $W$. 
\end{lemma}

\begin{lemma} \label{NormalBasis}
If $V \subseteq R^n$ is a summand and $B = \{\vec{v}_1, \dots, \vec{v}_m\}$ is a basis for $V$, then $B$ can be extended to a basis for $R^n$.
\end{lemma}

A \emph{partial basis} is a subset of a basis for $R^n$. A partial basis $\{\vec{v}_1, \dots, \vec{v}_m\}$ for $R^{2n}$ is \emph{isotropic} if $\omega(\vec{v}_i, \vec{v}_j) = 0$ for $i, j \in \{1, \dots, m\}$. A basis $\{\vec{v}_1, \vec{w}_1, \dots, \vec{v}_n, \vec{w}_n\}$ for $R^{2n}$ is a \emph{symplectic basis} if there is a symplectic (i.e., form-preserving) automorphism $R^{2n} \rightarrow R^{2n}$ that sends $\vec{v}_i$ to $\vec{e}_i$ and $\vec{w}_i$ to $\vec{f}_i$ for $i \in \{1, \dots, n\}$. A \emph{restricted basis} of $R^{2n}$ is a partial basis of a symplectic basis for $R^{2n}$ of size $2n - 1$. The \emph{symplectic complement} of a summand $V \subseteq R^{2n}$ is 
\begin{equation*}
V^\perp := \{ \vec{u} \in R^{2n} \mid \omega(\vec{v}, \vec{u}) = 0 \text{ for all } \vec{v} \in V\}. 
\end{equation*}

\begin{definition}
Let $V \subseteq R^{2n}$ be a summand. 
\begin{itemize}[topsep=-0.5cm]
\item We call $V$ a \emph{restricted summand} if $V \oplus W = R^{2n}$ for some rank-$1$ summand $W \subseteq R^{2n}$; 
\item We call $V$ an \emph{isotropic summand} if $\omega \vert_V \equiv 0$;
\item We call $V$ a \emph{symplectic summand} if $V$ is of rank $2m$ for $m \leq n$ and there is a symplectic isomorphism $V \cong R^{2m}$.
\end{itemize}
\end{definition}
\vspace{-.2cm}
The next three lemmas follow from Milnor--Husemoller \cite[Lemma I.2.6]{MH73}.

\begin{lemma}\label{IsotropicComplementSummands}
If $V \subseteq R^{2n}$ is a summand, then $V^\perp \subseteq R^{2n}$ is a summand.
\end{lemma}

\begin{lemma}\label{IsotropicPerp}
If $V \subsetneq R^{2n}$ is an isotropic summand, then the symplectic form on $R^{2n}$ induces a symplectic form on $V^\perp/V$. 
\end{lemma} 

\begin{lemma}\label{SymplecticBases}
If $V \subseteq R^{2n}$ is a summand and  
\begin{equation}\label{subset}
B = \{\vec{v}_1, \vec{w}_1, \dots, \vec{v}_m, \vec{w}_m, \vec{v}_{m + 1}, \dots, \vec{v}_{m + k}\}, 
\end{equation}
is a basis for $V$ such that 
\begin{equation*}
\omega(\vec{v}_{i_1}, \vec{v}_{i_2}) = \omega(\vec{w}_{j_1}, \vec{w}_{j_2}) = 0 \text{ and } \omega(\vec{v}_{i_1}, \vec{w}_{j_1}) = \delta_{i_1 j_1},
\end{equation*}
for all $i_1, i_2 \in \{1, \dots m + k\}$ and $j_1, j_2 \in \{1, \dots, m\}$. Then $B$ can be extended to a symplectic basis for $R^{2n}$.
\end{lemma}

There exist summands of $R^{2n}$ whose bases do not satisfy the conditions for Lemma \ref{SymplecticBases}. 

\begin{example}
Suppose $V \subset \mathbb{Z}^4$ is a summand spanned by $\{\vec{e}_1, 3\vec{f}_1 + \vec{e}_2\}$. This basis cannot be extended to a symplectic basis for $\mathbb{Z}^4$.
\end{example}

We show that restricted summands admit bases that satisfy the conditions for Lemma \ref{SymplecticBases}.  

\begin{definition}
Suppose $W$ is a symplectic $R$-module of rank $2n$ and let $\{\vec{w}_1, \dots, \vec{w}_{2n}\}$ be a set of vectors of $W$. Let $A(\vec{w}_1, \dots, \vec{w}_{2n})$ denote the $2n \times 2n$ \emph{anti-symplectic matrix for the set of vectors $\{\vec{w}_1, \dots, \vec{w}_{2n}\}$}, whose $(i, j)$-th entry is $\omega(\vec{w}_i, \vec{w}_j)$ for $i, j \in \{1, \dots, 2n\}$. Let $\Pi$ denote the set of partitions of $\{1, \dots, 2n\}$ into unordered pairs. We write the elements of $\Pi$ as 
\begin{equation*}
\nu = \{(i_1, j_1), \dots, (i_n, j_n)\}
\end{equation*}
so that $i_k < j_k$ for $k \in \{1, \dots, n\}$ and $i_1 < \dots < i_n$. Let $\Sigma_\Pi$ denote the subset of the symmetric group $\Sigma_{2n}$ whose elements $\lambda_\nu$ satisfy $\lambda_\nu(2k + 1) = i_{k + 1}$ and $\lambda_\nu(2k + 2) = j_{k + 1}$ for $(i_{k + 1}, j_{k + 1}) \in \nu$ and $k \in \{0, \dots, n - 1\}$. We define the \emph{Pfaffian of $A(\vec{w}_1, \dots, \vec{w}_{2n})$} as 
\begin{equation}
\pf(A(\vec{w}_1, \dots, \vec{w}_{2n})) := \sum_{\lambda_\nu \in \Sigma_\Pi} \sgn(\lambda_\nu) ~  \prod_{k = 0}^{n - 1} \omega(\vec{w}_{\lambda_\nu(2k + 1)}, \vec{w}_{\lambda_\nu(2k + 2)}). 
\end{equation}
\end{definition}

The next lemma follows from Knus \cite[Lemma 9.2.1]{Knus91}. 

\begin{lemma}\label{pfaffianOfDet1}
Suppose $\{\vec{w}_1, \dots, \vec{w}_{2n}\}$ is a basis for a symplectic $R$-module $W$ that is not necessarily symplectic. Then $\pf (A(\vec{w}_1, \dots, \vec{w}_{2n}))$ is a unit. 
\end{lemma}

\begin{lemma} \label{restrictedSummandTransitive}
Suppose $V \subsetneq R^{2n}$ is a restricted summand. Then there exists a basis 
\begin{equation}\label{maybesubset}
B = \{\vec{v}_1, \vec{w}_1, \dots, \vec{v}_{n - 1}, \vec{w}_{n - 1}, \vec{v}_n\}
\end{equation} 
for $V$ such that 
\begin{equation*}
\omega(\vec{v}_{i_1}, \vec{v}_{i_2}) = \omega(\vec{w}_{j_1}, \vec{w}_{j_2}) = 0 \text{ and } \omega(\vec{v}_{i_1}, \vec{w}_{j_1}) = \delta_{i_1 j_1},
\end{equation*}
for all $i_1, i_2 \in \{1, \dots n\}$ and $j_1, j_2 \in \{1, \dots, n - 1 \}$.
\end{lemma}

\begin{proof}
We proceed with induction on $n$. The case $2n = 2$ follows from Lemma \ref{SymplecticBases}, since a restricted summand of $R^2$ is isotropic. Fix $2n > 2$. For all $2m < 2n$, suppose for induction that if $U$ is a restricted summand of $R^{2m}$ then there exists a basis 
\begin{equation*}
\{\vec{v}_1, \vec{w}_1, \dots, \vec{v}_{m - 1}, \vec{w}_{m - 1}, \vec{v}_m\}
\end{equation*} 
for $U$ such that 
\begin{equation*}
\omega(\vec{v}_{i_1}, \vec{v}_{i_2}) = \omega(\vec{w}_{j_1}, \vec{w}_{j_2}) = 0 \text{ and } \omega(\vec{v}_{i_1}, \vec{w}_{j_1}) = \delta_{i_1 j_1},
\end{equation*}
for all $i_1, i_2 \in \{1, \dots m\}$ and $j_1, j_2 \in \{1, \dots, m - 1 \}$.\\
Let $V$ be a restricted summand of $R^{2n}$ for the inductive step. Choose a basis $\{\vec{u}_1, \dots, \vec{u}_{2n - 1}\}$ for $V$ and extend it to a basis $\{\vec{u}_1, \dots, \vec{u}_{2n}\}$ for $R^{2n}$ using Lemma \ref{NormalBasis}. We will construct a unimodula vector $\vec{v} \in V$ such that $\omega(\vec{v}, \vec{u}_{2n}) = 1$ using the Pfaffian of $A(\vec{u}_1, \dots, \vec{u}_{2n})$. This allows us to apply the inductive hypothesis on the symplectic $\Span{\vec{v}, \vec{u}_{2n}}^\perp$ which is of rank $2n - 2$.\\
Lemma \ref{pfaffianOfDet1} implies that the Pfaffian of $A(\vec{u}_1, \dots, \vec{u}_{2n})$ equals a unit $q$. Each term of the Pfaffian of $A(\vec{u}_1, \dots, \vec{u}_{2n})$ is a product of $n$-many symplectic forms and corresponds to some partition $\nu$. $\omega(\vec{u}_i, \vec{u}_{2n})$ is a term of each product for some $i \in \{1, \dots, 2n - 1\}$. For $i \in \{1, \dots, n\}$, let $\nu_i$ to be the partition that contains the term $(i, 2n)$ and set
\begin{equation*}
a_i = q^{-1} \cdot \left(\sum_{\lambda_{\nu_i} \in \Sigma_\Pi} \sgn(\lambda_{\nu_i})  ~  \prod_{\substack{k = 0, \\ \lambda_{\nu_i}(2k + 1) \neq i}}^{n - 1} \omega(\vec{u}_{\lambda_{\nu_i}(2k + 1)}, \vec{u}_{\lambda_{\nu_i}(2k + 2)})   \right).
\end{equation*}
It follows that
\begin{equation*}
\sum_{i = 1}^{2n - 1} a_i \cdot \omega(\vec{u}_i, \vec{u}_{2n}) = \omega \left( \sum_{i = 1}^{2n - 1} a_i \vec{u}_i, \vec{u}_{2n}   \right) = 1.
\end{equation*}
Set $\vec{w}_n = \vec{u}_{2n}$ and 
\begin{equation*}
\vec{v}_n = \sum_{i = 1}^{2n - 1} a_i \vec{u}_i.
\end{equation*}
Since $\gcd(\{a_i\}_{i = 1}^{n - 1}) = 1$, $\vec{v}_n$ is unimodular. We consider two cases: either $\Span{\vec{v}_n, \vec{w}_n}^\perp \subset V$ or not. If $\Span{\vec{v}_n, \vec{w}_n}^\perp \subset V$, then we use Lemma \ref{SymplecticBases} to complete the partial symplectic basis $\{\vec{v}_n, \vec{w}_n\}$ to a symplectic basis $\{\vec{v}_1, \vec{w}_1, \dots, \vec{v}_n, \vec{w}_n\}$ for $R^{2n}$ such that $\{\vec{v}_1, \vec{w}_1, \dots,\vec{v}_{n - 1}, \vec{w}_{n - 1}, \vec{v}_n\}$ is a basis for $V$.\\
If instead $\Span{\vec{v}_n, \vec{w}_n}^\perp \not \subseteq V$, then $V \cap \Span{\vec{v}_n, \vec{w}_n}^\perp$ is a restricted summand of $\Span{\vec{v}_n, \vec{w}_n}^\perp$. By the induction hypothesis, there exists a basis $\{\vec{x}_1, \vec{y}_1, \dots, \vec{x}_{n - 1}, \vec{y}_{n - 1}\}$ for $\Span{\vec{v}_n, \vec{w}_n}^\perp$ such that $\{\vec{x}_1, \vec{y}_1, \dots, \vec{x}_{n - 1}\}$ is a basis for $V \cap \Span{\vec{v}_n, \vec{w}_n}^\perp$ and 
\begin{equation*}
\omega(\vec{x}_i, \vec{x}_j) = \omega(\vec{y}_i, \vec{y}_j) = 0 \text{ and } \omega(\vec{x}_i, \vec{y}_j) = \delta_{ij},
\end{equation*}
for $i, j \in \{1, \dots, n - 1\}$.
\begin{equation*}
\{\vec{x}_1, \vec{y}_1, \dots, \vec{x}_{n - 1}, \vec{y}_{n - 1}, \vec{v}_n, \vec{w}_n\}
\end{equation*}
is a symplectic basis for $R^{2n}$. But $\{\vec{x}_1, \vec{y}_1, \dots, \vec{y}_{n - 2}, \vec{x}_{n - 1}, \vec{v}_n\}$ is a partial basis for $V$ and $\{\vec{y}_{n - 1}, \vec{w}_n\} \not \in V$. We remedy this issue by constructing a symplectic pair of vectors in $V$ from $\{\vec{x}_{n - 1}, \vec{y}_{n - 1}, \vec{v}_n, \vec{w}_n\}$ to then use the symplectic partial basis for $V$ to find a vector in $V$ that lies in the complement of this new symplectic partial basis for $V$.\\
Since $V$ is corank-1, then $\Span{\vec{y}_{n - 1}, \vec{w}_n} \cap V = \Span{b_1 \vec{y}_{n - 1} + b_2 \vec{w}_n}$ for some $b_1, b_2 \in R$. Set $\vec{y}_{n - 1}' = b_1 \vec{y}_{n - 1} + b_2 \vec{w}_n$. Since $\omega(\vec{x}_{n - 1}, \vec{y}_{n - 1}') = b_1$, $\omega(\vec{v}_n, \vec{y}_{n - 1}') = b_2$, and $\gcd(b_1, b_2)$ divides a unit, there exists $c_1, c_2 \in R$ such that 
\begin{equation*}
\omega(c_1 \vec{x}_{n - 1} + c_2 \vec{v}_n, \vec{y}_{n - 1}') = b_1c_1 + b_2c_2 = 1. 
\end{equation*}
Set $\vec{x}_{n - 1}' = c_1 \vec{x}_{n - 1} + c_2 \vec{v}_n$. The symplectic partial basis $\{\vec{x}_1, \vec{y}_1, \dots, \vec{x}_{n - 2}, \vec{y}_{n - 2}, \vec{x}_{n - 1}', \vec{y}_{n - 1}'\}$ is a partial basis of $V$. Since $V$ is corank-1,
\begin{equation*}
\Span{\vec{x}_1, \vec{y}_1, \dots, \vec{x}_{n - 2}, \vec{y}_{n - 2}, \vec{x}_{n - 1}', \vec{y}_{n - 1}'}^\perp \cap V = \Span{\vec{x}_n'},
\end{equation*}
for some $\vec{x}_n' \in V$. The basis $\{\vec{x}_1, \vec{y}_1, \dots, \vec{x}_{n - 2}, \vec{y}_{n - 2}, \vec{x}_{n - 1}', \vec{y}_{n - 1}', \vec{x}_n'\}$ of $V$ satisfies the conditions for Lemma \ref{SymplecticBases}. This concludes the inductive step and the proof.  
\end{proof}

The next proposition follows immediately from Lemmas \ref{SymplecticBases} and \ref{restrictedSummandTransitive}. 

\begin{proposition}\label{transitivityProved}
The group of symplectic automorphisms of $R^{2n}$, $\Sp_{2n}(R)$, acts transitively on the set of restricted summands of $R^{2n}$. 
\end{proposition}

\subsection{Posets and simplicial complexes}

We use the topological properties of several posets and the simplicial complexes associated to these posets to study the symplectic Steinberg module. Let $A$ be a poset. If $a_0 < \dots < a_k = a$ is a maximal chain ending with $a \in A$, then the \emph{height of $a$} is $k$. We work exclusively with posets whose elements have finite height. Let $\height\colon A \rightarrow \mathbb{Z}$ denote the height function that takes an element $a$ to its height. The \emph{dimension of $A$} is the maximum height attained by any $a \in A$. For $a \in A$, we call the subposet $A_{< a}:= \{b \in A \mid b < a\}$ the \emph{lower link of $a$}, we call the subposet $A_{> a}:= \{ b \in A \mid b > a\}$ the \emph{upper link of $a$}, and if $a < a' \in A$ then the subposet $A_{> a} \cap A_{< a'}:= \{b \in A \mid a < b < a'\}$ is the \emph{interval from $a$ to $a'$}.\\
The simplicial complex associated to $A$, denoted $\simp(A)$, is the simplicial complex whose vertices correspond to elements $a \in A$ and whose $k$-simplices correspond to chains of $(k + 1)$-many elements: $a_0 < \dots < a_k$. If $\sigma \in \simp(A)$ is a $k$-simplex that corresponds to $a_0 < \dots < a_k$, then the \emph{link of $\sigma$}, denoted $\lk_{\simp(A)}(\sigma)$, is the full sub-complex of $\simp(A)$ whose vertices correspond to $b \in A$ such that $b < a_0$, $b > a_k$ or $a_i < b < a_{i + 1}$ for $i \in \{0, \dots, k - 1\}$. And whose $k$-simplices are spanned by $(k + 1)$-many of these vertices. If $F\colon A \rightarrow B$ is a non-decreasing poset map, then the \emph{simplicial map associated to $F$}, $\simp(F)\colon \simp(A) \rightarrow \simp(B)$, is a simplicial map that sends a $k$-simplex associated to a chain $a_0 < \dots < a_k$ to $k$-simplex associated to the chain $F(a_0) < \dots < F(a_k)$.\\
The \emph{geometric realization} of $A$, denoted $\lvert A \rvert$, is isomorphic to the geometric realization of $\simp(A)$. For simplicity, when we say that a poset $A$ or the simplicial complex associated to $A$ has a topological property, we mean that $\lvert A \rvert$ has that property. For example, we write $\hh_*(A)$ when we mean $\hh_*(\lvert A \rvert)$. We make the convention that all posets and simplicial complexes are at least $(-2)$-connected and non-empty posets and simplicial complexes are always at least $(-1)$-connected.

\begin{definition}
A poset $A$ is \emph{Cohen--Macaulay} (abbreviated CM) of dimension $n$ if 
\begin{itemize}[nosep, topsep=-0.5cm]
\item $A$ is $n$-dimensional and $(n - 1)$-connected;
\item if for every $a \in A$ the lower link of $a$ is $(\height(a) - 1)$-dimensional and $(\height(a) - 2)$-connected, and the upper link of $a$ is $(n - \height(a) - 1)$-dimensional and $(n - \height(a) - 2)$-connected; 
\item and if for every pair $a < a'$ in $A$, the interval from $a$ to $a'$ is $(\height(a') - \height(a) - 2)$-dimensional and $(\height(a') - \height(a) - 3)$-connected.
\end{itemize}
\end{definition}

\begin{definition}
A simplicial complex $X$ is \emph{Cohen--Macaulay} of dimension $n$ if 
\begin{itemize}[nosep, topsep=-0.5cm]
\item $X$ is $n$-dimensional and $(n - 1)$-connected;
\item and if for every $(k - 1)$-simplex $\sigma \in X$ the link of $\sigma$ is $(n - k)$-dimensional and $(n - k - 1)$-connected.
\end{itemize}
\end{definition}

\begin{remark} \label{PosetToSimpLink}
If $\sigma \in \simp(A)$ is a $(k - 1)$-simplex associated to the chain $a_0 < \dots < a_{k - 1}$, then 
\begin{equation*}
\lk_{\simp(A)}(\sigma) \cong \simp(A_{< a_0}) * \simp(A_{> a_0} \cap A_{< a_1}) * \dots * \simp(A_{> a_{k - 2}} \cap A_{< a_{k - 1}}) * \simp(A_{> a_{k - 1}}). 
\end{equation*}
\end{remark}

A poset map $F\colon A \rightarrow B$ is \emph{$n$-connected} if the induced maps $\pi_i(A) \rightarrow \pi_i(B)$ are isomorphisms for $i < n$ and a surjection when $i = n$. This implies that the induced maps $\rh_i(A) \rightarrow \rh_i(B)$ are isomorphisms for $i < n$ and a surjection when $i = n$. We present several propositions and theorems used throughout the article that pertain to poset maps.

\begin{proposition} \textup{\cite[Proposition 1.6]{Q78}} \label{posetmaphomotopy2}
Let $F\colon A \rightarrow B$ be a poset map. If $F_{\leq b}$ is contractible for all $b \in B$, then $F$ is a homotopy equivalence. 
\end{proposition}  

\begin{proposition} \textup{\cite[Property 1.3]{Q78}} \label{posetmaphomotopy}
Let $F, G\colon A \rightarrow B$ be poset maps. If $F(a) \leq G(a)$ for all $a \in A$, then $F$ and $G$ are homotopic. 
\end{proposition}  

\begin{theorem} \textup{\cite[Corollary 9.7]{Q78}} \label{Q78Cor9.7}
Let $F\colon A \rightarrow B$ be a strictly increasing map of posets. Assume that $B$ is CM of dimension $n$ and that for all $b \in B$, the fiber $F_{\leq b}$ is $(\height(b) - 1)$-connected. Then $A$ is CM of dimension $n$. 
\end{theorem}   

\begin{proposition} \textup{\cite[Proposition 2.3]{ChurchPutnam17}} \label{CPQuillenFiber}
Let $F\colon A \rightarrow B$ be a map of posets. Assume that $B$ is CM of dimension $n$ and that for all $b \in B$, the fiber $F_{\leq b}$ is $(\height(b) - 1)$-connected. Then $F_*\colon \rh_n(A) \rightarrow \rh_n(B)$ is a surjection. 
\end{proposition}

\begin{theorem} \textup{\cite[Corollary 2.2]{VanDerKallenLooijenga}} \label{VdKaL11Cor2.2v2}
Let $F\colon A \rightarrow B$ be a map of posets, $n \in \mathbb{Z}$ and $t\colon B \rightarrow \mathbb{Z}$. Assume that for every $b \in B$, the poset fiber $F_{\leq b}$ is $(t(b) - 2)$-connected and that $B_{> b}$ is $(n - t(b) - 1)$-connected. Then the map $F$ is $n$-connected. 
\end{theorem}  

\subsection{Simplicial complexes and group actions}

We call a simplicial complex $X$ together with a discrete group $G$ acting simplicially on $X$ without rotations a \emph{$G$-complex}. A group acts without rotations when every simplex $\sigma \in X$ is stabilized pointwise by its stabilizer.

\begin{definition} \label{X/G}
Suppose $X$ is a $G$-complex. We denote the simplicial quotient of $X$ by $G$ with $X/G$. A vertex $s^*$ is in $X/G$ if $s^*$ is the $G$-orbit of a vertex $s \in X$. $\{s_0^*, \dots, s_k^*\}$ spans a $k$-simplex in $X/G$ if at least one choice of representatives $\{s_0, \dots, s_k\}$ of the $G$-orbits spans a $k$-simplex in $X$.
\end{definition}

The topological quotient $\lvert X \rvert/G$ does not always agree with $\lvert X/G \rvert$. However, there exists a map $\lvert X \rvert/G \rightarrow \lvert X/G \rvert$ that sends $G$-orbits of points of $\lvert X \rvert$ to points of $\lvert X/G \rvert$.  

\begin{example}
Let $X$ be a triangle spanned by the vertices $\{s_0, s_1, s_2\}$ and let $G$ be the cyclic group $C_3$. $X$ is a $C_3$-complex where all the vertices $s_i$ lie in the same orbit and all the edges $\{s_i, s_j\}$ lie in the same orbit. The simplicial quotient $X/C_3$ is a single vertex: $s^*$. The map $\varphi\colon \lvert X \rvert/C_3 \rightarrow \lvert X/C_3 \rvert$ sends all $C_3$-orbits of points of $X$ to $s^*$. $\varphi$ is not injective.
\end{example}

\begin{definition} \label{regularcomplex}
Let $X$ is a $G$-complex and let $g_0, \dots, g_k \in G$. Suppose that $\{s_0, \dots, s_k\}$ and $\{g_0 \cdot s_0, \dots, g_k \cdot s_k\}$ are sets of vertices that span simplices in $X$. The action of $G$ is called \emph{regular} if there exists $g \in G$ such that $g \cdot s_i = g_i \cdot s_i$ for $i \in \{0, \dots, k\}$. If the $G$-action is regular, then we call $X$ a \emph{regular $G$-complex}. 
\end{definition}

The following theorem is implied by Bredon \cite[Chapter III.1]{Bredon72}, although not stated explicitly. 

\begin{theorem} \textup{\cite[Chapter III.1]{Bredon72}}\label{Bredon}
Suppose $X$ is a regular $G$-complex. Then $\lvert X \rvert/G \cong \lvert X/G \rvert$ is a homeomorphism. 
\end{theorem}

We postpone the proof that the symplectic Tits building $\sT_{2n}(\mathbb{Q})$ is a regular $\Gamma_{2n}^\omega(p)$-complex to Section \ref{simplicialmodelisquotient}. We use the simplicial quotient $\sT_{2n}(\mathbb{Q})/\Gamma_{2n}^\omega(p)$ to compute the homology of the topological quotient $\lvert \sT_{2n}(\mathbb{Q}) \rvert/\Gamma_{2n}^\omega(p)$.  


\section{Tits Buildings}

\subsection{Tits Buildings of PIDs}\label{TitsBuildingSection}

Let $R$ be a PID and let $\Field$ be a field. In this subsection we recall the definitions of the classical Tits building for $\SL_n(\Field)$ and $\Sp_{2n}(\Field)$, and in the following subsections we generalize these constructions.

\begin{definition}
Suppose $V$ is a f.g.~free $R$-module. Let $\cbT(V)$ denote the poset of proper nonzero summands $U \subset V$ ordered by inclusion of summands. We call $\cbT(V)$ and $\cT(V):= \simp(\cbT(V))$ the \emph{Tits buildings of $V$}. 
\end{definition}

When $V = R^n$, we write $\cbT(R^n) = \cbT_n(R)$ and $\cT(R^n) = \cT_n(R)$. 

\begin{definition}\label{SymplecticTitsBuilding}
Suppose $W$ is a f.g.~symplectic $R$-module. Let $\sbT(W)$ denote the poset of proper nonzero isotropic summands $U \subset W$ ordered by inclusion of summands. We call $\sbT(W)$ and $\sT(W):= \simp(\sbT(W))$ the \emph{symplectic Tits buildings of $W$}.
\end{definition}

If $W = R^{2n}$, we write $\sbT(R^{2n}) = \sbT_{2n}(R)$ and $\sT(R^{2n}) = \sT_{2n}(R)$. 

\begin{theorem} \textup{\cite[Theorem 1]{Solomon}} \label{SBrB}
For $n > 0$, $m > 0$ and a field $\Field$, suppose $V$ is a $\Field$-vector space of dimension $n$ and $W$ is a symplectic $\Field$-vector space of dimension $2m$. Then $\cT(V)$ is CM of dimension $(n - 2)$ and $\sT(W)$ is CM of dimension $(m - 1)$. Furthermore, if $\lvert \Field \rvert$ is finite then $\cT(V)$ is homotopy equivalent to a wedge of $\lvert \Field \rvert^{n \choose 2}$ many $(n - 2)$-spheres and $\sT(W)$ is homotopy equivalent to a wedge of $\lvert \Field \rvert^{m^2}$ many $(m - 1)$-spheres. 
\end{theorem}

Brown \cite[Chapter IV.6]{BrownBuildings} provides a proof of Solomon--Tits' Theorem \ref{SBrB}. The following lemma mentioned by Miller--Patzt--Putman \cite[Lemma 3.10]{MiPaP21} extends Solomon--Tits' Theorem \ref{SBrB} from fields to PIDs. 

\begin{lemma} \textup{\cite[Lemma 3.10]{MiPaP21}} \label{MPP21Lem3.10}
Suppose $R$ is a PID, $\Field$ its field of fractions, and $V$ a f.g.~free $R$-module. Then $\cbT(V) \cong \cbT(V \otimes \Field)$. 
\end{lemma}

Lemma \ref{MPP21Lem3.10} follows from the fact that the poset map that takes a summand $U \subset V$ to $U \otimes \Field$ and the poset map that takes a subspace $U \subset V \otimes \Field$ to $U \cap V$ are inverses of each other. These poset maps have analogous poset maps in the symplectic context. Let $W$ be a f.g.~symplectic $R$-module. The poset map sending an isotropic summand $U \subset W$ to an isotropic subspace $U \otimes \Field$, and the poset map sending an isotropic subspace $U \subset \Field^{2n}$ to an isotropic summand $U \cap W$, are also inverses of each other.

\begin{proposition} \label{SympMPP21Lem3.10}
Suppose $R$ is a PID, $\Field$ its field of fractions, and $W$ a f.g.~symplectic $R$-module. Then $\sbT(W) \cong \sbT(W \otimes \Field)$. 
\end{proposition}  

\begin{definition}\label{SteinbergModules}
Let $V$ be a f.g.~free $R$-module of rank $n$ and let $W$ be a f.g.~symplectic $R$-module of rank $2m$. We denote by $\St(V) := \rh_{n - 2}(\cbT(V); \mathbb{Z})$ the \emph{Steinberg module of $V$} and we denote by $\St^\omega(W):= \rh_{m - 1}(\sbT(W); \mathbb{Z})$ the \emph{symplectic Steinberg module of $W$}.
\end{definition} 

We write $\St(R^n) = \St_n(R)$ and $\St^\omega(R^{2n}) = \St_{2n}^\omega(R)$ when $V = R^n$ and $W = R^{2n}$. We now consider the links of $\sbT(W)$. These are well-known computations. We include a proof to make it clear how we generalize these computations to similar contexts.

\begin{lemma}\label{SympPosetLinks}
Suppose $U, U' \in \sbT(W)$, then 
\begin{itemize}[nosep, topsep=-0.5cm]
\item the lower link of $U$ is isomorphic to $\cbT(U)$;
\item the upper link of $U$ is isomorphic to $\sbT(U^\perp/U)$;
\item and if $U \subsetneq U'$, the interval from $U$ to $U'$ is isomorphic to $\cbT(U'/U)$. 
\end{itemize}
\end{lemma}  

\begin{proof}
Since every summand of $U$ is isotropic, the lower link of $U$ is isomorphic to $\cbT(U)$. \\
The projection map $W \rightarrow U^\perp/U$ induces a poset isomorphism between $\sbT(W)_{> U}$ and $\sbT(U^\perp/U)$ in the following manner. Pick a symplectic summand $S$ such that $U \oplus S = U^\perp$. The poset maps 
\begin{align*}
f & \colon \sbT(W)_{> U} \rightarrow \sbT(S)\colon Q \mapsto Q \cap S, \\
g & \colon \sbT(S) \rightarrow \sbT(W)_{> U} \colon Q \mapsto Q \oplus U,
\end{align*}
are inverses of each other that send isotropic summands to isotropic summands. Thus, $\sbT(W)_{> U}$ is isomorphic to $\sbT(S)$. The canonical symplectic isomorphism $S \cong U^\perp/U$ induces the isomorphism $\sbT(S) \cong \sbT(U^\perp/U)$. \\
The interval from $U$ to $U'$ is isomorphic to $\cbT(U')_{> U}$. The projection map $U' \rightarrow U'/U$ induces a poset isomorphism between $\cbT(U')_{> U}$ and $\cbT(U'/U)$.
\end{proof}

\subsection{$\pm$-oriented Tits Buildings of $\Field_p$} \label{pmTitsBuildingSection}

Let $p$ be a prime integer and let $\Field_p$ denote the finite field with $p$-many elements. We denote $\Field_p$-vector spaces with an overline. We extend a construction introduced by Miller--Patzt--Putman \cite{MiPaP21} to the symplectic context. The construction we describe does not satisfy the technical definition of a building. In Section \ref{simplicialmodelisquotient} we show that this construction is isomorphic to the simplicial quotient $\sT_{2n}(\mathbb{Q})/\Gamma_{2n}^\omega(p)$. 

\begin{definition}
Let $\vec{v}$ be a nonzero vector of a f.g.~free $R$-module $V$. A \emph{$\pm$-vector} is a set $\{-\vec{v}, + \vec{v}\}$. We will write $\pm \vec{v}$ for the associated $\pm$-vector $\{-\vec{v}, + \vec{v}\}$. 
\end{definition}

\begin{definition}
Let $V$ be a f.g.~free $R$-module. We say that an \emph{orientation on $V$} is an element $\alpha \in \wedge^\text{top} V$ that generates $\wedge^\text{top} V$ as an $R$-module, where top is shorthand for $\rank V$. A \emph{$\pm$-orientation on $V$} is a $\pm$-vector $\pm \alpha$ where $\alpha$ is an orientation on $V$. Oftentimes, we denote $V$ equipped with $\pm \alpha$ as $(V, \pm \alpha)$.
\end{definition}  

\begin{remark}\label{setOfOrientations}
Note that $\wedge^{\text{top}} V$ is a rank-1 module. The set of orientations on an $R$-module $V$ is in bijection $R^\times$. So, the set of $\pm$-orientations on $V$ is in bijection with $R^\times/\{-1, + 1\}$. Also, the group $R^\times$ acts transitively on the set of $\pm$-orientations on $V$. 
\end{remark}

Let $R = \Field_p$ and let $\overline{V}$ be an $\Field_p$-vector space. By Remark \ref{setOfOrientations}, the set of $\pm$-orientations on $\overline{V}$ is in bijection with the set $\Field_p^\times/\mathbb{Z}^\times$ and $\Field_p^\times$ acts transitively on the set of $\pm$-orientations for $\overline{V}$. $\overline{V}$ admits $(\frac{p - 1}{2})$-many distinct $\pm$-orientations. 

\begin{definition}
Suppose $\overline{V}$ is an $\Field_p$-vector space.  Let $\cbTD(\overline{V})$ denote the poset of proper nonzero subspaces $\overline{U} \subset \overline{V}$ equipped with $\pm$-orientations ordered by inclusion of subspaces and let $\cTD(\overline{V}):= \simp(\cbTD(\overline{V}))$. We call $\cbTD(\overline{V})$ and $\cTD(\overline{V})$ the \emph{$\pm$-oriented Tits building of $\overline{V}$}. 
\end{definition}

\begin{remark}\label{pmorientationdonotmatter}
Let $(\overline{U}, \pm \alpha), (\overline{Q}, \pm \beta) \in \cbTD(\overline{V})$. $(\overline{U}, \pm \alpha) \subsetneq (\overline{Q}, \pm \beta)$ if and only if $\overline{U} \subsetneq \overline{Q}$. In the case that $\overline{U} = \overline{Q}$ but $\pm \alpha \neq \pm \beta$, then $(\overline{U}, \pm \alpha)$ and $(\overline{U}, \pm \beta)$ are incomparable. 
\end{remark}

If $\overline{V}= \Field_p^n$, we write $\cbTD(\Field_p^n) = \cbTD_n(\Field_p)$ and $\cTD(\Field_p^n) = \cTD_n(\Field_p)$.

\begin{definition}
Suppose $V$ and $V'$ are f.g.~free $R$-modules, and suppose there exists an isomorphism $\varphi\colon V \rightarrow V'$. The map $\varphi$ induces a map $\varphi_*\colon \wedge^\text{top} V \rightarrow \wedge^\text{top} V'$ that sends a $\pm$-orientation on $V$ to a $\pm$-orientation on $V'$. Let $\alpha_\varphi := \varphi_*(\alpha)$ in $\wedge^\text{top} V'$. We write $\varphi\colon (V, \pm \alpha) \rightarrow (V', \pm \alpha_\varphi)$ to imply that $\pm \alpha_\varphi$ is the induced $\pm$-orientation on $V'$. 
\end{definition}

\begin{definition}
Suppose $\overline{W}$ is a symplectic $\Field_p$-vector space. Let $\sbTD(\overline{W})$ denote the poset of proper nonzero isotropic subspaces $\overline{U} \subset \overline{W}$ equipped with $\pm$-orientations ordered by inclusion of subspaces and let $\sTD(\overline{W}):= \simp(\sbTD(\overline{W}))$. We call $\sbTD(\overline{W})$ and $\sTD(\overline{W})$ \emph{$\pm$-oriented symplectic Tits buildings of $\overline{W}$}.
\end{definition}

When $\overline{W} = \Field_p^{2n}$, we write $\sbTD(\Field_p^{2n}) = \sbTD_{2n}(\Field_p)$ and $\sTD(\Field_p^{2n}) = \sTD_{2n}(\Field_p)$. 

\begin{proposition} \textup{\cite[Lemma 3.15]{MiPaP21}}\label{lemma:tdconn}
Suppose $n \geq 0$ and $\overline{V}$ is an $\Field_p$-vector space of dimension $n$. Then $\cTD(\overline{V})$ is CM of dimension $(n - 2)$.
\end{proposition}

We prove the symplectic analogue of Miller--Patzt--Putman's Proposition \ref{lemma:tdconn} using the following definition and proposition.

\begin{definition}\label{completeJoinComplex}
Suppose $X, Y$ are simplicial complexes and $\varphi\colon Y \rightarrow X$ is a simplicial map. $Y$ with the map $\varphi$ is a \emph{complete join complex over $X$} if $\varphi$ is a surjection, $\varphi$ is an injection on individual simplices, and, if $\sigma$ is $k$-simplex in $X$ spanned by the vertex set $\{s_0, \dots, s_k\}$, then the full sub-complex of $Y$ that projects down to $\sigma$ is the join $\varphi^{-1}(s_0) * \dots * \varphi^{-1}(s_k)$.
\end{definition}

\begin{proposition} \textup{\cite[Proposition 3.5]{HatcherWahl}}\label{HWProp3.5}
Suppose $Y$ is a complete join complex over a CM complex $X$ of dimension $n$. Then $Y$ is also CM of dimension $n$.
\end{proposition}

\begin{proposition} \label{SympMPP21Lem3.15}
Suppose $n > 0$ and $\overline{W}$ is a symplectic $\Field_p$-vector space of dimension $2n$. Then $\sTD(\overline{W})$ is Cohen--Macaulay of dimension $(n - 1)$.
\end{proposition}  

\begin{proof}
Consider the following poset map 
\begin{align*}
\varphi & \colon \sbTD(\overline{W}) \rightarrow \sbT(\overline{W}),\\
\varphi & \colon (\overline{V}, \pm \alpha) \mapsto \overline{V}. 
\end{align*}
We check that $\sTD(\overline{W})$ with the simplicial map associated to $\varphi$, 
\begin{equation*}
\tilde{\varphi} \colon \sTD(\overline{W}) \rightarrow \sT(\overline{W}), 
\end{equation*}
is a complete join complex over $\sT(\overline{W})$. If $\overline{V} \in \sbT(\overline{W})$, then 
\begin{equation*}
\varphi^{-1}(\overline{V}) = \{(\overline{V}, \pm \alpha) \mid \alpha \in \wedge^{\text{top}} \overline{V}\}.
\end{equation*}
This is a discrete set because different $\pm$-orientations on the same vector space are incomparable. Given an isotropic flag 
\begin{equation*}
0 \subsetneq \overline{V}_0 \subsetneq \dots \subsetneq \overline{V}_k \subsetneq \overline{W}
\end{equation*}
associated to a $k$-simplex $\sigma \in \sT(\overline{W})$, the full sub-complex that $\tilde{\varphi}$ projects onto $\sigma$ is 
\begin{equation*}
    X := \simp(\varphi^{-1}(\overline{V}_0)) * \dots * \simp(\varphi^{-1}(\overline{V}_k)).
\end{equation*}
This follows from the fact that different $\pm$-orientations do not play a role in deciding if two simplices of different dimensions form a simplex. \\
If $\tau \in X$ is a $k$-simplex, then $\tilde{\varphi}(\tau) = \sigma$. Hence $\tilde{\varphi}$ is surjective. If we fix an $r$-simplex $\gamma \in X$, then there is a bijection between the faces of $\gamma$ and the faces of $\tilde{\varphi}(\gamma)$. Thus, $\tilde{\varphi}$ is injective on individual simplices. We conclude that $\sTD(\overline{W})$ with $\tilde{\varphi}$ is a complete join complex over $\sT(\overline{W})$. Hatcher--Wahl's Proposition \ref{HWProp3.5} and Solomon--Tit's Theorem \ref{SBrB} imply that $\sTD(\overline{W})$ is CM of dimension $(n - 1)$.
\end{proof}

\begin{definition}
Let $\overline{V}$ be an $n$-dimensional $\Field_p$-vector space and $\overline{W}$ is a $2m$-dimensional symplectic $\Field_p$-vector space. Denote by $\St^\pm(\overline{V}):= \rh_{n - 2}(\cTD(\overline{V}); \mathbb{Z})$ the \emph{$\pm$-oriented Steinberg module} and denote by $\St^{\omega, \pm}(\overline{W}):= \rh_{m - 1}(\sTD(\overline{W}); \mathbb{Z})$ the \emph{$\pm$-oriented symplectic Steinberg module}.
\end{definition}

We write $\St^\pm(\Field_p^n) = \St_n^\pm(\Field_p)$ and $\St^{\omega, \pm}(\Field_p^{2n}) = \St_{2n}^{\omega, \pm}(\Field_p)$ when $\overline{V} = \Field_p^n$ and $\overline{W} = \Field_p^{2n}$. \\
We conclude this section by computing the links of the $\pm$-oriented symplectic Tits building. We state several properties of $\pm$-orientations that we use.  

\begin{remark}\label{behaviorOfpmOrientations}
Let $U$ and $V$ be summands of a f.g.~free $R$-module. If a $\pm$-orientation on $U$ is fixed, we argue that there is a natural bijection between the set of $\pm$-orientations on $V/U$ and the set of $\pm$-orientations on $V$. Pick $W$ such that $U \oplus W = V$. The map $\varphi\colon U \oplus W \rightarrow V$ induces a map 
\begin{align*}
\varphi_* & \colon \wedge^{\text{top}} U \otimes \wedge^{\text{top}} W \rightarrow \wedge^{\text{top}} V, \\
\varphi_* & \colon \alpha \otimes \beta \mapsto \alpha \wedge \beta,
\end{align*}
where $\alpha \wedge \beta = \gamma \in \wedge^{\text{top}} V$. If we fix an orientation $\alpha \in \wedge^{\text{top}} U$, then we get a natural bijection of the set of orientations on $W$ and the set of orientations on $V$. Thus, if we fix the $\pm$-orientation on $U$, then there is a natural bijection between the set of $\pm$-orientations on $W$ and the set of $\pm$-orientations on $V$. \\
The canonical isomorphism $\psi\colon W \cong V/U$ induces an isomorphism $\psi_*\colon \wedge^{\text{top}} W \cong \wedge^{\text{top}} V/U$. Thus, there is a natural bijection between the set of $\pm$-orientations on $W$ and the set of $\pm$-orientations on $V/U$. We conclude that that there is a natural bijection between the set of $\pm$-orientations on $V/U$ and the set of $\pm$-orientations on $V$ when a $\pm$-orientation on $U$ is fixed. \\
It is easy to check that this natural bijection is independent of the choice of $W$. This is because the determinant of block upper triangular matrix associated to $V$ depends on the determinant of the blocks on the diagonal, which are associated to $U$ and $V/U$. 
\end{remark}

\begin{lemma}\label{pmSympPosetLinks}
Suppose $(\overline{U}, \pm \beta), (\overline{U}', \pm \beta') \in \sbTD(\overline{W})$ such that $\overline{U} \subsetneq \overline{U}'$. Then 
\begin{itemize}[nosep, topsep=-0.5cm]
\item the lower link of $(\overline{U}, \pm \beta)$ is isomorphic to $\cbTD(\overline{U})$;
\item the upper link of $(\overline{U}, \pm \beta)$ is isomorphic to $\sbTD(\overline{U}^\perp/\overline{U})$;
\item and the interval from $(\overline{U}, \pm \beta)$ to $(\overline{U}', \pm \beta')$ is isomorphic to $\cbTD(\overline{U}'/\overline{U})$.
\end{itemize}
\end{lemma}  

\begin{proof}
The lower link of $(\overline{U}, \pm \beta)$ is isomorphic to $\cbTD(\overline{U})$ for the same reason that the lower link of $\overline{U}$ in $\sT(\overline{W})$ is isomorphic to $\cbT(\overline{U})$. \\
By Lemma \ref{SympPosetLinks}, there exists an isomorphism between $\cbT(\overline{W})_{> \overline{U}}$ and $\cbT(\overline{U}^\perp/\overline{U})$. This isomorphism is induced by the projection map $\varphi\colon \overline{W} \rightarrow \overline{U}^\perp/\overline{U}$. $\varphi$ induces a map that sends $\pm$-orientations on $\overline{Q} \supsetneq \overline{U}$ to $\pm$-orientations on $\overline{Q}/\overline{U}$. Since the $\pm$-orientation on $\overline{U}$ is $\pm \beta$, Remark \ref{behaviorOfpmOrientations} says the induced map on $\pm$-orientations is a bijection. Thus, the induced map on posets 
\begin{align*}
\varphi_* & \colon \cbTD(\overline{W})_{> (\overline{U}, \pm \beta)} \rightarrow \cbTD(\overline{U}^\perp/\overline{U}), \\
\varphi_* & \colon (\overline{Q}, \pm \gamma) \mapsto (\overline{Q}/\overline{U}, \pm \gamma_\varphi),
\end{align*}
is an isomorphism. \\
The interval from $(\overline{U}, \pm \beta)$ to $(\overline{U}', \pm \beta')$ is isomorphic to $\cbTD(\overline{U}')_{>(\overline{U}, \pm \beta)}$. Using similar reasoning as with the upper link, the map $\overline{U}' \rightarrow \overline{U}'/\overline{U}$ induces a poset isomorphism between $\cbTD(\overline{U}')_{>(\overline{U}, \pm \beta)}$ and $\cbTD(\overline{U}'/\overline{U})$.
\end{proof}

\subsection{Symplectic Tits Buildings of $\mathbb{Z}$ with $\pm$-congruence conditions}\label{TitsBuildingPmCongruence}

A f.g.~free $\mathbb{Z}$-module has a unique $\pm$-orientation by Remark \ref{setOfOrientations}. 

\begin{definition}
Suppose $V$ is a f.g.~free $\mathbb{Z}$-module, $\overline{V}$ is an $\Field_p$-vector space, and $\pi\colon V \rightarrow \overline{V}$ is a surjection. The map $\pi$ induces a map $\pi_*\colon \wedge^\text{top} V \rightarrow \wedge^\text{top} \overline{V}$ that sends the unique $\pm$-orientation on $V$ to a $\pm$-orientation on $\overline{V}$. Set $\alpha_\pi := \pi_*(\alpha)$ for $\alpha \in \wedge^\text{top} V$. $\pm \alpha_\pi$ is the \emph{$\pm$-orientation on $\overline{V}$ induced by $\pi$}. We write $\pi\colon V \rightarrow (\overline{V}, \pm \alpha_\pi)$ to imply that $\pm \alpha_\pi$ is the induced $\pm$-orientation on $\overline{V}$. 
\end{definition}

Let $W$ be a f.g.~symplectic $\mathbb{Z}$-module, let $\overline{W}$ be a symplectic $\Field_p$-vector space and let $\pi\colon W \rightarrow \overline{W}$ be a surjection. $\pi$ induces a poset map $\sbT(W) \rightarrow \sbTD(\overline{W})$. In this section, we consider the poset fibers of $(\overline{V}, \pm \alpha) \in \sbTD(\overline{W})$ with respect to the induced poset map. This is a subposet of $\sbT(W)$. We postpone to Section \ref{titsconghighconnectivity} the demonstration that this subposet is CM of dimension $\height(\overline{V}, \pm \alpha)$. 

\begin{notation}
We decorate $\Field_p$-vector spaces with an overline to distinguish them from $\mathbb{Z}$-modules. In general, we use the same variable to imply that there exists a surjection between the $\mathbb{Z}$-module and the $\Field_p$-module. We use $W$ to denote a f.g.~symplectic $\mathbb{Z}$-module and $\overline{W}$ to denote a symplectic $\Field_p$-vector space, and we use $\overline{V} \subset \overline{W}$ to denote an isotropic $\Field_p$-subspace. Given a surjective map $\pi\colon W \rightarrow \overline{W}$, we write $\pi\colon \sbT(W) \rightarrow \sbTD(\overline{W})$ to mean the induced poset map.
\end{notation}

\begin{definition}
Let $W$ be a f.g.~symplectic $\mathbb{Z}$-module, let $\overline{W}$ be a symplectic $\Field_p$-vector space, let $\pi\colon W \rightarrow \overline{W}$ be a surjection, and let $(\overline{V}, \pm \alpha) \in \sbTD(\overline{W})$. We denote by $\sbT(W \vert (\overline{V}, \pm \alpha))$ the poset whose elements are isotropic summands $U \subset W$ ordered by inclusion that satisfy the following conditions: we require that $\overline{U} \subseteq \overline{V}$, where $\pi(U)= (\overline{U}, \pm \beta_\pi)$, and, if $\overline{U} = \overline{V}$, we require $\pm \beta_\pi = \pm \alpha$. We call $\sbT(W \vert (\overline{V}, \pm \alpha))$ and $\sT(W \vert (\overline{V}, \pm \alpha)):= \simp(\sbT(W \vert (\overline{V}, \pm \alpha)))$ the \emph{symplectic Tits buildings of $W$ with a $\pm$-congruence condition with respect to $(\overline{V}, \pm \alpha)$}. 
\end{definition}

\begin{remark}
$\sbT(W \vert (\overline{V}, \pm \alpha))$ is the poset fiber of $(\overline{V}, \pm \alpha) \in \sbTD(\overline{W})$ with respect to a surjective poset map $\pi\colon \sbT(W) \rightarrow \sbTD(\overline{W})$. That is, 
\begin{equation*}
\pi^{-1}(\sbTD(\overline{W})_{\leq (\overline{V}, \pm \alpha)}) = \sbT(W \vert (\overline{V}, \pm \alpha)). 
\end{equation*}
\end{remark}

\begin{lemma} \label{SympCongruencePosetLinks}
Let $\pi\colon W \rightarrow \overline{W}$ be a surjection. Suppose $U, U' \in \sbT(W \vert (\overline{V}, \pm \alpha))$ such that $U \subsetneq U'$. Then
\begin{itemize}[nosep, topsep=-0.5cm]
\item the lower link of $U$ is isomorphic to $\cbT(U)$;
\item the upper link of $U$ is isomorphic to $\sbT(U^\perp/U \vert (\overline{V}/\overline{U}, \pm \eta))$, where $(\overline{V}/\overline{U}, \pm \eta) \in \sbTD(\overline{U}^\perp/\overline{U})$ and $\pi_*\colon U^\perp/U \rightarrow \overline{U}^\perp/\overline{U}$ is a surjection induced by $\pi$;
\item the interval from $U$ to $U'$ is isomorphic to $\cbT(U'/U)$.
\end{itemize} 
\end{lemma}

\begin{proof}
Since every summand $Q \subset U$ is isotropic and $\pi(Q) = \overline{Q} \subset \overline{U} \subset \overline{V}$, the lower link of $U$ is isomorphic to $\cbT(U)$. Likewise, the interval from $U$ to $U'$ is isomorphic to $\cbT(U')_{> U}$. $\cbT(U')_{> U}$ is isomorphic to $\cbT(U'/U)$.   \\
We show that there is a natural bijection between $\sbT(W \vert (\overline{V}, \pm \alpha))_{> U}$ and $\sbT(U^\perp/U \vert (\overline{V}/\overline{U}, \pm \eta))$ by constructing a commutative diagram that delineates the congruence condition. Let $\pi(U) = (\overline{U}, \pm \beta_\pi)$, and consider the following diagram of poset maps:
\begin{equation} \label{computingLinksDiagram}
\begin{tikzcd}
\sbT(W)_{> U} \arrow[d, "\pi"'] \arrow[r, "\varphi"] & \sbT(U^\perp/U)   \arrow[d, "\pi_*"'] \\
\sbTD(\overline{W})_{> (\overline{U}, \pm \beta_\pi)} \arrow[r, "\psi"] & \sbTD(\overline{U}^\perp/\overline{U})
\end{tikzcd}
\end{equation}
Let $\varphi$ and $\psi$ be the poset maps induced by the projection maps $W \rightarrow U^\perp/U$ and $\overline{W} \rightarrow \overline{U}^\perp/\overline{U}$, respectively. These are poset isomorphisms by Lemmas \ref{SympPosetLinks} and \ref{pmSympPosetLinks}. $\pi_*$ is a surjection induced by $\pi$. Thus, $\pi_* \circ \varphi = \psi \circ \pi$. Set $\psi(\overline{V}, \pm \alpha) = (\overline{V}/\overline{U}, \pm \alpha_{\psi})$. Note that 
\begin{equation*}
\pi^{-1}(\sbTD(\overline{W})_{\leq (\overline{V}, \pm \alpha)} \cap \sbTD(\overline{W})_{> (\overline{U}, \pm \beta)}) = \sbT(W \vert (\overline{V}, \pm \alpha))_{> U}
\end{equation*}
and 
\begin{equation*}
\pi_*^{-1}(\sbTD(\overline{U}^\perp/\overline{U})_{\leq (\overline{V}/\overline{U}, \pm \alpha_{\psi})}) = \sbT(U^\perp/U \vert (\overline{V}/\overline{U}, \pm \alpha_{\psi})).
\end{equation*}
Since Diagram \ref{computingLinksDiagram} commutes, it follows that 
\begin{equation*}
\sbT(W \vert (\overline{V}, \pm \alpha))_{> U} \cong \sbT(U^\perp/U \vert (\overline{V}/\overline{U}, \pm \alpha_{\psi})).
\end{equation*}
If $\pm \alpha_{\psi} \neq \pm \eta$, let $\psi'\colon \sbTD(\overline{U}^\perp/\overline{U}) \rightarrow \sbTD(\overline{U}^\perp/\overline{U})$ be a poset isomorphism induced by a symplectic automorphism that sends $(\overline{V}/\overline{U}, \pm \alpha_{\psi})$ to $(\overline{V}/\overline{U}, \pm \eta)$. Pick $V/U \subset U^\perp/U$ such that $\pi_*(V/U) = (\overline{V}/\overline{U}, \pm \alpha_\psi)$ and pick $V'/U \subset U^\perp/U$ such that $\pi_*(V'/U) = (\overline{V}/\overline{U}, \pm \eta)$. Let $\varphi'\colon \sbT(U^\perp/U) \rightarrow \sbT(U^\perp/U)$ be the poset isomorphism induced by a symplectic automorphism that sends $V/U$ to $V'/U$. The following diagram is commutative:
\begin{equation}\label{commutativeDiagram2}
\begin{tikzcd}
\sbT(U^\perp/U) \arrow[d, "\pi_*"'] \arrow[r, "\varphi'"] & \sbT(U^\perp/U)   \arrow[d, "\pi_*"'] \\
\sbTD(\overline{U}^\perp/\overline{U}) \arrow[r, "\psi'"] & \sbTD(\overline{U}^\perp/\overline{U})
\end{tikzcd}
\end{equation}
So there exists an isomorphism between 
\begin{equation*}
\pi_*^{-1}(\sbTD(\overline{U}^\perp/\overline{U})_{\leq (\overline{V}/\overline{U}, \pm \alpha_{\psi})}) = \sbT(U^\perp/U \vert (\overline{V}/\overline{U}, \pm \alpha_{\psi})),
\end{equation*}
and 
\begin{equation*}
\pi_*^{-1}(\sbTD(\overline{U}^\perp/\overline{U})_{\leq (\overline{V}/\overline{U}, \pm \eta)}) = \sbT(U^\perp/U \vert (\overline{V}/\overline{U}, \pm \eta)).
\end{equation*}
\end{proof}

\subsection{Restricted symplectic Tits Buildings of $\mathbb{Z}$ with $\pm$-congruence conditions}\label{ResTitsBuildingPmCongruence}

Let $W$ be a symplectic $R$-module. Recall that a summand is a restricted summand if it is of corank-1. We make the convention of denoting restricted summands with a hat: $\hat{W}$. This applies for both $\mathbb{Z}$-modules and $\Field_p$-vector spaces. 

\begin{definition}
Let $R$ be a PID. Let $W$ be a f.g.~symplectic $R$-module and let $\hat{W} \subsetneq W$ be a restricted summand. Let $\sbT(\hat{W})$ denote the subposet of $\sbT(W)$ whose elements are nonzero proper isotropic summands $U \subseteq \hat{W}$. We call $\sbT(\hat{W})$ and $\sT(\hat{W}):= \simp(\sbT(\hat{W}))$ the \emph{restricted symplectic Tits building of $\hat{W}$}. 
\end{definition}

For a thorough investigation of a restricted symplectic Tits building of $\mathbb{Q}$ we refer the reader to Br\"uck--Sroka \cite{BruckSroka24}. We compute the poset links of $\sbT(\hat{W})$. 

\begin{lemma} \label{ResSympPosetLinks}
Suppose $U, U' \in \sbT(\hat{W})$ such that $U \subsetneq U'$. 
\begin{itemize}[nosep, topsep=-0.5cm]
\item The lower link of $U$ is isomorphic to $\cbT(U)$;
\item if $U^\perp \cap \hat{W} = U^\perp$, then the upper link of $U$ is isomorphic to $\sbT(U^\perp/U)$;
\item if $U^\perp \cap \hat{W} \neq U^\perp$, then the upper link of $U$ is isomorphic to $\sbT((U^\perp \cap \hat{W})/U)$;
\item the interval from $U$ to $U'$ is isomorphic to $\cbT(U'/U)$.
\end{itemize} 
\end{lemma}

\begin{proof}
Except for the upper link of $U$ in the case that $U^\perp \cap \hat{W} \neq U^\perp$, we compute the links of $\sbT(\hat{W})$ using the same arguments as in the proof of Lemma \ref{SympPosetLinks}. If $U^\perp \cap \hat{W} \neq U^\perp$ then $U^\perp \cap \hat{W}$ is a restricted summand of $U^\perp$. Choose a symplectic summand $S$ such that $U \oplus S = U^\perp$. In addition, choose a restricted summand $\hat{S} \subsetneq S$ such that $U \oplus \hat{S} = U^\perp \cap \hat{W}$. The poset maps 
\begin{align*}
\hat{f} & \colon \sbT(\hat{W})_{> U} \rightarrow \sbT(\hat{S})\colon Q \mapsto Q \cap S, \\
\hat{g} & \colon \sbT(\hat{S}) \rightarrow \sbT(\hat{W})_{> U} \colon Q \mapsto Q \oplus U,
\end{align*}
are inverses of each other that send isotropic summands to isotropic summands. So $\sbT(\hat{W})_{> U}$ is isomorphic to $\sbT(\hat{S})$. By Lemma \ref{restrictedSummandTransitive} we can find a basis for $\hat{S}$ that is subset of a symplectic basis for $S$. The canonical isomorphism $S \cong U^\perp/U$ sends the symplectic basis for $S$ to a symplectic basis for $U^\perp/U$ and it sends the basis for $\hat{S}$ to a basis for $(U^\perp \cap \hat{W})/U$. Thus, the canonical isomorphism induces the isomorphism $\hat{S} \cong (U^\perp \cap \hat{W})/U$ which in turn induces an isomorphism between $\sbT(\hat{S})$ and $\sbT((U^\perp \cap \hat{W})/U)$. This is independent of the choice of $\hat{S}$. Thus, the projection map $\hat{W} \rightarrow (U^\perp \cap \hat{W})/U$ induces a poset isomorphism between $\sbT(\hat{W})_{> U}$ and $\sbT((U^\perp \cap \hat{W})/U)$. 
\end{proof}

\begin{definition}
Let $\overline{W}$ be a symplectic $\Field_p$-module and let $\hat{\overline{W}} \subsetneq \overline{W}$ be a restricted subspace. Let $\sbTD(\hat{\overline{W}})$ denote the subposet of $\sbTD(\overline{W})$ whose elements are nonzero proper isotropic subspaces $\overline{U} \subseteq \hat{\overline{W}}$ equipped with a $\pm$-orientation. We call $\sbTD(\hat{\overline{W}})$ and $\sTD(\hat{\overline{W}}):= \simp(\sbTD(\hat{\overline{W}}))$ the \emph{restricted $\pm$-oriented symplectic Tits building of $\hat{\overline{W}}$}.
\end{definition}

\begin{lemma} \label{pmResSympPosetLinks}
Suppose $(\overline{U}, \pm \beta), (\overline{U}', \pm \beta') \in \sbTD(\hat{\overline{W}})$ such that $\overline{U} \subsetneq \overline{U}'$. 
\begin{itemize}[nosep, topsep=-0.5cm]
\item The lower link of $(\overline{U}, \pm \beta)$ is isomorphic to $\cbTD(\overline{U})$;
\item if $\overline{U}^\perp \cap \hat{\overline{W}} = \overline{U}^\perp$, then the upper link of $(\overline{U}, \pm \beta)$ is isomorphic to $\sbTD(\overline{U}^\perp/\overline{U})$;
\item if $\overline{U}^\perp \cap \hat{\overline{W}} \neq \overline{U}^\perp$, then the upper link of $(\overline{U}, \pm \beta)$ is isomorphic to $\sbTD((\overline{U}^\perp \cap \hat{\overline{W}})/\overline{U})$;
\item the interval from $(\overline{U}, \pm \beta)$ to $(\overline{U}', \pm \beta')$ is isomorphic to $\cbTD(\overline{U}'/\overline{U})$.
\end{itemize} 
\end{lemma}

We compute the links of $\sbTD(\hat{\overline{W}})$ using the same arguments presented in the proof of Lemma \ref{pmSympPosetLinks}. The computation of the upper link of $U$ for the case of $\overline{U}^\perp \cap \hat{\overline{W}} \neq \overline{U}^\perp$ follows from Lemma \ref{ResSympPosetLinks} instead of Lemma \ref{SympPosetLinks}. 

\begin{definition}
Let $W$ be a f.g.~symplectic $\mathbb{Z}$-module, let $\overline{W}$ be a symplectic $\Field_p$-vector space, let $\hat{W} \subsetneq W$ be a restricted summand, let $\hat{\overline{W}} \subsetneq \overline{W}$ be a restricted subspace, let $\pi\colon W \rightarrow \overline{W}$ be a surjection such that $\pi(\hat{W}) = \hat{\overline{W}}$, and let $(\overline{V}, \pm \alpha) \in \sbTD(\hat{\overline{W}})$. We denote by $\sbT(\hat{W} \vert (\overline{V}, \pm \alpha))$ the subposet of $\sbT(W \vert (\overline{V}, \pm \alpha))$ whose elements are isotropic summands $U \subseteq \hat{W}$. We call $\sbT(\hat{W} \vert (\overline{V}, \pm \alpha))$ and $\sT(\hat{W} \vert (\overline{V}, \pm \alpha)):= \simp(\sbT(\hat{W} \vert (\overline{V}, \pm \alpha)))$ the \emph{restricted symplectic Tits buildings with a $\pm$-congruence condition}. 
\end{definition} 

\begin{lemma} \label{linkposetrestricted}
Let $\pi\colon W \rightarrow \overline{W}$ be a surjection such that $\pi(\hat{W}) = \hat{\overline{W}}$. Suppose $U, U' \in \sbT(\hat{W} \vert (\overline{V}, \pm \alpha))$ such that $U \subset U'$. 
\begin{itemize}[nosep, topsep=-0.5cm]
\item The lower link of $U$ is isomorphic to $\cbT(U)$;
\item if $U^\perp \cap \hat{W} = U^\perp$, the upper link of $U$ is isomorphic to $\sbT(U^\perp/U \vert (\overline{V}/\overline{U}, \pm \eta))$, $(\overline{V}/\overline{U}, \pm \eta) \in \sbTD(\overline{U}^\perp/\overline{U})$ and $\pi_*\colon U^\perp/U \rightarrow \overline{U}^\perp/\overline{U}$ is a surjection induced by $\pi$;
\item if $U^\perp \cap \hat{W} \neq U^\perp$, the upper link of $U$ is isomorphic to $\sbT((U^\perp \cap \hat{W})/U \vert (\overline{V}/\overline{U}, \pm \eta))$, where $(\overline{V}/\overline{U}, \pm \eta) \in \sbTD((\overline{U}^\perp \cap \hat{\overline{W}})/\overline{U})$ and $\pi_*\colon (U^\perp \cap \hat{W})/U \rightarrow (\overline{U}^\perp \cap \hat{\overline{W}})/\overline{U}$ is a surjection induced by $\pi$;
\item the interval from $U$ to $U'$ is isomorphic to $\cbT(U'/U)$.
\end{itemize} 
\end{lemma}

\begin{proof}
We can compute the links of $\sbT(\hat{W} \vert (\overline{V}, \pm \alpha))$ using the arguments presented in the proof of Lemma \ref{SympCongruencePosetLinks}. The argument for the upper link of $U$ when $U^\perp \cap \hat{W} \neq U^\perp$ will use Lemmas \ref{ResSympPosetLinks} and \ref{pmResSympPosetLinks} to construct an analogous Diagram \ref{computingLinksDiagram}. In the case that the diagram produces a different $\pm$-orientation from $\pm \eta$, we use Proposition \ref{transitivityProved} to find symplectic automorphisms that induce poset isomorphisms as in Diagram \ref{commutativeDiagram2}. 
\end{proof}


\section{A simplicial model for $\lvert \sT_{2n}(\mathbb{Q}) \rvert/\Gamma_{2n}^\omega(p)$} \label{simplicialmodelisquotient}

Proposition \ref{SympMPP21Lem3.10} implies that $\sT_{2n}(\mathbb{Q}) \cong \sT_{2n}(\mathbb{Z})$. Hence, the topological quotients $\lvert \sT_{2n}(\mathbb{Q}) \rvert/\Gamma_{2n}^\omega(p)$ and $\lvert \sT_{2n}(\mathbb{Z}) \rvert/\Gamma_{2n}^\omega(p)$ are isomorphic. In this section we prove that $\sT_{2n}(\mathbb{Z})$ is a regular $\Gamma_{2n}^\omega(p)$-complex in the sense of Definition \ref{regularcomplex}. This means that the topological quotient $ \lvert \sT_{2n}(\mathbb{Z}) \rvert/\Gamma_{2n}^\omega(p)$ is homeomorphic to the geometric realization of the simplicial quotient $\sT_{2n}(\mathbb{Z})/\Gamma_{2n}^\omega(p)$, by Bredon's Theorem \ref{Bredon}. We finish this section by constructing a simplicial isomorphism between the simplicial quotient $\sT_{2n}(\mathbb{Z})/\Gamma_{2n}^\omega(p)$ and $\sTD_{2n}(\Field_p)$, providing a simplicial model for $\lvert \sT_{2n}(\mathbb{Q}) \rvert/\Gamma_{2n}^\omega(p)$. \\
The action of $\Sp_{2n}(\mathbb{Z})$ on a simplex $\sigma \in \sT_{2n}(\mathbb{Z})$ comes from an action on the summands of the isotropic flag $\sigma$ is associated to. That is, if $g \in \Sp_{2n}(\mathbb{Z})$ and $\sigma$ corresponds to the isotropic flag 
\begin{equation*}
0 \subsetneq U_0 \subsetneq \dots \subsetneq U_k \subsetneq \mathbb{Z}^{2n},
\end{equation*}
then $g \cdot \sigma$ corresponds to the isotropic flag
\begin{equation*}
0 \subsetneq g \cdot U_0 \subsetneq \dots \subsetneq g \cdot U_k \subsetneq \mathbb{Z}^{2n}. 
\end{equation*}
$\Sp_{2n}(\mathbb{Z})$ acts transitively on the set of isotropic summands of a fixed rank, so $g \cdot \sigma$ is also a $k$-simplex. $\Gamma_{2n}^\omega(p)$ acts on $\sT_{2n}(\mathbb{Z})$ via the restriction to $\Gamma_{2n}^\omega(p)$, hence $\sT_{2n}(\mathbb{Z})$ is a $\Gamma_{2n}^\omega(p)$-complex. 

\begin{theorem} \textup{\cite[Theorem VII.21]{N72}} \label{surjectivesympelctic}
Let $n \geq 1$ and let $p$ be a prime integer and $\pi\colon \Sp_{2n}(\mathbb{Z}) \rightarrow \Sp_{2n}(\Field_p)$ be the reduction mod-$p$ map. Then $\pi$ is a surjection. 
\end{theorem}

\begin{definition}
We say that a basis $\{\vec{v}_1, \dots, \vec{v}_n\}$ for a  f.g.~free $R$-module $V$ is \emph{compatible with a $\pm$-orientation $\pm \alpha$ on $V$} if $\pm \alpha = \pm (\vec{v}_1 \wedge \dots \wedge \vec{v}_n)$.
\end{definition}

\begin{lemma} \label{ElementOfCongruenceSubgroup}
Let $n \geq 1$, let $p$ be a prime integer, let $\overline{V} \in \Field_p^{2n}$ be an isotropic subspace equipped with $\pm \alpha$ and let $\pi\colon \mathbb{Z}^{2n} \rightarrow \Field_p^{2n}$ be the reduction mod-$p$ map. $\Gamma_{2n}^\omega(p)$ acts transitively on the set of isotropic summands $V \subset \mathbb{Z}^{2n}$ that satisfy $\pi(V) = (\overline{V}, \pm \alpha)$.
\end{lemma}  

\begin{proof}
Suppose $V, V' \subset \mathbb{Z}^{2n}$ are isotropic summands of rank $m$ that satisfy $\pi(V) = (\overline{V}, \pm \alpha)$. Let $\{\vec{e}_1, \dots, \vec{e}_m\}_\mathbb{Z}$ be a basis for a summand $E \subset \mathbb{Z}^{2n}$ and let $\{\vec{e}_1, \dots, \vec{e}_m\}_{\Field_p}$ be a basis for $\overline{E} \subset \Field_p^{2n}$ equipped with $\pm \epsilon = \pm (\vec{e}_1 \wedge \dots \wedge \vec{e}_m)$. Choose an isotropic basis $\{\vec{v}_1, \dots, \vec{v}_m\}_{\Field_p}$ for $\overline{V}$ that is compatible with $\pm \alpha$, and extend it to symplectic basis $\{\vec{v}_1, \vec{w}_1, \dots, \vec{v}_n, \vec{w}_n\}_{\Field_p}$ of $\Field_p^{2n}$ using Lemma \ref{SymplecticBases}. Let $\overline{g} \in \Sp_{2n}(\Field_p)$ be the symplectic automorphism that sends $\vec{e}_i$ to $\vec{v}_i$ and $\vec{f}_i$ to $\vec{w}_i$ for $i \in \{1, \dots, n\}$. By Newman's Theorem \ref{surjectivesympelctic}, there exists $g, h \in \Sp_{2n}(\mathbb{Z})$ such that $\pi(g) = \pi(h) = \overline{g}$. Since $\pi(E) = (\overline{E}, \pm \epsilon)$, $g \cdot E = V$ and $h \cdot E = V'$. Now, $hg^{-1} \cdot V = V'$ and 
\begin{equation*}
\pi(hg^{-1}) = \pi(h)\pi(g^{-1}) = \overline{g} \overline{g}^{-1},
\end{equation*}
so $hg^{-1} \in \Gamma_{2n}^\omega(p)$ as desired. 
\end{proof}

\begin{lemma} \textup{\cite[Lemma 3.19]{MiPaP21}}\label{MPP21Lem3.19}
Let $V$ be a f.g.~free $\mathbb{Z}$-module of rank $n$, let $\overline{V}$ be an $\Field_p$-vector space of dimension $n$ and let $\pi\colon V \rightarrow (\overline{V}, \alpha_\pi)$ be a surjection. Suppose $\{\vec{x}_1, \dots, \vec{x}_n\}$ is a basis for $\overline{V}$ compatible with $\pm \alpha$ and $\{\vec{X}_1, \dots, \vec{X}_m\}$ is a partial basis for $V$ such that $\pi(\vec{X}_i) = \vec{x}_i$ for $1 \leq i \leq m$. We can then complete our partial $\mathbb{Z}$-basis to a basis $\{\vec{X}_1, \dots, \vec{X}_n\}$ that spans $V$ and satisfies $\pi(\vec{X}_i) = \vec{x}_i$ for $1 \leq i \leq n$.
\end{lemma}

The following lemma is a corollary of Miller--Patzt--Putman's Lemma \ref{MPP21Lem3.19}.

\begin{lemma}\label{CompleteIsoBasis}
Let $V$ be an isotropic $\mathbb{Z}$-module of rank $n$, let $\overline{V}$ be an isotropic $\Field_p$-vector space of dimension $n$ and let $\pi\colon V \rightarrow (\overline{V}, \pm \alpha_\pi)$ be a surjection. Suppose $\{\vec{x}_1, \dots, \vec{x}_n\}$ is an isotropic basis for $\overline{V}$ compatible with $\pm \alpha$ and $\{\vec{X}_1, \dots, \vec{X}_m\}$ is an isotropic partial basis for $V$ such that $\pi(\vec{X}_i) = \vec{x}_i$ for $1 \leq i \leq m$. We can then complete our partial $\mathbb{Z}$-basis to a basis $\{\vec{X}_1, \dots, \vec{X}_n\}$ that spans $V$ and satisfies $\pi(\vec{X}_i) = \vec{x}_i$ for $1 \leq i \leq n$.
\end{lemma}

\begin{lemma} \label{regularityOfTitsBuildings}
For $n \geq 1$, $\sT_{2n}(\mathbb{Z})$ and $\sT_{2n}(\mathbb{Q})$ are regular $\Gamma_{2n}^\omega(p)$-complexes.
\end{lemma}

\begin{proof}
Proposition \ref{SympMPP21Lem3.10} says that $\sT_{2n}(\mathbb{Q}) \cong \sT_{2n}(\mathbb{Z})$, so it suffices to check that $\sT_{2n}(\mathbb{Z})$ is a regular $\Gamma_{2n}^\omega(p)$-complex. Suppose $\sigma \in \sT_{2n}(\mathbb{Z})$ is a $k$-simplex corresponding to the isotropic flag 
\begin{equation} \label{regular2}
0 \subsetneq V_0 \subsetneq \dots \subsetneq V_k \subsetneq \mathbb{Z}^{2n},
\end{equation}
and $\tau \in \sT_{2n}(\mathbb{Z})$ is a $k$-simplex associated to the isotropic flag
\begin{equation} \label{regular1}
0 \subsetneq V_0' \subsetneq \dots \subsetneq V_k' \subsetneq \mathbb{Z}^{2n},
\end{equation}
such that $V_i' = g_i \cdot V_i$ for $g_i \in \Gamma_{2n}^\omega(p)$, $i \in \{0, \dots, k\}$. Let $\pi\colon\mathbb{Z}^{2n} \rightarrow \Field_p^{2n}$ be the reduction mod-$p$ map, let $m_i = \Dim V_i$, and let $\pi(V_i) = (\overline{V}_i, \pm (\alpha_i)_\pi)$ for $i \in \{0, \dots,  k\}$. Since $g_i \in \Gamma_{2n}^\omega(p)$, then $\pi(V_i') = (\overline{V}_i, \pm (\alpha_i)_\pi)$ for $i \in \{0, \dots, k\}$. We pick an isotropic basis $\{\vec{x}_1, \dots, \vec{x}_{m_k}\}$ for $\overline{V}_k$ that is compatible with $\pm (\alpha_k)_\pi$ for which $\{\vec{x}_1, \dots, \vec{x}_{m_i}\}$ is an isotropic basis for $\overline{V}_i$ that is compatible with $\pm (\alpha_i)_\pi$ for $i \in \{0, \dots, k - 1\}$. Using Lemma \ref{CompleteIsoBasis} iteratively, we lift $\{\vec{x}_1, \dots, \vec{x}_{m_i}\}$ to an isotropic basis $\{\vec{X}_1, \dots, \vec{X}_{m_i}\}$ for $V_i$ such that $\pi(\vec{X}_j) = \vec{x}_j$ for $j \in \{1, \dots, m_i\}$, and to an isotropic basis $\{\vec{X}_1', \dots, \vec{X}_{m_i}'\}$ for $V_i'$ such that $\pi(\vec{X}_j') = \vec{x}_j$ for $j \in \{1, \dots, m_i\}$. By Lemma \ref{ElementOfCongruenceSubgroup}, there exists $h \in \Gamma_{2n}^\omega(p)$ such that $h \cdot \vec{X}_i = \vec{X}_i'$ for $i \in \{1, \dots, m_k\}$. We conclude that $h \cdot V_i = V_i'$ for $i \in \{0, \dots, k\}$. Thus, $\Gamma_{2n}^\omega(p)$ is a regular group action on $\sT_{2n}(\mathbb{Z})$.
\end{proof}

The next corollary follows directly from Lemma \ref{regularityOfTitsBuildings} and Bredon's Theorem \ref{Bredon}. 

\begin{corollary}
For $n \geq 1$ and $p$ a prime,
\begin{align*}
\lvert \sT_{2n}(\mathbb{Q}) \rvert/\Gamma_{2n}^\omega(p) & \cong \lvert \sT_{2n}(\mathbb{Q})/\Gamma_{2n}^\omega(p) \rvert, \\
\lvert \sT_{2n}(\mathbb{Z}) \rvert/\Gamma_{2n}^\omega(p) & \cong \lvert \sT_{2n}(\mathbb{Z})/\Gamma_{2n}^\omega(p) \rvert.
\end{align*}
\end{corollary}

For the rest of this article, we do not distinguish between the topological quotient and the simplicial quotient. 

\begin{definition}
We denote by 
\begin{equation*}
\Gamma_n(p) := \ker(\SL_n(\mathbb{Z}) \xrightarrow{\pi} \SL_n(\mathbb{Z}/p\mathbb{Z}))
\end{equation*}
the \emph{principal congruence subgroup of level-$p$ of $\SL_n(\mathbb{Z})$}, where $p$ is a prime integer and $\pi$ is the reduction mod-$p$ map.
\end{definition}

\begin{remark}
$\cT_n(\mathbb{Q})$ and $\cT_n(\mathbb{Z})$ are regular $\Gamma_n(p)$-complexes. By Bredon's Theorem \ref{Bredon}, 
\begin{align*}
\lvert \cT_{2n}(\mathbb{Q}) \rvert/\Gamma_{2n}^\omega(p) & \cong \lvert \cT_{2n}(\mathbb{Q})/\Gamma_{2n}^\omega(p) \rvert, \\
\lvert \cT_{2n}(\mathbb{Z}) \rvert/\Gamma_{2n}^\omega(p) & \cong \lvert \cT_{2n}(\mathbb{Z})/\Gamma_{2n}^\omega(p) \rvert.
\end{align*}
As with the symplectic Tits building, we do not distinguish between the topological quotient and the simplicial quotient. 
\end{remark}

\begin{proposition} \textup{\cite[Proposition 3.16]{MiPaP21}} \label{isomorphismlinear}
Let $n \geq 1$ and $p$ be prime. Then $\cT_n(\mathbb{Q})/\Gamma_n(p) \cong \cTD_n(\Field_p)$. 
\end{proposition}  

\begin{proposition}\label{aSimplicialModel}
Let $n \geq 1$ and $p$ be prime. Then $\sT_{2n}(\mathbb{Q})/\Gamma_{2n}^\omega(p) \cong \sTD_{2n}(\Field_p)$. 
\end{proposition}  

\begin{proof}
Proposition \ref{SympMPP21Lem3.10} says that $\sT_{2n}(\mathbb{Q}) \cong \sT_{2n}(\mathbb{Z})$, so $\sT_{2n}(\mathbb{Q})/\Gamma_{2n}^\omega(p) \cong \sT_{2n}(\mathbb{Z})/\Gamma_{2n}^\omega(p)$. We construct a simplicial isomorphism between $\sT_{2n}(\mathbb{Z})/\Gamma_{2n}^\omega(p)$ and $\sTD_{2n}(\Field_p)$.\\
Let $\pi\colon \mathbb{Z}^{2n} \rightarrow \Field_p^{2n}$ be the reduction mod-$p$ map and let $\varphi\colon \sT_{2n}(\mathbb{Z}) \rightarrow \sTD_{2n}(\Field_p)$ be the simplicial map induced by $\pi$. $\varphi$ takes simplices that correspond to flags of isotropic summands to simplices that correspond to flags of isotropic subspaces equipped with a $\pm$-orientation induced by $\pi$. $\varphi$ is $\Gamma_{2n}^\omega(p)$-invariant, so it induces a well-defined map
\begin{equation*}
\varphi_*\colon \sT_{2n}(\mathbb{Z})/\Gamma_{2n}^\omega(p) \rightarrow \sTD_{2n}(\Field_p).
\end{equation*}
We check that $\varphi_*$ is an isomorphism between the set of simplices of $\sT_{2n}(\mathbb{Z})/\Gamma_{2n}^\omega(p)$ and the set of simplices of $\sTD_{2n}(\Field_p)$.\\
We first prove that $\varphi_*$ is surjective on vertices. Suppose $\overline{s} \in \sTD_{2n}(\Field_p)$ corresponds to the isotropic flag $0 \subsetneq \overline{V} \subsetneq \Field_p^{2n}$ for some dimension-$m$ isotropic subspace $\overline{V}$ equipped with $\pm \alpha$. Pick an isotropic basis $\{\vec{x}_1, \dots, \vec{x}_m\}$ for $\overline{V}$ that is compatible with $\pm \alpha$. Using Lemma \ref{CompleteIsoBasis}, lift $\{\vec{x}_1, \dots, \vec{x}_m\}$ to an isotropic basis $\{\vec{X}_1, \dots, \vec{X}_m\}$ that satisfies $\pi(\vec{X}_i) = \vec{x}_i$ for $i \in \{1, \dots, m\}$. Let $V$ be the span of $\{\vec{X}_1, \dots, \vec{X}_m\}$ and let $s \in \sT_{2n}(\mathbb{Z})$ correspond to the isotropic flag $0 \subsetneq V \subsetneq \mathbb{Z}^{2n}$. Since $\varphi(s) = \overline{s}$, it follows that $\varphi_*$ is surjective on vertices. \\
We now establish the surjectivity of $\varphi_*$ on the set of $k$-simplices of $\sT_{2n}(\mathbb{Z})/\Gamma_{2n}^\omega(p)$ for $k \geq 1$. Suppose $\overline{\sigma} \in \sTD_{2n}(\Field_p)$ corresponds to the isotropic flag 
\begin{equation}
0 \subsetneq \overline{V}_0 \subsetneq \dots \subsetneq \overline{V}_k \subsetneq \Field_p^{2n},
\end{equation}
where $\overline{V}_i$ is of dimension $m_i$ and is equipped with $\pm \alpha_i$. But 
\begin{equation}
0 \subsetneq \overline{V}_0 \subsetneq \dots \subsetneq \overline{V}_{k - 1} \subsetneq \overline{V}_k
\end{equation}
is a flag that corresponds to a $(k - 1)$-simplex $\overline{\tau} \in \simp(\sbTD_{2n}(\Field_p)_{< (\overline{V}_k, \pm \alpha_k)})$. Lemma \ref{pmSympPosetLinks} states that $\simp(\sbTD_{2n}(\Field_p)_{< (\overline{V}_k, \pm \alpha_k)})$ is isomorphic to $\cTD(\overline{V}_k)$. We can find an isotropic $V_k \subset \mathbb{Z}^{2n}$ such that $\pi(V_k) = (\overline{V}_k, \pm \alpha_k)$ using an isotropic basis for $\overline{V}_k$ compatible with $\pm \alpha_k$ and Lemma \ref{CompleteIsoBasis}. By Miller--Patzt--Putman's Propositions \ref{MPP21Lem3.10} and \ref{isomorphismlinear}, there exists a simplicial isomorphism $\gamma\colon \cT(V_k)/\Gamma_{m_k}(p) \rightarrow \cTD(\overline{V}_k)$. Let $\iota\colon \cT(V_k) \rightarrow \cT(V_k)/\Gamma_{m_k}(p)$ be the quotient map. We can lift $\overline{\tau} \in \cTD(\overline{V}_k)$ using $(\iota \circ \gamma)$ to some $\tau \in \cT(V_k)$ such that $\varphi \vert_{V_k}(\tau) = \overline{\tau}$. Appending $V_k$ to the isotropic flag of $\tau$ yields an isotropic flag corresponding to a $k$-simplex $\sigma \in \sT_{2n}(\mathbb{Z})$ such that $\varphi(\sigma) = \overline{\sigma}$. Therefore $\varphi_*$ is surjective on $k$-simplices. \\
We prove that $\varphi_*$ is injective on vertices. Let $s, s' \in \sT_{2n}(\mathbb{Z})$ be vertices such that $\varphi(s) = \varphi(s') = \overline{s} \in \sTD_{2n}(\Field_p)$. Suppose $s, s', \overline{s}$ correspond to the isotropic flags
\begin{align*}
    0 & \subsetneq V \subsetneq \mathbb{Z}^{2n}, \\
    0 & \subsetneq V' \subsetneq \mathbb{Z}^{2n}, \\
    0 & \subsetneq \overline{V} \subsetneq \Field_p^{2n},
\end{align*}
respectively, where $\pi(V)= \pi(V') = (\overline{V}, \pm \alpha_\pi)$. By Lemma \ref{ElementOfCongruenceSubgroup}, there exists $g \in \Gamma_{2n}^\omega(p)$ such that $g \cdot V = V'$. So $g \cdot s = s'$. Therefore $\varphi$ is injective up to a $\Gamma_{2n}^\omega(p)$-action on vertices, which implies that $\varphi_*$ is injective on vertices.   \\
We now establish the injectivity of $\varphi_*$ on the set of $k$-simplices of $\sT_{2n}(\mathbb{Z})/\Gamma_{2n}^\omega(p)$ for $k \geq 1$. Suppose that $\sigma, \sigma' \in \sT_{2n}(\mathbb{Z})$ are $k$-simplices for which $\varphi(\sigma) = \varphi(\sigma') = \overline{\sigma} \in \sTD_{2n}(\Field_p)$. Let $\sigma, \sigma', \overline{\sigma}$ correspond to the isotropic flags
\begin{align*}
    0 & \subsetneq V_0 \subsetneq \dots \subsetneq V_k \subsetneq \mathbb{Z}^{2n},\\
    0 & \subsetneq V_0' \subsetneq \dots \subsetneq V_k' \subsetneq \mathbb{Z}^{2n}, \\
    0 & \subsetneq \overline{V}_0 \subsetneq \dots \subsetneq \overline{V}_k \subsetneq \Field_p^{2n}, 
\end{align*}
respectively, where $\overline{V}_i$ is equipped with $\pm \alpha_i$ for $i \in \{0, \dots, k\}$. $V_i$ and $V_i'$ are isotropic summands that satisfy $\pi(V_i) = \pi(V_i') = (\overline{V}_i, \pm (\alpha_i)_\pi)$ for $i \in \{0, \dots, k\}$. By Lemma \ref{ElementOfCongruenceSubgroup}, there exists $g_i \in \Gamma_{2n}^\omega(p)$ such that $g_i \cdot V_i = V_i'$. Since $\sT_{2n}(\mathbb{Z})$ is a regular $\Gamma_{2n}^\omega(p)$-complex by Lemma \ref{regularityOfTitsBuildings}, there exists $g \in \Gamma_{2n}^\omega(p)$ such that $g \cdot V_i = g_i \cdot V_i = V_i'$ for $i \in \{0, \dots, k\}$. So, $g \cdot \sigma = \sigma'$. Hence, $\varphi$ is injective up to a $\Gamma_{2n}^\omega(p)$-action on $k$-simplices. We conclude that $\varphi_*$ is injective on $k$-simplices. Hence, there is a one-to-one correspondence between the set of simplices of $\sT_{2n}(\mathbb{Z})/\Gamma_{2n}^\omega(p)$ and the set of simplices of $\sTD_{2n}(\Field_p)$.
\end{proof}


\section{Partial basis complexes}

Let $R$ be a PID. The \emph{partial basis complex} is a simplicial complex whose $(k - 1)$-simplices correspond to partial bases of size $k$.

\begin{definition}
Let $V$ be a f.g.~free $R$-module, $W$ a f.g.~symplectic $R$-module and $\hat{W}$ a restricted summand of $W$. 
    \begin{itemize}[nosep, topsep=-0.5cm]
        \item Let $\BP(V)$ denote the simplicial complex whose $(k - 1)$-simplices correspond to unordered partial $R$-bases of size $k$ of $V$.
        \item Let $\IP(W)$ denote the simplicial complex whose $(k - 1)$-simplices correspond to unordered partial isotropic $R$-bases of size $k$ of $W$.
        \item Let $\IP(\hat{W})$ denote the full simplicial sub-complex of $\IP(W)$ whose $(k - 1)$-simplices correspond to unordered partial isotropic $R$-bases of size $k$ of $\hat{W}$.
    \end{itemize}
    We call $\BP(V)$ the \emph{partial basis complex of $V$}, $\IP(W)$ the \emph{isotropic partial basis complex of $W$}, and $\IP(\hat{W})$ the \emph{restricted isotropic partial basis complex of $\hat{W}$}.  
\end{definition}  

\begin{theorem} \textup{\cite[Corollary 4.5]{Maaz}} \label{CP17Thm4.2}
Suppose $V$ is a f.g.~free $\mathbb{Z}$-module of rank $n$. Then $\BP(V)$ is CM of dimension $(n - 1)$. 
\end{theorem}  

The simplices of $\IP(W)$ correspond to isotropic partial bases. These isotropic partial bases span isotropic summands of $\sbT(W)$. $\IP(\hat{W})$ is related to $\sbT(\hat{W})$ in the same manner. We introduce simplicial complexes that are related to $\sbT(W \vert (\overline{V}, \pm \alpha))$ and $\sbT(\hat{W} \vert (\overline{V}, \pm \alpha))$.

\begin{definition}
Let $W$ be a f.g.~symplectic $\mathbb{Z}$-module, let $\overline{W}$ be a symplectic $\Field_p$-vector space, let $\hat{W} \subsetneq W$ be a restricted summand, let $\hat{\overline{W}} \subsetneq \overline{W}$ be a restricted subspace, let $\pi\colon W \rightarrow \overline{W}$ be a surjection such that $\pi(\hat{W}) = \hat{\overline{W}}$, and let $(\overline{V}, \pm \alpha) \in \sbTD(\hat{\overline{W}})$.
\begin{itemize}[nosep, topsep=-0.5cm]
\item We denote by $\IP(W \vert (\overline{V}, \pm \alpha))$ the simplicial complex for which a $(k - 1)$-simplex corresponds to an isotropic partial $\mathbb{Z}$-basis $\{\vec{v}_1, \dots, \vec{v}_k\}$ for $W$ such that $\{\pi(\vec{v}_1), \dots, \pi(\vec{v}_k)\}$ spans a subspace $\overline{U} \subseteq \overline{V}$. We require that $\{\pi(\vec{v}_1), \dots, \pi(\vec{v}_k)\}$ be compatible with $\pm \alpha$ when the isotropic partial basis spans $\overline{V}$. We call this simplicial complex an \emph{isotropic partial basis complex of $W$ with a $\pm$-congruence condition}.
\item We denote by $\IP(\hat{W} \vert (\overline{V}, \pm \alpha))$ the sub-complex of $\IP(W \vert (\overline{V}, \pm \alpha))$ whose $(k - 1)$-simplices correspond to isotropic partial $\mathbb{Z}$-basis of size $k$ for $\hat{W}$. We call $\IP(\hat{W} \vert (\overline{V}, \pm \alpha))$ a \emph{restricted isotropic partial basis complex of $\hat{W}$ with a $\pm$-congruence condition}.
\end{itemize}
\end{definition}

\begin{definition}\label{LabelingSystem}
Suppose $X$ is a simplicial complex and $S$ is a non-empty set. Let $X \Span{S}$ denote the simplicial complex whose vertices are the pairs $(x, s)$ such that $x \in X$ is a vertex and $s \in S$. A $k$-simplex of $X \Span{S}$ is a set of vertices $\{(x_0, s_0), \dots, (x_k, s_k)\}$. We call $X \Span{S}$ a \emph{labeling system for $X$}.
\end{definition}

\begin{remark}\label{LabelingSystemRemark}
$X \Span{S}$ is a complete join complex over $X$. The projection map $\varphi\colon X \Span{S} \rightarrow X$ sends a simplex spanned by $\{(x_0, s_0), \dots, (x_k, s_k)\}$ to a simplex spanned by $\{x_0, \dots, x_k\}$.
\end{remark}

\begin{lemma}\textup{\cite[Lemma 3.2]{MiVdK}}\label{linkofIsotropic}
Let $W$ be a f.g.~symplectic $R$-module, let $\sigma \in \IP(W)$ be a $(k - 1)$-simplex and let $\Span{\sigma} = U$. Then the simplicial map
\begin{align*}
\varphi & \colon \lk_{\IP(W)}(\sigma) \xrightarrow{\cong} \IP(U^\perp/U) \Span{U}, \\
\varphi & \colon \{\vec{v}_1, \dots, \vec{v}_t\} \mapsto \{(\vec{w}_1 + U , \vec{s}_1), \dots, (\vec{w}_t + U, \vec{s}_t)\}, 
\end{align*} 
is a simplicial isomorphism, where $\vec{s}_i \in U$ and $\vec{v}_i = \vec{w}_i + \vec{s}_i$ for $i \in \{1, \dots, t\}$. 
\end{lemma}

The proof of Mirzaii-van der Kallen's Lemma \ref{linkofIsotropic} is based on Charney's \cite[Proof of Lemma 3.4]{Charney87} work, who constructed the explicit map. We extend Mirzaii-van der Kallen's Lemma \ref{linkofIsotropic} to the context of $\IP(W \vert (\overline{V}, \pm \alpha))$. 

\begin{lemma} \label{SympCongruenceSimpLinks}
Let $W$ be a f.g.~symplectic $\mathbb{Z}$-module, let $\overline{W}$ be a symplectic $\Field_p$-vector space, let $\pi\colon W \rightarrow (\overline{W}, \pm \nu_\pi)$ be a surjection, and let $(\overline{V}, \pm \alpha) \in \sbTD(\overline{W})$. Furthermore, let $\sigma \in \IP(W \vert (\overline{V}, \pm \alpha))$ be a $(k - 1)$-simplex, let $\Span{\sigma} = U$, and let $\pi(U) = \overline{U}$. Then
\begin{equation*}
\lk_{\IP(W \vert (\overline{V}, \pm \alpha))}(\sigma) \cong \IP(U^\perp/U \vert (\overline{V}/\overline{U}, \pm \eta)) \Span{U}.
\end{equation*} 
for some $\pm$-orientation $\pm \eta$ on $\overline{V}/\overline{U}$.
\end{lemma}

\begin{proof}
Set $\sigma = \{\vec{u}_1, \dots, \vec{u}_k\}_\mathbb{Z}$, $\overline{\sigma} = \{\pi(\vec{u}_1), \dots,\pi(\vec{u}_k)\}_{\Field_p}$, $\Span{\overline{\sigma}}_{\Field_p} = \overline{U}$, and set 
\begin{equation*}
\pm \beta_\pi := \pm \pi(\vec{u}_1) \wedge \dots \wedge \pi(\vec{u}_k).
\end{equation*}
By Lemma \ref{linkofIsotropic},
\begin{align*}
\varphi& \colon \lk_{\IP(W)}(\sigma) \xrightarrow{\cong} \IP(U^\perp/U) \Span{U}, \\
\varphi& \colon \{\vec{v}_1, \dots, \vec{v}_t\}_\mathbb{Z} \mapsto \{(\vec{w}_1 + U , \vec{s}_1), \dots, (\vec{w}_t + U, \vec{s}_t)\}_\mathbb{Z}
\end{align*}
Let $\psi\colon \overline{W} \rightarrow \overline{U}^\perp/\overline{U}$ be the projection map that sends isotropic subspaces of $\overline{W}$ to isotropic subspaces of $\overline{U}^\perp/\overline{U}$. Since the $\pm$-orientation on $\overline{U}$ is $\pm \beta_\pi$, Remark \ref{behaviorOfpmOrientations} says that there is a natural bijection between the set of $\pm$-orientations on $\overline{V}$ and the $\pm$-orientations on $\overline{V}/\overline{U}$. Set $\psi(\overline{V}, \pm \alpha) = (\overline{V}/\overline{U}, \pm \alpha_\psi)$. We check that
\begin{equation*}
\varphi(\lk_{\IP(W \vert (\overline{V}, \pm \alpha))}(\sigma)) = \IP(U^\perp/U \vert (\overline{V}/\overline{U}, \pm \alpha_\psi)) \Span{U}. 
\end{equation*}
Let $\tau = \{\vec{y}_1, \dots, \vec{y}_t\}$ be a $(t - 1)$-simplex in $\lk_{\IP(W \vert (\overline{V}, \pm \alpha))}(\sigma)$ and set $\varphi(\tau) = \{(\vec{x}_1 + U, \vec{s}_1), \dots, (\vec{x}_t + U, \vec{s}_t)\}$. Since $\vec{y}_i = \vec{x}_i + \vec{s}_i$ for $i \in \{1, \dots, t\}$, then  
\begin{equation*}
\Span{\pi(\vec{u}_1), \dots, \pi(\vec{u}_k), \pi(\vec{y}_1), \dots, \pi(\vec{y}_t)}_{\Field_p} \subseteq \overline{V},
\end{equation*}
if and only if 
\begin{equation*}
\Span{\pi(\vec{u}_1), \dots, \pi(\vec{u}_k), \pi(\vec{x}_1), \dots, \pi(\vec{x}_t)}_{\Field_p} \subseteq \overline{V},
\end{equation*}
if and only if 
\begin{equation*}
\Span{\pi(\vec{x}_1 + U), \dots, \pi(\vec{x}_t + U)}_{\Field_p} \subseteq \overline{V}/\overline{U}. 
\end{equation*}
Since the exterior product is a skew-symmetric bilinear form, 
\begin{align*}
\pi(\vec{u}_1) \wedge \dots \wedge \pi(\vec{u}_k) \wedge \pi(\vec{y}_i) & = \pi(\vec{u}_1) \wedge \dots \wedge \pi(\vec{u}_k) \wedge \pi(\vec{x}_i) + \pi(\vec{u}_1) \wedge \dots \wedge \pi(\vec{u}_k) \wedge \pi(\vec{s}_i), \\
									& = \pi(\vec{u}_1) \wedge \dots \wedge \pi(\vec{u}_k) \wedge \pi(\vec{x}_i),
\end{align*}
for $i \in \{1, \dots, t\}$. When $\dim \overline{V} = k + t$, 
\begin{equation*}
\pi(\vec{u}_1) \wedge \dots \wedge \pi(\vec{u}_k) \wedge \pi(\vec{y}_1) \wedge \dots \wedge \pi(\vec{y}_t) = \pm \alpha,
\end{equation*}
if and only if 
\begin{equation*}
\pi(\vec{u}_1) \wedge \dots \wedge \pi(\vec{u}_k) \wedge \pi(\vec{x}_1) \wedge \dots \wedge \pi(\vec{x}_t) = \pm \alpha,
\end{equation*}
if and only if 
\begin{equation*}
\pi(\vec{x}_1) \wedge \dots \wedge \pi(\vec{x}_t) = \pm \gamma,
\end{equation*}
for some $\pm$-orientation $\pm \gamma$ on $\Span{\pi(\vec{x}_1), \dots, \pi(\vec{x}_t)}_{\Field_p}$ that satisfies $\pm \beta_\pi \wedge \gamma = \pm \alpha$. By Remark \ref{behaviorOfpmOrientations}, $\pm \gamma$ is in a natural bijection with $\pm \alpha_\psi$. Thus, 
\begin{equation*}
\pi(\vec{x}_1) \wedge \dots \wedge \pi(\vec{x}_t) = \pm \gamma,
\end{equation*}
if and only if 
\begin{equation*}
\pi(\vec{x}_1 + U) \wedge \dots \wedge \pi(\vec{x}_t + U) = \pm \alpha_\psi. 
\end{equation*}
We conclude that $\tau \in \lk_{\IP(W \vert (\overline{V}, \pm \alpha))}(\sigma)$ if and only if $\varphi(\tau) \in \IP(U^\perp/U \vert (\overline{V}/\overline{U}, \pm \alpha_\psi)) \Span{U}$. Since $\varphi$ is a simplicial isomorphism, $\gamma \in \IP(U^\perp/U \vert (\overline{V}/\overline{U}, \pm \alpha_\psi)) \Span{U}$ if and only if $\varphi^{-1}(\gamma) \in \lk_{\IP(W \vert (\overline{V}, \pm \alpha))}(\sigma)$. Thus, 
\begin{equation*}
\varphi(\lk_{\IP(W \vert (\overline{V}, \pm \alpha))}(\sigma)) = \IP(U^\perp/U \vert (\overline{V}/\overline{U}, \pm \alpha_\psi)) \Span{U}. 
\end{equation*}
Therefore,
\begin{equation*}
\varphi \vert_{(\lk_{\IP(W \vert (\overline{V}, \pm \alpha))}(\sigma))} \colon \lk_{\IP(W \vert (\overline{V}, \pm \alpha))}(\sigma) \xrightarrow{\cong} \IP(U^\perp/U \vert (\overline{V}/\overline{U}, \pm \alpha_\psi)) \Span{U},
\end{equation*}
If $\pm \alpha_\psi \neq \pm \eta$, then we consider a symplectic automorphism of $\Field_p$-vector spaces $\overline{U}^\perp/\overline{U} \rightarrow \overline{U}^\perp/\overline{U}$ that sends $ (\overline{V}/\overline{U}, \pm \alpha_\psi)$ to $ (\overline{V}/\overline{U}, \pm \eta)$. We lift this automorphism using Newman's Theorem \ref{surjectivesympelctic} to a symplectic automorphism of $\mathbb{Z}$-modules $U^\perp/U \rightarrow U^\perp/U$. This symplectic automorphism induces an isomorphism 
\begin{equation*}
\IP(U^\perp/U \vert (\overline{V}/\overline{U}, \pm \alpha_\psi)) \Span{U} \cong \IP(U^\perp/U \vert (\overline{V}/\overline{U}, \pm \eta)) \Span{U}.
\end{equation*} 
\end{proof} 


\section{Symplectic Tits buildings of $\mathbb{Z}$ with $\pm$-congruence conditions are highly connected} \label{titsconghighconnectivity}

This section is devoted to proving the following proposition using induction. 

\begin{proposition} \label{anotherInduction2}
Let $W$ be a f.g.~symplectic $\mathbb{Z}$-module, let $\overline{W}$ be a symplectic $\Field_p$-vector space, let $\pi\colon W \rightarrow \overline{W}$ be a surjection, and let $(\overline{V}, \pm \alpha) \in \sbTD(\overline{W})$. If $W$ is of rank $2m$ and $\overline{V}$ is of dimension $n$ for $m \geq n > 0$, then $\sbT(W \vert (\overline{V}, \pm \alpha))$ is CM of dimension $(n - 1)$. 
\end{proposition}  

Our proof for Proposition \ref{anotherInduction2} is inspired by Br\"uck--Sroka's \cite[Proof of Theorem 3.5]{BruckSroka24} work. Our proof strategy involves using an induction argument on $n$ to prove that the following four constructions are CM of dimension $(n - 1)$: $\sbT(\hat{W} \vert (\overline{V}, \pm \alpha))$, $\IP(\hat{W} \vert (\overline{V}, \pm \alpha))$, and $\IP(W \vert (\overline{V}, \pm \alpha))$, and $\sbT(W \vert (\overline{V}, \pm \alpha))$. When we consider one of these constructions, we use the high connectivity of the previous construction to prove that the construction is highly connected. Thus, we prove they are highly connected in the order we wrote them down. In the next subsection we present preliminary results needed to prove Proposition \ref{anotherInduction2}.

\subsection{Preliminaries}

We computed the links of $\sbT(\hat{W} \vert (\overline{V}, \pm \alpha))$, $\IP(W \vert (\overline{V}, \pm \alpha))$, and $\sbT(W \vert (\overline{V}, \pm \alpha))$ already: Lemmas \ref{linkposetrestricted}, \ref{SympCongruenceSimpLinks}, and \ref{SympCongruencePosetLinks}, respectively. The links are isomorphic to smaller dimensional constructions. Our inductive argument implies that these links are sufficiently connected. We omit the link computations of $\IP(\hat{W} \vert (\overline{V}, \pm \alpha))$ since they are not necessary for our proof. \\
Let $\posety(X)$ denote \emph{the poset associated to a simplicial complex $X$}. The elements of $\posety(X)$ are $k$-simplices and the order of $\posety(X)$ is inclusion of simplices. $\posety(X)$ and $X$ have the same homotopy type. The posets associated to $\IP(W \vert (\overline{V}, \pm \alpha))$ and $\IP(\hat{W} \vert(\overline{V}, \pm \alpha))$ are posets whose elements are unordered partial bases and whose order is proper inclusion of unordered partial bases as sets. \\
We can construct poset maps between $\posety(\IP(W \vert (\overline{V}, \pm \alpha)))$ and $\sbT(W \vert (\overline{V}, \pm \alpha))$, and between $\posety(\IP(\hat{W} \vert (\overline{V}, \pm \alpha)))$ and $\sbT(\hat{W} \vert (\overline{V}, \pm \alpha))$. We invoke Quillen's Theorem \ref{Q78Cor9.7} to establish that $\posety(\IP(\hat{W} \vert (\overline{V}, \pm \alpha)))$ is CM. We use van der Kallen--Looijenga's Theorem \ref{VdKaL11Cor2.2v2} to argue that $\sbT(W \vert (\overline{V}, \pm \alpha))$ is highly connected. \\
The connectivity arguments for $\sbT(\hat{W} \vert (\overline{V}, \pm \alpha))$ and $\IP(W \vert (\overline{V}, \pm \alpha))$ are in terms of simplicial complexes instead of posets. We use the following proposition due to Galatius--Randall-Williams \cite{GalatiusRandal-Williams14} to argue that $\sT(\hat{W} \vert (\overline{V}, \pm \alpha))$ is $(n - 2)$-connected.

\begin{proposition} \textup{\cite[Proposition 2.5]{GalatiusRandal-Williams14}} \label{GalatiusRandal-Williams14}
Let $X$ be a simplicial complex and $Y \subset X$ a full sub-complex. Let $n$ be an integer with the property that for each $k$-simplex $\sigma \in X \setminus Y$, the complex $Y \cap \lk_X(\sigma)$ is $(n - k - 1)$-connected. Then the inclusion $Y \hookrightarrow X$ is $n$-connected.
\end{proposition}  

We use a link retraction argument to show that $\IP(W \vert (\overline{V}, \pm \alpha))$ is highly connected. The link retraction argument we use was invented by Maazen \cite[Corollary 4.5]{Maaz}. Church--Putman's \cite[Proof of Theorem 4.2]{ChurchPutnam17} provide an exposition of this link retraction argument. We require the following definition and lemma. 

\begin{definition}
Suppose $X$ is sub-complex of $\IP(W \vert (\overline{V}, \pm \alpha))$. Let $F\colon W \rightarrow \mathbb{Z}$ be a linear map and let $N > 0$. We denote by $X^{< N}$ the sub-complex of $X$ spanned by the set of vertices $\Span{\vec{v}}$ of $X$ satisfying $\lvert F(\vec{v}) \rvert < N$. 
\end{definition}

\begin{lemma} \label{linkRetractionsymp}
Let $W$ be a f.g.~symplectic $\mathbb{Z}$-module, let $\overline{W}$ be a symplectic $\Field_p$-vector space, let $\pi\colon W \rightarrow (\overline{W}, \pm \nu_\pi)$ be a surjection, and let $(\overline{V}, \pm \alpha) \in \sbTD(\overline{W})$. Furthermore, let $F\colon W \rightarrow \mathbb{Z}$ be a linear map and $N > 0$. Suppose $\sigma$ is a simplex of $\IP(W \vert (\overline{V}, \pm \alpha))$ such that some vertex $\Span{\vec{w}}$ of $\sigma$ satisfies $\lvert F(\vec{w})\rvert= N$. Then there exists a simplicial retraction 
\begin{equation*}
\varphi\colon \lk_{\IP(W \vert (\overline{V}, \pm \alpha))}(\sigma) \twoheadrightarrow \lk_{\IP(W \vert (\overline{V}, \pm \alpha))}(\sigma)^{< N}.
\end{equation*}
\end{lemma}  

\begin{proof}
Let $\sigma = \{\vec{u}_1, \dots, \vec{u}_k\}$ be a $(k - 1)$-simplex of $\IP(W \vert (\overline{V}, \pm \alpha))$, $\Span{\sigma} = U$, and set $\lvert F(\vec{u}_1)\rvert= N$. We describe the simplicial map that induces a retraction 
\begin{equation*}
\lk_{\IP(W)}(\sigma) \twoheadrightarrow \lk_{\IP(W)}(\sigma)^{<N}
\end{equation*}
and then check that this map respects the $\pm$-congruence condition. \\
Let $\tau = \{\vec{x}_1, \dots, \vec{x}_t\}$ be a $(t - 1)$-simplex of $\lk_{\IP(W)}(\sigma)$. If $F(\vec{x}_i) > 0$ set $q_{\vec{x}_i} = \lfloor \frac{F(\vec{x}_i)}{N} \rfloor$; if $F(\vec{x}_i) < 0$ set $q_{\vec{x}_i} = \lceil \frac{F(\vec{x}_i)}{N} \rceil$; and if $F(\vec{x}_i) = 0$ set $q_{\vec{x}_i} = 0$. In all three cases, $0 \leq \lvert F(\vec{x}_i) - q_{\vec{x}_i}N \rvert  < N$. Set $\vec{x}_i' = \vec{x}_i - q_{\vec{x}_i}\vec{u}_1$ and define
\begin{equation*}
\varphi(\{\vec{x}_1, \dots,\vec{x}_t\}) = \{\vec{x}_1', \dots,\vec{x}_t'\}. 
\end{equation*}
$\{\vec{u}_1, \dots, \vec{u}_k, \vec{x}_1, \dots, \vec{x}_t\}$ is unimodular if and only if $\{\vec{u}_1, \dots, \vec{u}_k, \vec{x}_1', \dots, \vec{x}_t'\}$ is unimodular. Since 
\begin{equation*}
\omega(\vec{u}_i, \vec{x}_j') = \omega(\vec{u}_i, \vec{x}_j) + q_{\vec{x}_i}\omega(\vec{u}_i, \vec{u}_1) = \omega(\vec{u}_i, \vec{x}_j)
\end{equation*}
for $i \in \{1, \dots, k\}$ and $j \in \{1, \dots, t\}$, then $\{\vec{u}_1, \dots, \vec{u}_k, \vec{x}_1, \dots, \vec{x}_t\}$ is isotropic if and only if $\{\vec{u}_1, \dots, \vec{u}_k, \vec{x}_1', \dots, \vec{x}_t'\}$ is isotropic. $\lvert F(\vec{x}_i') \rvert < N$ for $i \in \{1, \dots, t\}$ by construction. Thus, $\varphi(\tau) \in \lk_{\IP(W)}(\sigma)^{<N}$. If $\gamma = \{\vec{y}_1, \dots, \vec{y}_t\}$ is a $(t - 1)$ simplex in $\lk_{\IP(W)}(\sigma)^{<N}$, then $\lvert F(\vec{y}_i) \rvert < N$ for $i \in \{1, \dots, t\}$. Thus, $\varphi(\gamma) = \gamma$. Therefore
\begin{equation*}
\varphi \colon \lk_{\IP(W)}(\sigma) \twoheadrightarrow \lk_{\IP(W)}(\sigma)^{<N}
\end{equation*}
is a simplicial retraction. \\
We establish that the following simplicial map is a simplicial retraction:
\begin{equation*}
\varphi \vert_{(\lk_{\IP(W \vert (\overline{V}, \pm \alpha))}(\sigma))}\colon \lk_{\IP(W \vert (\overline{V}, \pm \alpha))}(\sigma) \twoheadrightarrow \lk_{\IP(W \vert (\overline{V}, \pm \alpha))}(\sigma)^{< N}.
\end{equation*}
For this purpose, we check that 
\begin{equation*}
\varphi(\lk_{\IP(W \vert (\overline{V}, \pm \alpha))}(\sigma)) \subseteq \lk_{\IP(W \vert (\overline{V}, \pm \alpha))}(\sigma)^{< N},
\end{equation*}
and 
\begin{equation*}
\varphi(\lk_{\IP(W \vert (\overline{V}, \pm \alpha))}(\sigma)^{< N}) = \lk_{\IP(W \vert (\overline{V}, \pm \alpha))}(\sigma)^{< N}. 
\end{equation*}
Let $\tau = \{\vec{x}_1, \dots, \vec{x}_t\}$ be a $(t - 1)$-simplex of $\lk_{\IP(W \vert (\overline{V}, \pm \alpha))}(\sigma)$ and let $\varphi(\tau) = \{\vec{x}_1', \dots, \vec{x}_t'\}$. Since 
\begin{equation*}
\pi(\vec{x}_i') = \pi(\vec{x}_i) + \pi(q_{\vec{x}_i} \vec{u}_1),
\end{equation*}
for $i \in \{1, \dots, t\}$, then 
\begin{equation*}
\Span{\pi(\vec{u}_1), \dots, \pi(\vec{u}_k), \pi(\vec{x}_1), \dots, \pi(\vec{x}_t)}_{\Field_p} \subseteq \overline{V},
\end{equation*}
if and only if 
\begin{equation*}
\Span{\pi(\vec{u}_1), \dots, \pi(\vec{u}_k), \pi(\vec{x}_1'), \dots, \pi(\vec{x}_t')}_{\Field_p} \subseteq \overline{V}.
\end{equation*}
Note that 
\begin{equation*}
\pi(\vec{u}_1) \wedge \pi(\vec{x}_i') = \pi(\vec{u}_1) \wedge \pi(\vec{x}_i) + q_{\vec{x}_i} \pi(\vec{u}_1) \wedge \pi(\vec{u}_1) = \pi(\vec{u}_1) \wedge \pi(\vec{x}_i),
\end{equation*}
for $i \in \{1, \dots, t\}$ because the exterior product is a skew-symmetric bilinear form. When $\Dim \overline{V} = k + t$,
\begin{equation*}
\pi(\vec{u}_1) \wedge \dots \wedge \pi(\vec{u}_k) \wedge \pi(\vec{x}_1) \wedge \dots \wedge \pi(\vec{x}_t) = \pm \alpha
\end{equation*}
if and only if 
\begin{equation*}
\pi(\vec{u}_1) \wedge \dots \wedge \pi(\vec{u}_k) \wedge \pi(\vec{x}_1') \wedge \dots \wedge \pi(\vec{x}_t') = \pm \alpha.
\end{equation*}
We conclude that $\tau \in \lk_{\IP(W \vert (\overline{V}, \pm \alpha))}(\sigma)$ if and only if $\varphi(\tau) \in \lk_{\IP(W \vert (\overline{V}, \pm \alpha))}(\sigma)^{< N}$. Thus, 
\begin{equation*}
\varphi(\lk_{\IP(W \vert (\overline{V}, \pm \alpha))}(\sigma)) \subseteq \lk_{\IP(W \vert (\overline{V}, \pm \alpha))}(\sigma)^{< N}. 
\end{equation*}
As with the non-restricted version, if $\gamma = \{\vec{y}_1, \dots, \vec{y}_t\}$ is a $(t - 1)$ simplex in $\lk_{\IP(W \vert (\overline{V}, \pm \alpha))}(\sigma)^{<N}$, $\lvert F(\vec{y}_i) \rvert < N$ for $i \in \{1, \dots, t\}$ in addition to satisfying the $\pm$-congruence conditions. Thus, $\varphi(\gamma) = \gamma$. Therefore, 
\begin{equation*}
\varphi \vert_{(\lk_{\IP(W \vert (\overline{V}, \pm \alpha))}(\sigma))}\colon \lk_{\IP(W \vert (\overline{V}, \pm \alpha))}(\sigma) \twoheadrightarrow \lk_{\IP(W \vert (\overline{V}, \pm \alpha))}(\sigma)^{<N}
\end{equation*}
is a simplicial retraction.
\end{proof}

\subsection{The symplectic Tits buildings and isotropic partial basis complexes with congruence conditions are CM} 

Proposition \ref{anotherInduction2} is a special case of the following proposition. 

\begin{proposition} \label{anotherInduction3}
Let $W$ be a f.g.~symplectic $\mathbb{Z}$-module, let $\overline{W}$ be a symplectic $\Field_p$-vector space, let $\hat{W} \subsetneq W$ be a restricted summand, let $\hat{\overline{W}} \subsetneq \overline{W}$ be a restricted subspace, let $\pi\colon W \rightarrow \overline{W}$ be a surjection, let $\pi(\hat{W}) = \hat{\overline{W}}$, and let $(\overline{V}, \pm \alpha) \in \sbTD(\hat{\overline{W}}) \subset \sbTD(\overline{W})$. If $W$ is of rank $2m$ and $\overline{V}$ is of dimension $n$ for $m \geq n > 0$, then 
\begin{enumerate}[nosep, topsep=-0.5cm]
\item \label{Item1} $\sbT(\hat{W} \vert (\overline{V}, \pm \alpha))$ is CM of dimension $(n - 1)$;
\item \label{Item2} $\IP(\hat{W} \vert (\overline{V}, \pm \alpha))$ is CM of dimension $(n - 1)$;
\item \label{Item3} $\IP(W \vert (\overline{V}, \pm \alpha))$ is CM of dimension $(n - 1)$;
\item \label{Item4} $\sbT(W \vert (\overline{V}, \pm \alpha))$ is CM of dimension $(n - 1)$. 
\end{enumerate}
\end{proposition}

\begin{proof}
We proceed with induction on $n$. Set $n = 1$ and let $m \geq n$. In this case, $\sbT(\hat{W} \vert (\overline{V}, \pm \alpha)) = \sbT(W \vert (\overline{V}, \pm \alpha))$ and $\IP(\hat{W} \vert (\overline{V}, \pm \alpha)) = \IP(W \vert (\overline{V}, \pm \alpha))$. $\sbT(W \vert (\overline{V}, \pm \alpha))$ is the set of summands $U \in \pi^{-1}(\overline{V}, \pm \alpha)$. Hence, $\sbT(W \vert (\overline{V}, \pm \alpha))$ is a non-empty poset whose summands all have height 0 and for which the links of these summands is empty. Therefore $\sbT(W \vert (\overline{V}, \pm \alpha))$ and $\sbT(\hat{W} \vert (\overline{V}, \pm \alpha))$ are CM of dimension $0$. $\IP(\hat{W} \vert (\overline{V}, \pm \alpha))$ is the simplicial complex composed of vertices that are associated to bases of the rank 1 summands $U \in \pi^{-1}(\overline{V}, \pm \alpha)$. The links of these vertices are empty. Thus $\IP(W \vert (\overline{V}, \pm \alpha))$ and $\IP(\hat{W} \vert (\overline{V}, \pm \alpha))$ are CM of dimension $0$.

\emph{Inductive Hypothesis.} Fix $n > 1$. Suppose that
\begin{enumerate}[labelindent=\parindent,leftmargin=*,nosep, topsep=-0.5cm, label=\emph{\arabic*($n'$).}, ref=\arabic*($n'$)]
\item \label{ItemA} $\sbT(\hat{W}' \vert (\overline{V}', \pm \alpha'))$ is CM of dimension $(n' - 1)$;
\item \label{ItemB} $\IP(\hat{W}' \vert (\overline{V}', \pm \alpha'))$ is CM of dimension $(n' - 1)$;
\item \label{ItemC} $\IP(W' \vert (\overline{V}', \pm \alpha'))$ is CM of dimension $(n' - 1)$;
\item \label{ItemD} $\sbT(W' \vert (\overline{V}', \pm \alpha'))$ is CM of dimension $(n' - 1)$;
\end{enumerate}
for all f.g.~symplectic $\mathbb{Z}$-modules $W'$ of rank $2m'$, all restricted summands $\hat{W}' \subsetneq W'$ of rank $2m' - 1$, and all $(\overline{V}', \pm \alpha') \in \sbTD(\hat{\overline{W}}')$ of dimension $n'$ such that $1 < n' < n$ and $m' \geq n' > 0$. 

For the inductive steps, let $W$ be a f.g.~symplectic $\mathbb{Z}$-module of rank $2m$, let $\hat{W} \subsetneq W$ be a restricted summand of rank $2m - 1$, and let $(\overline{V}, \pm \alpha) \in \sbTD(\hat{\overline{W}})$ be of dimension $n$. We require that $m \geq n$.

\setword{\emph{Inductive step for Item \ref{Item1}}}{InductiveStep1}. Suppose $U, U' \in \sbT(\hat{W} \vert (\overline{V}, \pm \alpha))$ such that $U \subsetneq U'$, $\height(U) = k$, and $\height(U') = k'$. Let $\pi(U) = (\overline{U}, \pm \beta_{\pi})$ and $\pi(U') = (\overline{U}', \pm \beta_{\pi}')$. Lemma \ref{linkposetrestricted} says that
\begin{itemize}[nosep, topsep=-0.5cm]
\item the lower link of $U$ is isomorphic to $\cbT(U)$, which is $(\height(U) - 1)$-dimensional and $(\height(U) - 2)$-connected by Solomon--Tits' Theorem \ref{SBrB} and Miller--Patzt--Putman's Lemma \ref{MPP21Lem3.10}.
\item The upper link of $U$ is isomorphic to either $\sbT(U^\perp/ U \vert (\overline{V}/\overline{U}, \pm \eta))$ if $U^\perp \cap \hat{W} = U^\perp$ or $\sbT((U^\perp \cap \hat{W})/ U \vert (\overline{V}/\overline{U}, \pm \eta))$ if $U^\perp \cap \hat{W} \neq U^\perp$. Note that
\begin{equation*}
n - (k + 1) = (n - 1) - \height(U).
\end{equation*}
For the case $U^\perp \cap \hat{W} = U^\perp$, the Inductive Hypothesis of Item \emph{\ref{ItemD}} for $n' = (n - 1) - \height(U)$ says that the upper link of $U$ is $((n - 1) - \height(U) - 1)$-dimensional and $((n - 1) - \height(U) - 2)$-connected. For the case $U^\perp \cap \hat{W} \neq U^\perp$, the Inductive Hypothesis of Item \emph{\ref{ItemA}} for $n' = (n - 1) - \height(U)$ says that the upper link of $U$ is $((n - 1) - \height(U) - 1)$-dimensional and $((n - 1) - \height(U) - 2)$-connected.
\item The interval from $U$ to $U'$ is isomorphic to $\cbT(U'/U)$. One computes that 
\begin{equation*}
(k' + 1) - (k  + 1) =  \height(U') - \height(U).
\end{equation*}
By Solomon--Tits' Theorem \ref{SBrB} and Miller--Patzt--Putman's Lemma \ref{MPP21Lem3.10} $\cbT(U'/U)$ is $(\height(U') - \height(U) - 2)$-dimensional and $(\height(U') - \height(U) - 3)$-connected.
\end{itemize}
$(\overline{V}, \pm \alpha)$ is of height $(n - 1)$. So $\sbT(\hat{W} \vert (\overline{V}, \pm \alpha))$ is $(n - 1)$-dimensional. All that remains to prove is that $\sbT(\hat{W} \vert (\overline{V}, \pm \alpha))$ is $(n - 2)$-connected. Pick a rank-1 summand $L \subset \hat{W}$ such that $\pi(L) \subset \overline{V}$. There exists a summand $W'$ that satisfies $\omega(\vec{l}, \vec{w}) = 0$ for all $\vec{l} \in L$ and $\vec{w} \in W'$, as well as $W' \cap L = 0$. Set $\overline{V}' := \pi(W') \cap \overline{V}$, which is codimension-1. $W'$ is either a symplectic summand of rank $2n - 2$, in which case $W' \oplus L = \hat{W}$, or a restricted summand of rank $2n - 3$, in which case there exists a rank-1 summand $L'$ such that $L \oplus L'$ is symplectic and $W' \oplus L \oplus L' = \hat{W}$. \\
Suppose $W'$ is a symplectic summand of rank $2n - 2$. Consider the poset inclusion maps
\begin{equation*}
\begin{tikzcd}
\sbT(W' \vert (\overline{V}', \pm \alpha')) \arrow[r, "g"', shift right, hook] \arrow[r, "f", shift left, hook] & \sbT(\hat{W} \vert (\overline{V}, \pm \alpha)),
\end{tikzcd}
\end{equation*}
where $f$ is an inclusion poset map that sends a summand $U$ to $U$ and $g$ is a poset map that sends a summand $U$ to $U \oplus L$. The Inductive Hypothesis of Item \emph{\ref{ItemD}} for $n' = n - 1$ says that $\sbT(W' \vert (\overline{V}', \pm \alpha'))$ is $(n - 3)$-connected. We demonstrate that $f$ is an $(n - 2)$-connected poset map with which we conclude that $\sbT(\hat{W} \vert (\overline{V}, \pm \alpha))$ is $(n - 3)$-connected. In addition, we show that $f$ is nullhomotopic which implies the stronger and desired result: $\sbT(\hat{W} \vert (\overline{V}, \pm \alpha))$ is $(n - 2)$-connected. \\
$L \subset g(U)$ for all $U \in \sbT(W' \vert (\overline{V}', \pm \alpha'))$. So the image of $g$ is contractible with cone point $L$. Since $f(U) \subseteq g(U)$ for all $U \in \sbT(W' \vert (\overline{V}', \pm \alpha'))$, then Quillen's Proposition \ref{posetmaphomotopy} says that $f$ and $g$ are homotopic. We conclude that $f$ is nullhomotopic. All that remains to show is that $f$ is $(n - 3)$-connected. \\
Consider the following filtration of $\sbT(\hat{W} \vert (\overline{V}, \pm \alpha))$:  
\begin{align*}
F_{1} & := \sbT(W' \vert (\overline{V}', \pm \alpha')),\\
F_2 & := F_1 \cup \{U \in \sbT(\hat{W} \vert (\overline{V}, \pm \alpha)) \mid \height(U) > 0\}, \\
F_3 &:= \sbT(\hat{W} \vert (\overline{V}, \pm \alpha)).
\end{align*}
Set $h_1\colon F_1 \hookrightarrow F_2$ and $h_2\colon F_2 \hookrightarrow F_3$ to be the inclusion maps that send summands $U$ to $U$. $(h_1)_{\leq U}$ is contractible with cone point $U$ if $U \in F_1$ or with cone point $U \cap W'$ if $U \in F_2 \setminus F_1$. By Quillen's Proposition \ref{posetmaphomotopy2}, $h_1$ is a homotopy equivalence. Since $F_2$ adds summands of rank $1$, the fibers of these summands do not have an obvious cone point like before. \\
Let $\tilde{h}_{2}\colon \simp(F_{2}) \hookrightarrow \simp(F_3)$ be the simplicial map that is associated to the poset inclusion map $h_{2}\colon F_{2} \hookrightarrow F_3$. The isotropic flag that corresponds to a $k$-simplex $\sigma \in \simp(F_3) \setminus \simp(F_{2})$ must be of the form
\begin{equation*}
0 \subsetneq U \subsetneq Q_1 \subsetneq \dots \subsetneq Q_k \subsetneq \hat{W},
\end{equation*}
where $U \in F_3 \setminus F_{2}$ of rank $1$ and $Q_i \in F_{2}$ for $i \in \{1, \dots, k\}$. So, 
\begin{equation*}
\simp(F_{2}) \cap \lk_{\simp(F_3)}(\sigma) = \lk_{\sT(\hat{W} \vert (\overline{V}, \pm \alpha))}(\sigma).
\end{equation*}
By Lemma \ref{linkposetrestricted} and Remark \ref{PosetToSimpLink},
\begin{equation*}
\lk_{\sT(\hat{W} \vert (\overline{V}, \pm \alpha))}(\sigma) \cong \cT(U) * \cT(Q_1/U) * \dots *\cT(Q_{k - 1}/Q_{k - 1}) * \sT(Q_k^\perp/Q_k \vert (\overline{V}/\overline{Q}_k, \pm \eta)),
\end{equation*}
if $Q_k^\perp \cap \hat{W} = Q_k^\perp$, or
\begin{equation*}
\lk_{\sT(\hat{W} \vert (\overline{V}, \pm \alpha))}(\sigma) \cong \cT(U) * \cT(Q_1/U) * \dots *\cT(Q_{k - 1}/Q_{k - 1}) * \sT((Q_k^\perp \cap \hat{W})/Q_k \vert (\overline{V}/\overline{Q}_k, \pm \eta))
\end{equation*}
if $Q_k^\perp \cap \hat{W} \neq Q_k^\perp$, for some $\pm$-orientation $\pm \eta$ on $\overline{V}/\overline{Q}_k$. From our work above on the links of $\sbT(\hat{W} \vert (\overline{V}, \pm \alpha))$, we know that $\lk_{\sT(\hat{W} \vert (\overline{V}, \pm \alpha))}(\sigma)$ is $((n - 1) - k - 1)$-dimensional and $((n - 1) - k - 2) = ((n - 2) - k - 1)$-connected in either case. Using Galatius--Randall-Williams' Proposition \ref{GalatiusRandal-Williams14}, we conclude that $\tilde{h}_{2}$ is $(n - 2)$-connected. Since $\tilde{h}_{2}$ and $h_{2}$ are the same at the level of spaces and homology, then $h_{2}$ is also $(n - 2)$-connected. $f = h_{2} \circ h_1$ and $h_1$ is a homotopy equivalence. So $f$ is $(n - 2)$-connected as desired.\\
Suppose instead that $W'$ is a restricted summand of rank $2n - 3$. We write $\hat{W}'$ instead of $W'$ to keep with our conventions. Consider the poset inclusion maps
\begin{equation*}
\begin{tikzcd}
\sbT(\hat{W}' \vert (\overline{V}', \pm \alpha')) \arrow[r, "\hat{g}"', shift right, hook] \arrow[r, "\hat{f}", shift left, hook] & \sbT(\hat{W} \vert (\overline{V}, \pm \alpha)),
\end{tikzcd}
\end{equation*}
where $\hat{f}$ is an inclusion poset map that sends a summand $U$ to $U$ and $\hat{g}$ is a poset map that sends a summand $U$ to $U \oplus L$. Our strategy remains the same: show that $\hat{f}$ is nullhomotopic and $(n - 2)$-connected, and invoke the inductive hypothesis with which we conclude that $\sbT(\hat{W} \vert (\overline{V}, \pm \alpha))$ is $(n - 2)$-connected. \\
$\hat{g}$ is nullhomotopic since the image of $\hat{g}$ has cone point $L$. $\hat{f}$ and $\hat{g}$ are homotopic by Quillen's Proposition \ref{posetmaphomotopy} since $\hat{f}(U) \subseteq \hat{g}(U)$ for all $U \in \sbT(\hat{W}' \vert (\overline{V}', \pm \alpha'))$. So $\hat{f}$ is also nullhomotopic. \\
Consider the following filtration of $\sbT(\hat{W} \vert (\overline{V}, \pm \alpha))$:  
\begin{align*}
\hat{F}_{1} & := \sbT(\hat{W}' \vert (\overline{V}', \pm \alpha')),\\
\hat{F}_2 & := F_1 \cup \{U \in \sbT(\hat{W} \vert (\overline{V}, \pm \alpha)) \mid \height(U) > 0\}, \\
\hat{F}_3 & := \sbT(\hat{W} \vert (\overline{V}, \pm \alpha)).
\end{align*}
Set $\hat{h}_1\colon \hat{F}_1 \hookrightarrow \hat{F}_2$ and $\hat{h}_2\colon \hat{F}_2 \hookrightarrow \hat{F}_3$ to be the inclusion maps that send isotropic summands to themselves. $(\hat{h}_1)_{\leq U}$ is contractible with cone point $U$ if $U \in \hat{F}_1$ or with cone point $U \cap \hat{W}'$ if $U \in \hat{F}_2 \setminus \hat{F}_1$. There is a subtlety here: $U \cap \hat{W}'$ is a corank-1 summand instead of a corank-2 summand for isotropic summands $U \in F_2$. There exists a rank-1 summand $L'$ such that $\hat{W}' \oplus L \oplus L' = \hat{W}$ and $L \oplus L'$ is a symplectic summand, but $\pi(L') \not \subseteq \overline{V}$ since $\pi(L) \subset \overline{V}$. $U \cap \hat{W}' = U \cap (\hat{W}' \oplus L')$ so $U \cap \hat{W}'$ is corank-1. This fact is relevant to the rank-2 summands of $F_2$. By Quillen's Proposition \ref{posetmaphomotopy2}, $\hat{h}_1$ is a homotopy equivalence. \\
The rest of this argument is the same as the case where $W'$ is symplectic. Let $\tilde{\hat{h}}_{2}\colon \simp(F_{2}) \hookrightarrow \simp(F_3)$ be the simplicial map that is associated to the poset map $\hat{h}_2$. The isotropic flag associated to a $k$-simplex $\sigma \in \simp(\hat{F}_3) \setminus \simp(\hat{F}_{2})$ must be of the form
\begin{equation*}
0 \subsetneq U \subsetneq Q_1 \subsetneq \dots \subsetneq Q_k \subsetneq \hat{W},
\end{equation*}
where $U \in \hat{F}_3 \setminus \hat{F}_{2}$ of rank $1$ and $Q_i \in \hat{F}_{2}$ for $i \in \{1, \dots, k\}$. So, 
\begin{equation*}
\simp(\hat{F}_{2}) \cap \lk_{\simp(\hat{F}_3)}(\sigma) = \lk_{\sT(\hat{W} \vert (\overline{V}, \pm \alpha))}(\sigma).
\end{equation*}
By Lemma \ref{linkposetrestricted} and Remark \ref{PosetToSimpLink},
\begin{equation*}
\lk_{\sT(\hat{W} \vert (\overline{V}, \pm \alpha))}(\sigma) \cong \cT(U) * \cT(Q_1/U) * \dots *\cT(Q_{k - 1}/Q_{k - 1}) * \sT(Q_k^\perp/Q_k \vert (\overline{V}/\overline{Q}_k, \pm \eta)),
\end{equation*}
if $Q_k^\perp \cap \hat{W} = Q_k^\perp$, or
\begin{equation*}
\lk_{\sT(\hat{W} \vert (\overline{V}, \pm \alpha))}(\sigma) \cong \cT(U) * \cT(Q_1/U) * \dots *\cT(Q_{k - 1}/Q_{k - 1}) * \sT((Q_k^\perp \cap \hat{W})/Q_k \vert (\overline{V}/\overline{Q}_k, \pm \eta))
\end{equation*}
if $Q_k^\perp \cap \hat{W} \neq Q_k^\perp$, for some $\pm$-orientation $\pm \eta$ on $\overline{V}/\overline{Q}_k$. Once again, $\lk_{\sT(\hat{W} \vert (\overline{V}, \pm \alpha))}(\sigma)$ is $((n - 1) - k - 1)$-dimensional and $((n - 1) - k - 2) = ((n - 2) - k - 1)$-connected in either case. Galatius--Randall-Williams' Proposition \ref{GalatiusRandal-Williams14} says that $\tilde{\hat{h}}_{2}$ is $(n - 2)$-connected. Since $\tilde{\hat{h}}_{2}$ and $\hat{h}_{2}$ are the same at the level of spaces and homology, $\hat{h}_{2}$ is also $(n - 2)$-connected. Finally, $\hat{f} = \hat{h}_{2} \circ \hat{h}_1$ and $\hat{h}_1$ is a homotopy equivalence, which implies $\hat{f}$ is $(n - 2)$-connected. \\
$\hat{f}$ is $(n - 2)$-connected and nullhomotopic, and $\sbT(\hat{W}' \vert (\overline{V}', \pm \alpha'))$ is $(n - 3)$-connected by the Inductive Hypothesis of Item \emph{\ref{ItemA}} for $n' = n - 1$. So $\sbT(\hat{W} \vert (\overline{V}, \pm \alpha))$ is $(n - 2)$-connected.  

\setword{\emph{Inductive step for Item \ref{Item2}}}{InductiveStep2}. Consider the strictly increasing poset map 
\begin{align*}
\text{Span} & \colon \posety(\IP(\hat{W} \vert (\overline{V}, \pm \alpha)))   \rightarrow \sbT(\hat{W} \vert (\overline{V}, \pm \alpha))\\
\text{Span} & \colon  \{\vec{v}_1, \dots, \vec{v}_k\}  \mapsto \Span{\vec{v}_1, \dots, \vec{v}_k}. 
\end{align*}
 By the \ref{InductiveStep1}, $\sbT(\hat{W} \vert (\overline{V}, \pm \alpha))$ is CM of dimension $(n - 1)$. The fiber of $U \in \sbT(\hat{W} \vert (\overline{V}, \pm \alpha))$ with respect to the map $\text{Span}$ is $\posety(\BP(U))$, which is homotopy equivalent to $\BP(U)$. Maazen's Theorem \ref{CP17Thm4.2} says that $\BP(U)$ is CM of dimension $\height(U)$. So $\BP(U)$ is $(\height(U) - 1)$-connected. Quillen's Theorem \ref{Q78Cor9.7} says that $\IP(\hat{W} \vert (\overline{V}, \pm \alpha))$ is CM of dimension $(n - 1)$. 

\setword{\emph{Inductive step for Item \ref{Item3}}}{InductiveStep3}. We begin with the links of $\IP(W \vert (\overline{V}, \pm \alpha))$. Let $\sigma \in \IP(W \vert (\overline{V}, \pm \alpha))$ be a $(k - 1)$-simplex and set $\Span{\sigma} = U$ and $\pi(U) = (\overline{U}, \pm \beta_\pi)$. By Lemma \ref{SympCongruenceSimpLinks}, 
\begin{equation*}
\lk_{\IP(W \vert (\overline{V}, \pm \alpha))}(\sigma) \cong \IP(U^\perp/U \vert (\overline{V}/\overline{U}, \pm \eta)) \Span{U}
\end{equation*}
for some $\pm$-orientation $\pm \eta$ on $\overline{V}/\overline{U}$. $\IP(U^\perp/U \vert (\overline{V}/\overline{U}, \pm \eta)) \Span{U}$ is a complete join complex over $\IP(U^\perp/U \vert (\overline{V}/\overline{U}, \pm \eta))$ by Remark \ref{LabelingSystemRemark}. The Inductive Hypothesis of Item \ref{ItemC} for $n' = n - k$ says that $\IP(U^\perp/U \vert (\overline{V}/\overline{U}, \pm \eta))$ is CM of dimension $(n - k - 1)$. We invoke Hatcher--Wahl's Proposition \ref{HWProp3.5}, with which we conclude that $\IP(U^\perp/U \vert (\overline{V}/\overline{U}, \pm \eta)) \Span{U}$ is CM of dimension $((n - 1) - k)$. Thus $\lk_{\IP(W \vert (\overline{V}, \pm \alpha))}(\sigma)$ is $((n - 1) - k)$-dimensional and $((n - 1) - k - 1)$-connected.\\
$\IP(W \vert (\overline{V}, \pm \alpha))$ is $(n - 1)$-dimensional because a maximal simplex in $\IP(W \vert (\overline{V}, \pm \alpha))$ corresponds to an isotropic partial basis of size $n$. All that remains to prove is that $\IP(W \vert (\overline{V}, \pm \alpha))$ is $(n - 2)$-connected. Fix $q \in \{0, \dots, n - 2\}$ and let $S^q$ be a combinatorial triangulation of a $q$-sphere, and let 
\begin{equation*}
\varphi\colon S^q \rightarrow\IP(W \vert (\overline{V}, \pm \alpha))
\end{equation*}
be a simplicial map. We show that $\varphi$ can be homotoped to a constant map for all $q$. This implies that $\IP(W \vert (\overline{V}, \pm \alpha))$ is $(n - 2)$-connected. \\
Pick $\vec{l} \in W$ so that $\pi(\vec{l}) \notin \overline{V}$ and let $L = \Span{\vec{l}}$. Set $\hat{W}$ to be the restricted summand that satisfies $\hat{W} \oplus L = W$ and $\overline{V} \subset \pi(\hat{W})$. Let $F\colon W \rightarrow \mathbb{Z}$ be a linear map that takes a vector $\vec{v} \in W$ to the $\vec{l}$-coordinate of $\vec{v}$. Let $r\colon \IP(W \vert (\overline{V}, \pm \alpha)) \rightarrow \mathbb{Z}$ be a map that takes a vertex $\{\vec{v}\}$ to $\lvert F(\vec{v}) \rvert$. We define the complexity of $\varphi$ by 
\begin{align*}
R(\varphi) = \{\max r(\varphi(s)) \mid s \in S^q\}. 
\end{align*}
If $R(\varphi) = 0$, then $\varphi(S^q) \subseteq \IP(\hat{W} \vert (\overline{V}, \pm \alpha))$. The \ref{InductiveStep2} says that $\IP(\hat{W} \vert (\overline{V}, \pm \alpha))$ is $(n - 2)$-connected, which means that $\varphi$ can be homotoped to the constant map for $q \in \{0, \dots, n - 2\}$ as desired.\\
Suppose instead $R(\varphi) > 0$ and let $R(\varphi) = R$. Consider the following condition on a simplex $\sigma$ of $S^q$: 
\begin{equation}\label{conditionInduction}
r(\varphi(s)) = \text{ $R$ for all vertices $s$ of $\sigma$ }.
\end{equation}
By definition, there exists some $\sigma \in S^q$ satisfying Condition \ref{conditionInduction} because $R(\varphi) = R$. We can suppose $\sigma$ is simplex of maximal dimension $k \leq n - 1$. Let $\ell$ be the dimension of the simplex $\varphi(\sigma) \in \IP(W \vert (\overline{V}, \pm \alpha))$. Necessarily $\ell \leq k$, where $\ell < k$ if $\varphi$ is not injective on $\sigma$. Set $Y = \lk_{\IP(W \vert (\overline{V}, \pm \alpha))}(\varphi(\sigma))$. By the work above on the computation of the connectivity of links of $\IP(W \vert (\overline{V}, \pm \alpha))$, $Y$ is $((n - 1) - \ell - 2)$-connected. Since $Y$ retracts to $Y^{< R}$ by Lemma \ref{linkRetractionsymp}, then $Y^{< R}$ is also $(n - \ell - 3)$-connected. \\
Zeeman's relative simplicial approximation theorem \cite{Zee64} says that there exists a sub-complex $K \subset S^q$ that contains $\lk_{S^q}(\sigma)$ and a simplicial map $\psi\colon K \rightarrow Y^{< R}$ such that $\psi \simeq \varphi$ and $\psi \vert_{\lk_{S^q}(\sigma)} = \varphi \vert_{\lk_{S^q}(\sigma)}$. $\lk_{S^q}(\sigma)$ is a combinatorial $(q - k - 1)$-sphere. Since $q \leq n - 2$ and $\ell \leq k$, then $q - k - 1 \leq n - \ell - 3$. So $\varphi \vert_{\lk_{S^q}(\sigma)}$ is null-homotopic via a homotopy within $Y^{< R}$, because this simplicial complex is $(n - \ell - 3)$-connected. So there exists a $(q - k)$-ball $B$ and a simplicial isomorphism $\gamma$ such that $\gamma\colon K \cong B$ and $\gamma\colon \partial B \cong \lk_{S^q}(\sigma)$. The composition of simplicial maps $\tilde{\psi} = \psi \circ \gamma^{-1}$ gives a map 
\begin{equation*}
\tilde{\psi}\colon B \rightarrow Y^{< R}
\end{equation*}
such that $\tilde{\psi} \vert_{\partial B} = \varphi \vert_{\lk_{S^q}(\sigma)}$. \\
Consider the sub-complex $\sigma * B$. $\sigma *B$ is a $(q + 1)$-ball whose boundary is the union of the $q$-balls $\sigma * (\partial B)$ and $(\partial \sigma) * B$. $\sigma * (\partial B) = \starry_{S^q}(\sigma)$, which contains simplices that satisfy condition \ref{conditionInduction}, while $(\partial \sigma) * B$ contains simplices that do not satisfy Condition \ref{conditionInduction}. The map 
\begin{equation*}
(\varphi \vert_\sigma) * \tilde{\psi}\colon \sigma * B \rightarrow \varphi(\sigma) * Y^{< R}
\end{equation*}
extends $\tilde{\psi}$ to $\sigma * B$. Note that $(\varphi \vert_\sigma) * \tilde{\psi} = \varphi \vert_{\starry_{S^q}(\sigma)}$ and $\varphi(\sigma) * Y^{< R}$ is a sub-complex of $\IP(W \vert (\overline{V}, \pm \alpha))$. We can homotope $\varphi$ across $\sigma * B$ to replace $\varphi \vert_{\starry_{S^q}(\sigma)}$ with 
\begin{equation*}
(\varphi \vert_{\partial \sigma}) * \tilde{\psi}\colon (\partial \sigma) * B \rightarrow \IP(W \vert (\overline{V}, \pm \alpha)).
\end{equation*}
This modification to $\varphi$ eliminates $\sigma$ and does not add simplices that satisfy condition \ref{conditionInduction}. Every simplex added by this homotope is the join of a simplex in $\partial \sigma$ and $B$, where $\tilde{\psi}(B)$ is a sub-complex of $Y^{< R}$. Hence, each added simplex contains at least one vertex $s$ such that $r(\varphi(s)) < R$. Hence, repeating the process described above removes all simplices that satisfy condition \ref{conditionInduction}. We can iterate this procedure for smaller $R$ until $R(\varphi) = 0$.

\setword{\emph{Inductive step for Item \ref{Item4}}}{InductiveStep4}. Suppose $U, U' \in \sbT(W \vert (\overline{V}, \pm \alpha))$ such that $U \subsetneq U'$. Let $\pi(U) = \overline{U}$ and $\pi(U') = \overline{U}'$. Lemma \ref{SympCongruencePosetLinks} says that
\begin{itemize}[nosep, topsep=-0.5cm]
\item the lower link of $U$ is isomorphic to $\cbT(U)$, which is $(\height(U) - 1)$-dimensional and $(\height(U) - 2)$-connected by Solomon--Tits' Theorem \ref{SBrB} and Miller--Patzt--Putman's Lemma \ref{MPP21Lem3.10}.
\item The upper link of $U$ is isomorphic to $\sbT(U^\perp/ U \vert (\overline{V}/\overline{U}, \pm \eta))$ for some $\pm$-orientation $\eta$. By the Inductive Hypothesis of Item \emph{\ref{ItemA}} for $n' = (n - 1) - \height(U)$, $\sbT(U^\perp/ U \vert (\overline{V}/\overline{U}, \pm \eta))$ is $((n - 1) - \height(U) - 1)$-dimensional and $((n - 1) - \height(U) - 2)$-connected.
\item The interval from $U$ to $U'$ is isomorphic to $\cbT(U'/U)$ which is $(\height(U') - \height(U) - 2)$-dimensional and $(\height(U') - \height(U) - 3)$-connected by Solomon--Tits' Theorem \ref{SBrB} and Miller--Patzt--Putman's Lemma \ref{MPP21Lem3.10}.
\end{itemize}
Since $\overline{V}$ is of dimension $n$, then summands in $\sbT(W \vert (\overline{V}, \pm \alpha))$ have height at most $(n - 1)$. Thus, $\sbT(W \vert (\overline{V}, \pm \alpha))$ is $(n - 1)$-dimensional. Define 
\begin{align*}
t & \colon \sbT(W \vert (\overline{V}, \pm \alpha)) \rightarrow \mathbb{Z}, \\
t & \colon U \mapsto \rank (U), 
\end{align*} 
and consider the poset spanning map 
\begin{align*}
\text{Span} & \colon \posety(\IP(W \vert (\overline{V}, \pm \alpha))) \rightarrow \sbT(W \vert (\overline{V}, \pm \alpha)), \\
\text{Span} & \colon \{\vec{v}_1, \dots, \vec{v}_k\} \mapsto \Span{\vec{v}_1, \dots, \vec{v}_k}. 
\end{align*}
The fiber of $U \in \sbT(W \vert (\overline{V}, \pm \alpha))$ with respect to $\text{Span}$ is $\posety(\BP(U))$. By Maazen's Theorem \ref{CP17Thm4.2}, $\BP(U)$ is CM of dimension $\height(U) = (\rank U - 1)$ and, in particular, $(\rank(U) - 2) = (t(U) - 2)$-connected. The upper link of $U$ is $((n - 1) - \height(U) - 2) = ((n - 1) - t(U) - 1)$-connected by the work above. By van der Kallen--Looijenga's Theorem \ref{VdKaL11Cor2.2v2}, the poset map $\text{Span}$ is $(n - 1)$-connected. The \ref{InductiveStep3} says that $\IP(W \vert (\overline{V}, \pm \alpha))$ is $(n - 2)$-connected. We conclude that $\sbT(W \vert (\overline{V}, \pm \alpha))$ is $(n - 2)$-connected. 

The inductive step of all four items is completed and so is the proof of Proposition \ref{anotherInduction3}. 
\end{proof}


\section{Top-degree cohomology of $\Gamma_{2n}^\omega(p)$} \label{proofofA}

We state a proposition we require for the proof of Theorem \ref{theorem:main:surjection}.

\begin{proposition}\label{surjectivityProposition}
For $n \geq 0$ and a prime $p$, the map
\begin{equation}
\St_{2n}^\omega(\mathbb{Z}) \rightarrow \St_{2n}^{\omega, \pm}(\Field_p)
\end{equation}
is a surjection.
\end{proposition}

\begin{proof}
Let $\pi\colon \mathbb{Z}^{2n} \rightarrow \Field_p^{2n}$ be the reduction mod-$p$ map, and let 
\begin{equation} \label{THEmap2}
\pi\colon \sbT_{2n}(\mathbb{Z}) \rightarrow \sbTD_{2n}(\Field_p)
\end{equation}
be an induced poset map that takes a nonzero isotropic direct summand $V \subset \mathbb{Z}^{2n}$ to an isotropic subspace $\overline{V} \subset \Field_p^{2n}$ equipped with an induced $\pm$-orientation. We check the conditions for Church--Putman's Proposition \ref{CPQuillenFiber}:
\begin{itemize}[nosep, topsep=-0.5cm]
  \item $\sTD_{2n}(\Field_p)$ is CM of dimension $(n - 1)$ by Proposition \ref{SympMPP21Lem3.15}. 
  \item the poset fiber of $(\overline{V}, \pm \alpha) \in \sbTD_{2n}(\Field_p)$ with respect to $\varphi$ is a symplectic Tits building with a $\pm$-congruence condition: $\sbT(\mathbb{Z}^{2n} \vert (\overline{V}, \pm \alpha))$. By Proposition \ref{anotherInduction2}, the fiber is CM of dimension $(\Dim \overline{V} - 1) = \height(\overline{V})$. So the poset fiber is $(\height(\overline{V}) - 1)$-connected. 
\end{itemize}
It follows that 
\begin{align} \label{themap5}
\St_{2n}^\omega(\mathbb{Z}) = \rh_{n - 1}(\sbT_{2n}(\mathbb{Z});\mathbb{Z}) \rightarrow \rh_{n - 1}(\sbTD_{2n}(\Field_p);\mathbb{Z}) = \St_{2n}^{\omega, \pm}(\Field_p)
\end{align}
is a surjection as desired.
\end{proof}

We restate Theorem \ref{theorem:main:surjection}. 

\begin{theorem} \label{theoremArestrated}
For $n \geq 0$ and $p$ a prime integer, the map
\begin{equation*}
\hh^{n^2}(\Gamma_{2n}^\omega(p);\mathbb{Q}) \rightarrow \rh_{n - 1}(\vert \sT_{2n}(\mathbb{Q})\vert /\Gamma_{2n}^\omega(p); \mathbb{Q})
\end{equation*}
is a surjection.
\end{theorem}

\begin{proof}[Proof of Theorem \ref{theorem:main:surjection} \& Theorem \ref{theoremArestrated}]
By Borel--Serre duality \cite[Theorem 11.4.2]{BorelSerreCorners}, 
\begin{equation*}
\hh^{n^2}(\Gamma_{2n}^\omega(p);\mathbb{Q}) \cong (\St_{2n}^\omega(\mathbb{Q}))_{\Gamma_{2n}^\omega(p)}. 
\end{equation*}
Consider the quotient map 
\begin{equation*}
f\colon \vert \sT_{2n}(\mathbb{Q})\vert \rightarrow \lvert \sT_{2n}(\mathbb{Q})\rvert/\Gamma_{2n}^\omega(p).
\end{equation*}
$f$ is $\Gamma_{2n}^\omega(p)$-invariant and it induces the following map on homology. 
\begin{equation*}
f_*\colon\St_{2n}^\omega(\mathbb{Q}) = \rh_{n - 1}(\lvert \sT_{2n}(\mathbb{Q})\rvert; \mathbb{Q}) \rightarrow \rh_{n - 1}(\lvert \sT_{2n}(\mathbb{Q})\rvert/\Gamma_{2n}^\omega(p); \mathbb{Q}). 
\end{equation*}
The coinvariants factor through this map. Hence, it suffices to demonstrate that $f_*$ is surjective. \\
Lemma \ref{SympMPP21Lem3.10} says that $\St_{2n}^\omega(\mathbb{Q}) \cong \St_{2n}^\omega(\mathbb{Z})$ while Proposition \ref{aSimplicialModel} implies 
\begin{equation*}
\rh_{n - 1}(\sT_{2n}(\mathbb{Q})/\Gamma_{2n}^\omega(p); \mathbb{Q}) \cong \St_{2n}^{\omega, \pm}(\Field_p).
\end{equation*}
Since the map
\begin{equation*}
\St_{2n}^\omega(\mathbb{Z}) \rightarrow \St_{2n}^{\omega, \pm}(\Field_p)
\end{equation*} 
is surjective by Proposition \ref{surjectivityProposition}, $f_*$ is surjective as desired.
\end{proof}

\begin{remark}
We examined several alternative approaches to the proof of Proposition \ref{surjectivityProposition}. One approach involved finding an explicit set of generators for $\St_{2n}^{\omega, \pm}(\Field_p)$ that lift easily to generators of $\St_{2n}^\omega(\Field_p)$ using Br\"uck--Sroka's \cite[Proposition 5.1 and 5.2]{BruckSroka24} strategy. This approach fails for two reasons. The first reason is that the $\pm$-oriented restricted symplectic Tits building of $\Field_p$ is not obviously contractible. The second, and more important reason, is that their strategy required the augmented partial basis complex to be spherical. For $R = \mathbb{Z}$, this simplicial complex is one more connected than the partial basis complex. We would require that the quotient of the augmented partial basis complex by a congruence subgroup be one more connected than a partial basis complex as well. Miller--Patzt-Putman determined that this is not true for primes $p > 7$ (c.f. \cite[Lemmas 2.43 and 2.50]{MiPaP21}). \\
Another approach involved demonstrating that the map
\begin{equation*}
f\colon \IP(\mathbb{Z}^{2n}) \rightarrow \IP(\Field_p^{2n}),
\end{equation*}
induces a surjection on homology using Church--Putman's Proposition \ref{CPQuillenFiber}. $\IP(\Field_p^{2n})$ is related to $\sbTD_{2n}(\Field_p)$ via poset map. This approach fails because the fibers of $f$ are not spherical. 
\end{remark}


\section{Numerical results} \label{computationalsection}

We restate Theorem \ref{theorem:main:computational}. 

\begin{theorem} \label{theoremBrestrated}
Fix a prime $p \geq 3$ and let $t^\omega(n, p):= \rank (\rh_{n - 1}(\sT_{2n}(\mathbb{Q})/\Gamma_{2n}^\omega(p); \mathbb{Z}))$. For $n \geq 1$, $t^\omega(n, p)$ equals
\begin{equation} \label{mainequation2}
\sum_{0 = m_{-1} < m_0 < \dots < m_k \leq n} \left( \prod_{i = 0}^{k - 1} p^{m_{i+1}-m_i\choose 2} \prod_{j = 0}^{m_i - 1} \frac{p^{m_{i + 1} - j} - 1}{p^{m_i - j} - 1} \right) \left( \prod_{i = 0}^{m_k - 1} \frac{p^{2n - 2i} - 1}{p^{m_k - i} - 1}\right) p^{{m_0 \choose 2} + (n - m_k)^2} \left(\frac{p - 3}{2}\right)^{k + 1}. 
\end{equation}
\end{theorem}

Proposition \ref{aSimplicialModel} implies that
\begin{equation}
\rh_{n - 1}(\sT_{2n}(\mathbb{Q})/\Gamma_{2n}^\omega(p);\mathbb{Z}) \cong \rh_{n - 1}(\sTD_{2n}(\Field_p);\mathbb{Z}) = \St_{2n}^{\omega, \pm}(\Field_p). 
\end{equation}
We prove Theorem \ref{theorem:main:computational} by computing the rank of $\St_{2n}^{\omega, \pm}(\Field_p)$. The technique we use relies on the following theorem. 

\begin{theorem} \textup{\cite[Theorem 1.1]{BjornerWachsWelker}}\label{nonrecursiveFormula} Let $A$ and $B$ be finite posets where $\height(B_{\leq b}) < \infty$ for all $b \in B$. Let $F\colon A \rightarrow B$ be a map of posets such that for all $b \in B$ the fibre $F_{\leq b}$ is $(\Dim(F_{< b}) - 1)$-connected. Then we have 
\begin{align*}
\lvert A \rvert \simeq \lvert B \rvert \vee \left( \bigvee_{b \in B} \lvert F_{\leq b} \rvert * \lvert B_{> b} \rvert \right).
\end{align*}
\end{theorem}

Given a poset $A$ and its associated simplicial complex $\simp(A)$, we can associate a poset to the simplicial complex $\simp(A)$, $\posety(\simp(A))$, which we call a \emph{poset of flags of $A$}. The elements of $\posety(\simp(A))$ are flags $[a_0 < \dots < a_k]$ and the order is inclusion of flags. Note that $\posety(\simp(A))$ is homotopy equivalent to $A$. We apply Bj\"orner--Wachs--Welker's Theorem \ref{nonrecursiveFormula} to $\posety(\simp(\sbTD_{2n}(\Field_p)))$ and $\posety(\simp(\sbT_{2n}(\Field_p)))$. We denote the elements of $\posety(\simp(\sbT_{2n}(\Field_p)))$ by $[\overline{V}_0 \subsetneq \dots \subsetneq \overline{V}_k]$. We decorate the a $\pm$ the elements of $\posety(\simp(\sbTD_{2n}(\Field_p)))$: $[\overline{V}_0 \subsetneq \dots \subsetneq \overline{V}_k]^\pm$. \\
Recall that there exists a bijection between the set of $\pm$-orientations of $\overline{V}$ and the set $\Field_p^\times/\mathbb{Z}^\times$. $\lvert \Field_p^\times/\mathbb{Z}^\times \rvert = \frac{p - 1}{2}$. The set $\Field_p^\times/\mathbb{Z}^\times$ is homotopy equivalent to a wedge of $(\frac{p - 3}{2})$-many $S^0$. We make the convention of setting $\St(0) = \St^\omega(0) = \mathbb{Z}$ and $\St(\overline{V}) = \mathbb{Z}$ if $\Dim \overline{V} = 1$. These terms appear in a tensor product of $\mathbb{Z}$ so they vanish. 

\begin{theorem}
\label{theorem:main:computational:restated}
Fix a prime $p \geq 3$. For $n \geq 1$, $\St_{2n}^{\omega, \pm}(\Field_p)$ is isomorphic to 
\begin{equation}\label{symp2}
\bigoplus_{\substack{0 = \overline{V}_{-1} \subsetneq \overline{V}_0 \subsetneq \dots \subsetneq \overline{V}_k \subsetneq \Field_p^{2n}, \\ \omega \vert_{\overline{V}_i} \equiv 0}} \St(\overline{V}_0) \otimes \St(\overline{V}_1/\overline{V}_0) \otimes \dots \otimes \St(\overline{V}_k/\overline{V}_{k - 1}) \otimes \St^\omega(\overline{V}_k^\perp/\overline{V}_k) \otimes \rh_0(\Field_p^\times/\mathbb{Z}^\times; \mathbb{Z})^{\oplus (k + 1)},
\end{equation}
as abelian groups.
\end{theorem}

\begin{proof}
Fix $2n \geq 2$, set $\mathcal{A}^\pm := \posety(\simp(\sbTD_{2n}(\Field_p)))$ and $\mathcal{B}:= \posety(\simp(\sbT_{2n}(\Field_p)))$. Since $\mathcal{A}^\pm \simeq \sbTD_{2n}(\Field_p)$ and $\mathcal{B} \simeq \sbT_{2n}(\Field_p)$, it suffices to apply Bj\"orner--Wachs--Welker's Theorem \ref{nonrecursiveFormula} to $\mathcal{A}^\pm $ and $\mathcal{B}$ to conclude the theorem. Set
\begin{align*}
\mathcal{F}& \colon \mathcal{A}^\pm  \rightarrow \mathcal{B}\\
\mathcal{F}& \colon[\overline{V}_0 \subsetneq \dots \subsetneq \overline{V}_k]^\pm \mapsto [\overline{V}_0 \subsetneq \dots \subsetneq \overline{V}_k]
\end{align*}
to be the projection map that forgets the $\pm$-orientations on the elements of the isotropic flag. Fix $ [\overline{V}_0 \subsetneq \dots \subsetneq \overline{V}_k]
 \in \mathcal{B}$ and note that $\Dim(\mathcal{F}_{< [\overline{V}_0 \subsetneq \dots \subsetneq \overline{V}_k]}) = k$. The upper link of $[\overline{V}_0 \subsetneq \dots \subsetneq \overline{V}_k]$, $\mathcal{B}_{> [\overline{V}_0 \subsetneq \dots \subsetneq \overline{V}_k]}$, is homotopy equivalent to the simplicial link of the $k$-simplex in $\sT_{2n}(\Field_p)$ that corresponds to the isotropic flag $[\overline{V}_0 \subsetneq \dots \subsetneq \overline{V}_k]$. Using Lemma \ref{SympPosetLinks} and Remark \ref{PosetToSimpLink}, we determine
\begin{equation*}
\mathcal{B}_{> [\overline{V}_0 \subsetneq \dots \subsetneq \overline{V}_k]} \simeq \cT(\overline{V}_0)* \cT(\overline{V}_1/\overline{V}_0)* \dots * \cT(\overline{V}_k/\overline{V}_{k - 1})* \sT(\overline{V}_k^\perp/\overline{V}_k).
\end{equation*}
$\mathcal{B}_{> [\overline{V}_0 \subsetneq \dots \subsetneq \overline{V}_k]}$ is $(n - k - 2)$-dimensional and $(n - k - 3)$-connected by Solomon--Tits' Theorem \ref{SBrB}. Passing to homology we have
\begin{equation}\label{proofB1}
\rh_{n - k - 2}(\mathcal{B}_{> [\overline{V}_0 \subsetneq \dots \subsetneq \overline{V}_k]}) \cong \St(\overline{V}_0) \otimes \St(\overline{V}_1/\overline{V}_0) \otimes \dots \otimes \St(\overline{V}_k/\overline{V}_{k - 1}) \otimes \St^\omega(\overline{V}_k^\perp/\overline{V}_k).
\end{equation}
In the proof of Proposition \ref{SympMPP21Lem3.15} we demonstrated that $\sTD_{2n}(\Field_p)$ is a complete join complex over $\sT_{2n}(\Field_p)$. So there exists a projection map $\varphi\colon \sTD_{2n}(\Field_p) \rightarrow \sT_{2n}(\Field_p)$ for which $\varphi$ projects the full sub-complex $\varphi(s_0)^{-1} * \dots * \varphi(s_k)^{-1}$ onto the simplex associated to the isotropic flag $[\overline{V}_0 \subsetneq \dots \subsetneq \overline{V}_k]$, where $s_i$ are the vertices associated to the isotropic flags $[V_i]$ for $i \in \{0, \dots, k\}$. Since $\overline{V}_i$ can be equipped with $(\frac{p - 1}{2})$-many distinct $\pm$-orientations, then there exists a bijection between $\varphi(s_i)^{-1}$ and $\Field_p^\times/\mathbb{Z}^\times$ as sets. We conclude that 
\begin{equation*}
\mathcal{F}_{\leq [\overline{V}_0 \subsetneq \dots \subsetneq \overline{V}_k]} \cong \posety(*_{i = 0}^k \varphi(s_i)^{-1}) \simeq *_{i = 0}^k (\Field_p^\times/\mathbb{Z}^\times) \simeq \bigvee_{(\frac{p - 3}{2})^{k + 1}} S^k,
\end{equation*}
which is $k$-dimensional and $(k - 1)$-connected. Passing to homology we have
\begin{equation}\label{proofB2}
\rh_k(\mathcal{F}_{\leq [\overline{V}_0 \subsetneq \dots \subsetneq \overline{V}_k]}) \cong \rh_0(\Field_p^\times/\mathbb{Z}^\times;\mathbb{Z})^{\oplus (k + 1)}. 
\end{equation}
We invoke Bj\"orner--Wachs--Welker's Theorem \ref{nonrecursiveFormula} to get
\begin{equation*}
\lvert \mathcal{A}^\pm \rvert \simeq \lvert \mathcal{B} \rvert \vee \left( \bigvee_{[\overline{V}_0 \subsetneq \dots \subsetneq \overline{V}_k] \in \mathcal{B}} \lvert \mathcal{F}_{\leq [\overline{V}_0 \subsetneq \dots \subsetneq \overline{V}_k]} \rvert * \lvert \mathcal{B}_{> [\overline{V}_0 \subsetneq \dots \subsetneq \overline{V}_k]} \rvert \right).
\end{equation*}
Using Equations \ref{proofB1} and \ref{proofB2}, when we pass to homology we have that $\St_{2n}^{\omega, \pm}(\Field_p)$ is isomorphic to the direct sum of $\St_{2n}^\omega(\Field_p)$ and
\begin{equation}\label{proofB3}
\bigoplus_{\substack{0 \subsetneq \overline{V}_0 \subsetneq \cdots \subsetneq \overline{V}_k \subsetneq \Field_p^{2n}, \\ \omega \vert_{\overline{V}_i} \equiv 0}} \St(\overline{V}_0) \otimes \St(\overline{V}_1/\overline{V}_0) \otimes \dots \otimes \St(\overline{V}_k/\overline{V}_{k - 1}) \otimes \St^\omega(\overline{V}_k^\perp/\overline{V}_k) \otimes \rh_0(\Field_p^\times/\mathbb{Z}^\times;\mathbb{Z})^{\oplus (k + 1)}
\end{equation}
which we can rewrite as 
\begin{equation}\label{proofB4}
 \bigoplus_{\substack{0 = \overline{V}_{-1} \subsetneq \overline{V}_0 \subsetneq \cdots \subsetneq \overline{V}_k \subsetneq \Field_p^{2n}, \\ \omega \vert_{\overline{V}_i} \equiv 0}} \St(\overline{V}_0) \otimes \St(\overline{V}_1/\overline{V}_0) \otimes \dots \otimes \St(\overline{V}_k/\overline{V}_{k - 1}) \otimes \St^\omega(\overline{V}_k^\perp/\overline{V}_k) \otimes \rh_0(\Field_p^\times/\mathbb{Z}^\times;\mathbb{Z})^{\oplus (k + 1)}.
\end{equation}
\vspace{-.5cm}
\end{proof}

The following corollary is derived from Equation \ref{symp2}. Recall that $\rank \St_n(\Field_p) = p^{n \choose 2}$ and $\rank \St_{2n}^\omega(\Field_p) = p^{n^2}$. 

\begin{corollary}\label{corollaryB1}
For a prime $p \geq 3$, let $t^\omega(n, p) = \rank \St_{2n}^{\omega, \pm}(\Field_p)$. Then for $n \geq 1$
\begin{equation}\label{tomegaequation2}
t^\omega(n, p) := \sum_{\substack{0 = \overline{V}_{(-1)} \subsetneq \overline{V}_0 \subsetneq \dots \subsetneq \overline{V}_k \subsetneq \Field_p^{2n},\\ m_i = \Dim \overline{V}_i, \\ \omega \vert_{\overline{V}_i} \equiv 0}} p^{m_0 \choose 2} \cdot p^{m_1 - m_0 \choose 2} \cdots p^{m_k - m_{k - 1} \choose 2} \cdot p^{(n - m_k)^2} \cdot \left(\frac{p - 3}{2}\right)^{k + 1}.
\end{equation} 
\end{corollary}

Let $\Gr_m(\Field_p^n)$ denote the \emph{$m$-th Grassmannian of $\Field_p^n$} and $\Gr^\omega_m(\Field_p^{2n})$ denote the \emph{$m$-th isotropic Grassmannian of $\Field_p^{2n}$}. An easy computation gives 
\begin{align}
  \lvert \Gr_{m}(\Field_p^{n}) \rvert & = \prod_{i = 0}^{m - 1} \frac{ p^{n - i} - 1}{ p^{m - i} - 1} \\ 
  \lvert \Gr^\omega_{m}(\Field_p^{2n}) \rvert & = \prod_{i = 0}^{m - 1} \frac{p^{2n - 2i} - 1}{p^{m - i} - 1}.
\end{align}
\begin{remark}
Given a sequence of integers $m_0 < \dots < m_k$, the product 
\begin{equation} \label{isoSuperGrass}
\left( \prod_{i = 0}^{k - 1} \lvert \Gr_{m_{i}}(\mathbb{F}_p^{m_{i + 1}}) \rvert \right) \lvert \Gr^\omega_{m_{k}}(\Field_p^{2n}) \rvert = \left(\prod_{i = 0}^{k - 1} \left( \prod_{j = 0}^{m_i - 1} \frac{ p^{m_{i + 1} - j} - 1}{ p^{m_i - j} - 1} \right) \right) \left(\prod_{i = 0}^{m_k - 1} \frac{p^{2n - 2i} - 1}{p^{m_k - i} - 1} \right),
\end{equation}
counts the number of isotropic flags 
\begin{equation*}
0 \subsetneq \overline{V}_0 \subsetneq \dots \subsetneq \overline{V}_k \subsetneq \Field_p^{2n}
\end{equation*}
that satisfy $\Dim \overline{V}_i = m_i$ for $i \in \{0, \dots, k\}$. 
\end{remark}

\begin{proof}[Proof of Theorem \ref{theorem:main:computational} \& Theorem \ref{theoremBrestrated}]
Proposition \ref{aSimplicialModel} implies that
\begin{equation}
\rh_{n - 1}(\sT_{2n}(\mathbb{Q})/\Gamma_{2n}^\omega(p);\mathbb{Z}) \cong \rh_{n - 1}(\sTD_{2n}(\Field_p);\mathbb{Z}) = \St_{2n}^{\omega, \pm}(\Field_p). 
\end{equation}
By Corollary \ref{corollaryB1}, $t^\omega(n, p)$ of Equation \ref{tomegaequation2} is the rank of $\St_{2n}^{\omega, \pm}(\Field_p)$. Hence, $t^\omega(n, p)$ is the rank of $\rh_{n - 1}(\sT_{2n}(\mathbb{Q})/\Gamma_{2n}^\omega(p);\mathbb{Z})$. All that remains to establish is that $t^\omega(n, p)$ of Equation \ref{mainequation} from Theorem \ref{theorem:main:computational} (Equation \ref{mainequation2} from Theorem \ref{theoremBrestrated}) is equivalent to $t^\omega(n, p)$ of Equation \ref{tomegaequation2} from Corollary \ref{corollaryB1}. But this follows from noticing that $t^\omega(n, p)$ of Equation \ref{tomegaequation2} depends on unique isotropic flags while $t^\omega(n, p)$ of Equation \ref{mainequation} depends on increasing sequences $m_0 < \dots < m_k$ that correspond to isotropic flags whose subspaces have dimension $\{m_0, \dots, m_k\}$. We account for this difference by including Equation \ref{isoSuperGrass} to each term of $t^\omega(n, p)$ of Equation \ref{mainequation}. 
\end{proof}

\begin{remark}
When $p = 3$, Equation \ref{tomegaequation2} simplifies to $\St_{2n}^\omega(\Field_3)$ since $\frac{p - 3}{2} = 0$. 
\end{remark}

A similar statement holds for the $\pm$-oriented Tits building and can be proven using a similar argument as for Theorem \ref{theorem:main:computational}. 

\begin{theorem} \textup{\cite{PeterCommunication}}\label{patzt}
Fix a prime $p \geq 3$. Then for $n \geq 2$, $\St_n^\pm(\Field_p)$ is isomorphic to 
\begin{equation}\label{normal2}
\bigoplus_{0 = \overline{V}_{-1} \subsetneq \overline{V}_0 \subsetneq \dots \subsetneq \overline{V}_k \subsetneq \Field_p^{n}} \St(\overline{V}_0) \otimes \St(\overline{V}_1/\overline{V}_0) \otimes \dots \otimes \St(\overline{V}_k/\overline{V}_{k - 1}) \otimes \St(\Field_p^n/\overline{V}_k) \otimes \rh_0(\Field_p^\times/\mathbb{Z}^\times; \mathbb{Z})^{\oplus (k + 1)},
\end{equation}
as abelian groups.
\end{theorem}

The following corollary is derived from Equation \ref{normal2}.

\begin{corollary} \textup{\cite{PeterCommunication}}\label{compformula1}
For a prime $p \geq 3$, let $t(n, p) = \rank \St_n^\pm(\Field_p)$. Then for $n \geq 2$,
\begin{equation*} \label{normalpmsteinberg}
t(n, p) := \sum_{\substack{0 = \overline{V}_{-1} \subsetneq \overline{V}_0 \subsetneq \dots \subsetneq \overline{V}_k \subsetneq \Field_p^{n},\\ m_i = \Dim \overline{V}_i}} p^{m_0 \choose 2} \cdot p^{m_1 - m_0 \choose 2} \cdots p^{m_k - m_{k - 1} \choose 2} \cdot p^{n - m_k \choose 2} \cdot \left(\frac{p - 3}{2}\right)^{k + 1}.
\end{equation*}
\end{corollary}

\begin{remark}
Theorem \ref{theorem:main:computational:restated} establishes $\St_{2n}^{\omega, \pm}(\Field_p)$ is isomorphic to the group described in Equation \ref{symp2}, call it $G$. There is a natural way of turning both $\St_{2n}^{\omega, \pm}(\Field_p)$ and $G$ into $\GL_{2n}^\pm (\Field_p)$-representations. The proof of Theorem \ref{theorem:main:computational:restated} does not prove that $\St_{2n}^{\omega, \pm}(\Field_p)$ and $G$ are isomorphic as representations. But the proof does show that they have filtrations with filtration quotients isomorphic as representations. \\
Likewise, $\St_{n}^{\pm}(\Field_p)$ and the group described in Equation \ref{normal2} might not be isomorphic as $\GL_n^\pm(\Field_p)$-representations, but a similar proof does show that they have filtrations with filtration quotients isomorphic as representations.
\end{remark}

We restate and prove Corollary \ref{corollary:main:computational}.

\begin{corollary}\label{corollaryB2}
Fix a prime $p \geq 3$ and let $t^\omega(n, p):= \rank (\rh_{n - 1}(\sT_{2n}(\mathbb{Q})/\Gamma_{2n}^\omega(p); \mathbb{Z}))$. For $n \geq 1$,
\begin{equation}
t^\omega(n, p) \geq p^{n^2} \cdot \left( \frac{p - 1}{2} \right)^n. 
\end{equation}
\end{corollary}

\begin{proof}[Proof of Corollary \ref{corollary:main:computational} \& Corollary \ref{corollaryB2}]
Note that 
\begin{equation}\label{simplesum}
p^{n^2} \cdot \left(\frac{p - 1}{2} \right)^n = \sum_{k = 0}^n {n \choose k} ~ p^{n^2} \cdot \left(\frac{p - 3}{2} \right)^k.
\end{equation} 
For $k \in \{0, \dots, n\}$ set 
\begin{equation}
S_1(k):= \sum_{\substack{0 \subsetneq \overline{V}_0 \subsetneq \dots \subsetneq \overline{V}_{k - 1} \subsetneq \Field_p^{2n},\\ m_i = \Dim \overline{V}_i, \\ \omega \vert_{\overline{V}_i} \equiv 0}} p^{m_0 \choose 2} \cdot p^{m_1 - m_0 \choose 2} \cdots p^{m_{k - 1} - m_{k - 2} \choose 2} \cdot p^{(n - m_{k - 1})^2} \cdot \left(\frac{p - 3}{2} \right)^k,
\end{equation}
and for $\{r_0, \dots, r_{k - 1}\} \subset \{1, \dots, n\}$ such that $r_i < r_{i + 1}$ set
\begin{equation} \label{Shortsum}
S_2(r_0, \dots, r_{k - 1})  := \sum_{\substack{r_i = \Dim \overline{V}_i, \\0 \subsetneq \overline{V}_0 \subsetneq \dots \subsetneq \overline{V}_{k - 1} \subsetneq \Field_p^{2n}, \\ \omega \vert_{\overline{V}_i} \equiv 0}} p^{r_0 \choose 2} \cdot p^{r_1 - r_0 \choose 2} \cdots p^{r_{k - 1} - r_{k - 2} \choose 2} \cdot p^{(n - r_{k - 1})^2} \cdot \left(\frac{p - 3}{2} \right)^k.
\end{equation}
$S_1(k)$ is the sum of the terms in $t^\omega(n, p)$ that correspond to isotropic flags of length $k$. It follows that
\begin{equation*}
t^\omega(n, p) = \sum_{k = 0}^n S_1(k). 
\end{equation*}
$S_2(r_0, \dots, r_{k - 1})$ is the sum of the terms in $S_1(k)$ that correspond to isotropic flags whose subspaces have dimensions $r_0 < \dots < r_{k - 1}$. It follows that
\begin{equation}\label{s1s2}
S_1(k) = \sum_{1 \leq r_0 < \dots < r_{k - 1} \leq n} S_2(r_0, \dots, r_{k - 1}). 
\end{equation}
To complete this proof we check that for each $k \in \{0, \dots, n\}$, 
\begin{equation}\label{completeproof}
S_1(k) \geq {n \choose k} ~ p^{n^2} \cdot \left(\frac{p - 3}{2} \right)^k.
\end{equation}
When $k = 0$, 
\begin{equation*}
S_1(0) = \rank \St_{2n}^\omega(\Field_p) = p^{n^2} = {n \choose 0} ~ p^{n^2} \cdot \left(\frac{p - 3}{2} \right)^0.
\end{equation*}
Fix $k \in \{1, \dots, n\}$. We estimate $S_2(r_0, \dots, r_{k - 1})$ for an arbitrary subset $\{r_0, \dots, r_{k - 1}\} \subset \{1, \dots, n\}$. The product 
\begin{equation} \label{isoSuperGrass2}
\left( \prod_{i = 1}^{k - 1} \lvert \Gr_{r_{i - 1}}(\mathbb{F}_p^{r_i}) \rvert \right) \lvert \Gr^\omega_{r_{k - 1}}(\Field_p^{2n}) \rvert
\end{equation}
corresponds to the number of terms in $S_2(r_0, \dots, r_{k - 1})$. That is, the number of isotropic flags whose subspaces have dimensions $r_0 < \dots < r_{k - 1}$. An easy computation gives 
\begin{align}
  \lvert \Gr_{m}(\Field_p^{n}) \rvert & \geq p^{m(n - m)}, \label{isoSuperGrass2a} \\ 
  \lvert \Gr^\omega_{m}(\Field_p^{2n}) \rvert & \geq p^{m(2n - m) - {m \choose 2}}. \label{isoSuperGrass3}
\end{align}
An equivalent formulation of $S_2(r_0, \dots, r_{k - 1})$ is 
\begin{equation}\label{completeS2}
p^{r_0 \choose 2} \cdot \left( \prod_{i = 1}^{k - 1} \lvert \Gr_{r_{i - 1}}(\mathbb{F}_p^{r_i}) \rvert \cdot p^{r_i - r_{i - 1} \choose 2} \right) \lvert \Gr^\omega_{r_{k - 1}}(\Field_p^{2n}) \rvert \cdot p^{(n - r_{k - 1})^2} \cdot \left( \frac{p - 3}{2} \right)^k. 
\end{equation}
For $i \in \{1, \dots, k - 1\}$
\begin{equation*}
\lvert \Gr_{r_{i - 1}}(\mathbb{F}_p^{r_i}) \rvert \cdot p^{r_i - r_{i - 1} \choose 2} \geq p^{r_{i - 1}(r_i - r_{i - 1})} \cdot p^{r_i - r_{i - 1} \choose 2} = p^{{r_i \choose 2} - {r_{i - 1} \choose 2}},
\end{equation*}
so
\begin{equation}\label{S2c}
\left( \prod_{i = 1}^{k - 1} \lvert \Gr_{r_{i - 1}}(\mathbb{F}_p^{r_i}) \rvert \cdot p^{r_i - r_{i - 1} \choose 2} \right) \geq p^{\sum_{i = 1}^{k - 1} {r_i \choose 2} - {r_{i - 1} \choose 2} } = p^{{r_{k - 1} \choose 2} - {r_0 \choose 2}}.
\end{equation}
And 
\begin{equation} \label{S2b}
 \lvert \Gr^\omega_{r_{k - 1}}(\Field_p^{2n}) \rvert p^{(n - r_{k - 1})^2} \geq p^{r_{k - 1}(2n - r_{k - 1}) - {r_{k - 1} \choose 2}} \cdot p^{(n - r_{k - 1})^2} = p^{n^2 - {r_{k - 1} \choose 2}}.
\end{equation}
Thanks to our approximations in Equations \ref{S2c} and \ref{S2b}, we get 
\begin{equation*}
S_2(r_0, \dots, r_{k - 1}) \geq p^{{r_0 \choose 2} + {r_{k - 1} \choose 2} - {r_0 \choose 2} + n^2 - {r_{k - 1} \choose 2}} \cdot \left(\frac{p - 3}{2} \right)^k = p^{n^2} \cdot \left(\frac{p - 3}{2} \right)^k. 
\end{equation*}
Since $\{r_0, \dots, r_{k - 1}\} \subset \{1, \dots, n\}$, then there are ${n \choose k}$-many subsets $\{r_0, \dots, r_{k - 1}\}$ and so 
\begin{equation*}
S_1(k) = \sum_{1 \leq r_0 < \dots < r_{k - 1} \leq n} S_2(r_0, \dots, r_{k - 1}) \geq \sum_{1 \leq r_0 < \dots < r_{k - 1} \leq n} p^{n^2} \cdot \left(\frac{p - 3}{2} \right)^k = {n \choose k} ~ p^{n^2} \cdot \left(\frac{p - 3}{2} \right)^k,
\end{equation*}
We conclude that 
\begin{equation}
t^\omega(n, p) = \sum_{k = 0}^n S_1(k) \geq \sum_{k = 0}^n {n \choose k} ~ p^{n^2} \cdot \left(\frac{p - 3}{2}\right)^k = p^{n^2} \cdot \left( \frac{p - 1}{2} \right)^n. 
\end{equation}
\end{proof}

\printbibliography

@article {BjornerWachsWelker,
    AUTHOR = {Bj\"orner, Anders and Wachs, Michelle L. and Welker, Volkmar},
     TITLE = {Poset fiber theorems},
   JOURNAL = {Trans. Amer. Math. Soc.},
  FJOURNAL = {Transactions of the American Mathematical Society},
    VOLUME = {357},
      YEAR = {2005},
    NUMBER = {5},
     PAGES = {1877--1899},
      ISSN = {0002-9947,1088-6850},
   MRCLASS = {05E25 (06A11 55P10)},
  MRNUMBER = {2115080},
MRREVIEWER = {Edward\ B.\ Swartz},
       DOI = {10.1090/S0002-9947-04-03496-8},
       URL = {https://doi.org/10.1090/S0002-9947-04-03496-8},
}

@article {BorelSerreCorners,
    AUTHOR = {Borel, A. and Serre, J.-P.},
     TITLE = {Corners and arithmetic groups},
   JOURNAL = {Comment. Math. Helv.},
  FJOURNAL = {Commentarii Mathematici Helvetici},
    VOLUME = {48},
      YEAR = {1973},
     PAGES = {436--491},
      ISSN = {0010-2571,1420-8946},
   MRCLASS = {22E40},
  MRNUMBER = {387495},
MRREVIEWER = {M.\ S.\ Raghunathan},
       DOI = {10.1007/BF02566134},
       URL = {https://doi.org/10.1007/BF02566134},
}

@book {BrownBuildings,
    AUTHOR = {Brown, Kenneth S.},
     TITLE = {Buildings},
    SERIES = {Springer Monographs in Mathematics},
      NOTE = {Reprint of the 1989 original},
 PUBLISHER = {Springer-Verlag, New York},
      YEAR = {1998},
     PAGES = {viii+215},
      ISBN = {0-387-98624-3},
   MRCLASS = {20E42},
  MRNUMBER = {1644630},
}

@article {ChurchPutnam17,
    AUTHOR = {Church, Thomas and Putman, Andrew},
     TITLE = {The codimension-one cohomology of {${\rm SL}_n\mathbb{Z}$}},
   JOURNAL = {Geom. Topol.},
  FJOURNAL = {Geometry \& Topology},
    VOLUME = {21},
      YEAR = {2017},
    NUMBER = {2},
     PAGES = {999--1032},
      ISSN = {1465-3060,1364-0380},
   MRCLASS = {11F75 (20E42 20G30 20J06 57Q05)},
  MRNUMBER = {3626596},
MRREVIEWER = {Steffen\ Kionke},
       DOI = {10.2140/gt.2017.21.999},
       URL = {https://doi.org/10.2140/gt.2017.21.999},
}

@article {GalatiusRandal-Williams14,
    AUTHOR = {Galatius, S\o ren and Randal-Williams, Oscar},
     TITLE = {Homological stability for moduli spaces of high dimensional
              manifolds. {I}},
   JOURNAL = {J. Amer. Math. Soc.},
  FJOURNAL = {Journal of the American Mathematical Society},
    VOLUME = {31},
      YEAR = {2018},
    NUMBER = {1},
     PAGES = {215--264},
      ISSN = {0894-0347,1088-6834},
   MRCLASS = {57R90 (55P47 57R15 57R56)},
  MRNUMBER = {3718454},
MRREVIEWER = {Sam\ Nariman},
       DOI = {10.1090/jams/884},
       URL = {https://doi.org/10.1090/jams/884},
}

@article {HatcherWahl,
    AUTHOR = {Hatcher, Allen and Wahl, Nathalie},
     TITLE = {Stabilization for mapping class groups of 3-manifolds},
   JOURNAL = {Duke Math. J.},
  FJOURNAL = {Duke Mathematical Journal},
    VOLUME = {155},
      YEAR = {2010},
    NUMBER = {2},
     PAGES = {205--269},
      ISSN = {0012-7094,1547-7398},
   MRCLASS = {57M07 (20F28)},
  MRNUMBER = {2736166},
MRREVIEWER = {Mihalis\ A.\ Sykiotis},
       DOI = {10.1215/00127094-2010-055},
       URL = {https://doi.org/10.1215/00127094-2010-055},
}

@article {MiPaP21,
    AUTHOR = {Miller, Jeremy and Patzt, Peter and Putman, Andrew},
     TITLE = {On the top-dimensional cohomology groups of congruence
              subgroups of {${\rm SL}(n, \mathbb{Z})$}},
   JOURNAL = {Geom. Topol.},
  FJOURNAL = {Geometry \& Topology},
    VOLUME = {25},
      YEAR = {2021},
    NUMBER = {2},
     PAGES = {999--1058},
      ISSN = {1465-3060,1364-0380},
   MRCLASS = {11F75},
  MRNUMBER = {4251441},
MRREVIEWER = {Alexander\ D.\ Rahm},
       DOI = {10.2140/gt.2021.25.999},
       URL = {https://doi.org/10.2140/gt.2021.25.999},
}

@article {MiVdK,
    AUTHOR = {Mirzaii, B. and van der Kallen, W.},
     TITLE = {Homology stability for unitary groups},
   JOURNAL = {Doc. Math.},
  FJOURNAL = {Documenta Mathematica},
    VOLUME = {7},
      YEAR = {2002},
     PAGES = {143--166},
      ISSN = {1431-0635,1431-0643},
   MRCLASS = {19G99 (11E70 20J99)},
  MRNUMBER = {1911214},
MRREVIEWER = {Stanis\l aw\ Betley},
}

@article {Q78,
    AUTHOR = {Quillen, Daniel},
     TITLE = {Homotopy properties of the poset of nontrivial {$p$}-subgroups
              of a group},
   JOURNAL = {Adv. in Math.},
  FJOURNAL = {Advances in Mathematics},
    VOLUME = {28},
      YEAR = {1978},
    NUMBER = {2},
     PAGES = {101--128},
      ISSN = {0001-8708},
   MRCLASS = {20J99},
  MRNUMBER = {493916},
MRREVIEWER = {Kenneth\ S.\ Brown},
       DOI = {10.1016/0001-8708(78)90058-0},
       URL = {https://doi.org/10.1016/0001-8708(78)90058-0},
}

@incollection {Solomon,
    AUTHOR = {Solomon, Louis},
     TITLE = {The {S}teinberg character of a finite group with {$BN$}-pair},
 BOOKTITLE = {Theory of {F}inite {G}roups ({S}ymposium, {H}arvard {U}niv.,
              {C}ambridge, {M}ass., 1968)},
     PAGES = {213--221},
 PUBLISHER = {W. A. Benjamin, Inc., New York-Amsterdam},
      YEAR = {1969},
   MRCLASS = {20.25},
  MRNUMBER = {246951},
MRREVIEWER = {T.\ Ono},
}

@article {VanDerKallenLooijenga,
    AUTHOR = {van der Kallen, Wilberd and Looijenga, Eduard},
     TITLE = {Spherical complexes attached to symplectic lattices},
   JOURNAL = {Geom. Dedicata},
  FJOURNAL = {Geometriae Dedicata},
    VOLUME = {152},
      YEAR = {2011},
     PAGES = {197--211},
      ISSN = {0046-5755,1572-9168},
   MRCLASS = {11E57 (05E18 11H06 19B14)},
  MRNUMBER = {2795243},
MRREVIEWER = {B.\ Sury},
       DOI = {10.1007/s10711-010-9553-0},
       URL = {https://doi.org/10.1007/s10711-010-9553-0},
}

@article {BruckSroka24,
    AUTHOR = {Br\"uck, Benjamin and Sroka, Robin J.},
     TITLE = {Apartment classes of integral symplectic groups},
   JOURNAL = {J. Topol. Anal.},
  FJOURNAL = {Journal of Topology and Analysis},
    VOLUME = {17},
      YEAR = {2025},
    NUMBER = {6},
     PAGES = {1821--1840},
      ISSN = {1793-5253,1793-7167},
   MRCLASS = {20E42 (20G20)},
  MRNUMBER = {4926622},
       DOI = {10.1142/S1793525324500286},
       URL = {https://doi.org/10.1142/S1793525324500286},
}

@book {MH73,
    AUTHOR = {Milnor, John and Husemoller, Dale},
     TITLE = {Symmetric bilinear forms},
    SERIES = {Ergebnisse der Mathematik und ihrer Grenzgebiete [Results in
              Mathematics and Related Areas]},
    VOLUME = {Band 73},
 PUBLISHER = {Springer-Verlag, New York-Heidelberg},
      YEAR = {1973},
     PAGES = {viii+147},
   MRCLASS = {15A63 (10C05 57D65)},
  MRNUMBER = {506372},
MRREVIEWER = {Louis\ H.\ Kauffman},
}

@book {N72,
    AUTHOR = {Newman, Morris},
     TITLE = {Integral matrices},
    SERIES = {Pure and Applied Mathematics},
    VOLUME = {Vol. 45},
 PUBLISHER = {Academic Press, New York-London},
      YEAR = {1972},
     PAGES = {xvii+224},
   MRCLASS = {15A33},
  MRNUMBER = {340283},
MRREVIEWER = {B.\ M.\ Stewart},
}

@article {Zee64,
    AUTHOR = {Zeeman, E. C.},
     TITLE = {Relative simplicial approximation},
   JOURNAL = {Proc. Cambridge Philos. Soc.},
  FJOURNAL = {Proceedings of the Cambridge Philosophical Society},
    VOLUME = {60},
      YEAR = {1964},
     PAGES = {39--43},
      ISSN = {0008-1981},
   MRCLASS = {55.25},
  MRNUMBER = {158403},
MRREVIEWER = {M.\ L.\ Curtis},
       DOI = {10.1017/s0305004100037415},
       URL = {https://doi.org/10.1017/s0305004100037415},
}

@article {Maaz,
    AUTHOR = {Maazen, Hendrik},
     TITLE = {Stabilit\'e{} de l'homologie de {${\rm GL}\sb{n}$}},
   JOURNAL = {C. R. Acad. Sci. Paris S\'er. A-B},
  FJOURNAL = {Comptes Rendus Hebdomadaires des S\'eances de l'Acad\'emie des
              Sciences. S\'eries A et B},
    VOLUME = {288},
      YEAR = {1979},
    NUMBER = {15},
     PAGES = {707--708},
      ISSN = {0151-0509},
   MRCLASS = {20J05},
  MRNUMBER = {532394},
MRREVIEWER = {S.\ Moran},
}

@book {Bredon72,
    AUTHOR = {Bredon, Glen E.},
     TITLE = {Introduction to compact transformation groups},
    SERIES = {Pure and Applied Mathematics},
    VOLUME = {Vol. 46},
 PUBLISHER = {Academic Press, New York-London},
      YEAR = {1972},
     PAGES = {xiii+459},
   MRCLASS = {57E15},
  MRNUMBER = {413144},
}

@misc{BruckPatztSroka2023Arxiv,
      title={A presentation of symplectic Steinberg modules and cohomology of $\operatorname{Sp}_{2n}(\mathbb{Z})$}, 
      author={Benjamin Brück and Peter Patzt and Robin J. Sroka},
      year={2023},
      eprint={2306.03180},
      archivePrefix={arXiv},
      primaryClass={math.AT},
      url={https://arxiv.org/abs/2306.03180}, 
}

@inproceedings {Charney87,
    AUTHOR = {Charney, Ruth},
     TITLE = {A generalization of a theorem of {V}ogtmann},
 BOOKTITLE = {Proceedings of the {N}orthwestern conference on cohomology of
              groups ({E}vanston, {I}ll., 1985)},
   JOURNAL = {J. Pure Appl. Algebra},
  FJOURNAL = {Journal of Pure and Applied Algebra},
    VOLUME = {44},
      YEAR = {1987},
    NUMBER = {1-3},
     PAGES = {107--125},
      ISSN = {0022-4049,1873-1376},
   MRCLASS = {18F25 (11E70 19D55 19G99 20J05)},
  MRNUMBER = {885099},
MRREVIEWER = {Ross\ Staffeldt},
       DOI = {10.1016/0022-4049(87)90019-3},
       URL = {https://doi.org/10.1016/0022-4049(87)90019-3},
}

@article {Par97,
    AUTHOR = {Paraschivescu, Andrei},
     TITLE = {On a generalization of the double coset formula},
   JOURNAL = {Duke Math. J.},
  FJOURNAL = {Duke Mathematical Journal},
    VOLUME = {89},
      YEAR = {1997},
    NUMBER = {1},
     PAGES = {1--8},
      ISSN = {0012-7094,1547-7398},
   MRCLASS = {20J05},
  MRNUMBER = {1458968},
MRREVIEWER = {David\ Benson},
       DOI = {10.1215/S0012-7094-97-08901-8},
       URL = {https://doi.org/10.1215/S0012-7094-97-08901-8},
}

@book {Knus91,
    AUTHOR = {Knus, Max-Albert},
     TITLE = {Quadratic and {H}ermitian forms over rings},
    SERIES = {Grundlehren der mathematischen Wissenschaften [Fundamental
              Principles of Mathematical Sciences]},
    VOLUME = {294},
      NOTE = {With a foreword by I. Bertuccioni},
 PUBLISHER = {Springer-Verlag, Berlin},
      YEAR = {1991},
     PAGES = {xii+524},
      ISBN = {3-540-52117-8},
   MRCLASS = {11Exx (11E39 11E81 16E20 19Gxx)},
  MRNUMBER = {1096299},
MRREVIEWER = {Rudolf\ Scharlau},
       DOI = {10.1007/978-3-642-75401-2},
       URL = {https://doi.org/10.1007/978-3-642-75401-2},
}

@article {BieriEckmann73,
    AUTHOR = {Bieri, Robert and Eckmann, Beno},
     TITLE = {Groups with homological duality generalizing {P}oincar\'e{}
              duality},
   JOURNAL = {Invent. Math.},
  FJOURNAL = {Inventiones Mathematicae},
    VOLUME = {20},
      YEAR = {1973},
     PAGES = {103--124},
      ISSN = {0020-9910,1432-1297},
   MRCLASS = {20J05},
  MRNUMBER = {340449},
MRREVIEWER = {L.\ Ribes},
       DOI = {10.1007/BF01404060},
       URL = {https://doi.org/10.1007/BF01404060},
}

@article {MR1749441,
    AUTHOR = {Gunnells, Paul E.},
     TITLE = {Symplectic modular symbols},
   JOURNAL = {Duke Math. J.},
  FJOURNAL = {Duke Mathematical Journal},
    VOLUME = {102},
      YEAR = {2000},
    NUMBER = {2},
     PAGES = {329--350},
      ISSN = {0012-7094,1547-7398},
   MRCLASS = {11F75 (11F80)},
  MRNUMBER = {1749441},
MRREVIEWER = {Ian\ Kiming},
       DOI = {10.1215/S0012-7094-00-10226-8},
       URL = {https://doi.org/10.1215/S0012-7094-00-10226-8},
}

@misc{68a8695d93b544d594c5cde50fd83301,
title = "Patterns in the homology of algebras: Vanishing, stability, and higher structures",
author = "Sroka, {Robin Janik}",
year = "2021",
language = "English",
isbn = "978-87-7125-043-5",
publisher = "Department of Mathematical Sciences, Faculty of Science, University of Copenhagen",
}

@incollection {MR3290086,
    AUTHOR = {Church, Thomas and Farb, Benson and Putman, Andrew},
     TITLE = {A stability conjecture for the unstable cohomology of {${\rm
              SL}_n(\mathbb{Z})$}, mapping class groups, and {${\rm Aut}(F_n)$}},
 BOOKTITLE = {Algebraic topology: applications and new directions},
    SERIES = {Contemp. Math.},
    VOLUME = {620},
     PAGES = {55--70},
 PUBLISHER = {Amer. Math. Soc., Providence, RI},
      YEAR = {2014},
      ISBN = {978-0-8218-9474-3},
   MRCLASS = {11F75 (20F28 20G10)},
  MRNUMBER = {3290086},
MRREVIEWER = {Jean\ Raimbault},
       DOI = {10.1090/conm/620/12366},
       URL = {https://doi.org/10.1090/conm/620/12366},
}

@misc {PeterCommunication,
    AUTHOR = {Patzt, Peter},
    howpublished="Private communication"
}

@article {MR4806366,
    AUTHOR = {Br\"uck, Benjamin and Santos Rego, Yuri and Sroka, Robin J.},
     TITLE = {On the top-dimensional cohomology of arithmetic {C}hevalley
              groups},
   JOURNAL = {Proc. Amer. Math. Soc.},
  FJOURNAL = {Proceedings of the American Mathematical Society},
    VOLUME = {152},
      YEAR = {2024},
    NUMBER = {10},
     PAGES = {4131--4139},
      ISSN = {0002-9939,1088-6826},
   MRCLASS = {20E42 (11F75 57M07)},
  MRNUMBER = {4806366},
MRREVIEWER = {Martino\ Garonzi},
       DOI = {10.1090/proc/16948},
       URL = {https://doi.org/10.1090/proc/16948},
}

@article {MR4011804,
    AUTHOR = {Church, Thomas and Farb, Benson and Putman, Andrew},
     TITLE = {Integrality in the {S}teinberg module and the top-dimensional
              cohomology of {$\SL_n \mathcal{O}_K$}},
   JOURNAL = {Amer. J. Math.},
  FJOURNAL = {American Journal of Mathematics},
    VOLUME = {141},
      YEAR = {2019},
    NUMBER = {5},
     PAGES = {1375--1419},
      ISSN = {0002-9327,1080-6377},
   MRCLASS = {20E42 (20G10 51E24)},
  MRNUMBER = {4011804},
MRREVIEWER = {Matthias\ Wendt},
       DOI = {10.1353/ajm.2019.0036},
       URL = {https://doi.org/10.1353/ajm.2019.0036},
}

@article {MR4917220,
    AUTHOR = {Br\"uck, Benjamin and Himes, Zachary},
     TITLE = {Top-degree rational cohomology in the symplectic group of a
              number ring},
   JOURNAL = {Selecta Math. (N.S.)},
  FJOURNAL = {Selecta Mathematica. New Series},
    VOLUME = {31},
      YEAR = {2025},
    NUMBER = {3},
     PAGES = {Paper No. 55, 23},
      ISSN = {1022-1824,1420-9020},
   MRCLASS = {11F75 (20E42 55U10)},
  MRNUMBER = {4917220},
MRREVIEWER = {B.\ Sury},
       DOI = {10.1007/s00029-025-01051-8},
       URL = {https://doi.org/10.1007/s00029-025-01051-8},
}

\end{document}